\documentclass[reqno]{amsart}

\usepackage{amssymb}
\usepackage{hyperref}
\usepackage{mathtools}
\usepackage[all]{xy}
\usepackage{verbatim}
\usepackage{ifthen}
\usepackage{xargs}
\usepackage{bussproofs}
\usepackage{turnstile}
\usepackage{etex}

\hypersetup{colorlinks=true,linkcolor=blue}

\newcommand{\axref}[1]{(\hyperref[ax:#1]{#1})}

\newcommand{\newref}[4][]{
\ifthenelse{\equal{#1}{}}{\newtheorem{h#2}[hthm]{#4}}{\newtheorem{h#2}{#4}[#1]}
\expandafter\newcommand\csname r#2\endcsname[1]{#3~\ref{#2:##1}}
\expandafter\newcommand\csname R#2\endcsname[1]{#4~\ref{#2:##1}}
\expandafter\newcommand\csname n#2\endcsname[1]{\ref{#2:##1}}
\newenvironmentx{#2}[2][1=,2=]{
\ifthenelse{\equal{##2}{}}{\begin{h#2}}{\begin{h#2}[##2]}
\ifthenelse{\equal{##1}{}}{}{\label{#2:##1}}
}{\end{h#2}}
}

\newref[section]{thm}{Theorem}{Theorem}
\newref{lem}{Lemma}{Lemma}
\newref{prop}{Proposition}{Proposition}
\newref{cor}{Corollary}{Corollary}
\newref{cond}{Condition}{Condition}

\theoremstyle{definition}
\newref{defn}{Definition}{Definition}
\newref{example}{Example}{Example}

\theoremstyle{remark}
\newref{remark}{Remark}{Remark}

\newcommand{\type}{}
\newcommand{\ob}{}
\newcommand{\term}{1}
\newcommand{\unit}{()}

\newcommand{\fs}[1]{\mathrm{#1}}

\newcommand{\Hom}{\fs{Hom}}
\newcommand{\Id}{\fs{Id}}
\newcommand{\refl}{\fs{refl}}
\newcommand{\sym}[1]{#1^{-1}}
\newcommand{\id}{\fs{id}}
\newcommand{\pmap}{\fs{ap}}
\newcommand{\Fib}{\fs{Fib}}
\newcommand{\fib}{\ \fs{fib}}
\newcommand{\El}{\fs{El}}

\numberwithin{figure}{section}

\newcommand{\ct}{%
  \mathchoice{\mathbin{\raisebox{0.25ex}{$\displaystyle\centerdot$}}}%
             {\mathbin{\raisebox{0.25ex}{$\centerdot$}}}%
             {\mathbin{\raisebox{0.25ex}{$\scriptstyle\,\centerdot\,$}}}%
             {\mathbin{\raisebox{0.25ex}{$\scriptscriptstyle\,\centerdot\,$}}}
}

\newcommand{\pb}[1][dr]{\save*!/#1-1.2pc/#1:(-1,1)@^{|-}\restore}
\newcommand{\po}[1][dr]{\save*!/#1+1.2pc/#1:(1,-1)@^{|-}\restore}

\begin{document}

\title{Indexed type theories}

\author{Valery Isaev}

\begin{abstract}
In this paper, we define indexed type theories which are related to indexed ($\infty$-)categories in the same way as (homotopy) type theories are related to ($\infty$-)categories.
We define several standard constructions for such theories including finite (co)limits, arbitrary (co)products, exponents, object classifiers, and orthogonal factorization systems.
We also prove that these constructions are equivalent to their type theoretic counterparts such as $\Sigma$-types, unit types, identity types, finite higher inductive types, $\Pi$-types, univalent universes, and higher modalities.
\end{abstract}

\maketitle

\section{Introduction}

Indexed categories were defined in \cite{indexed-cats} (see also \cite[B1]{elephant}).
We define an analogue of this notion using the language of type theory.
Ordinary homotopy type theory is an internal language of $\infty$-categories with some additional structure depending on constructions that we assume in the theory.
Often we need to assume that the $\infty$-category is at least locally Cartesian closed.
Indexed type theories allows us to discuss properties of arbitrary (indexed) $\infty$-categories.

Indexed type theories can be useful even when applied to $\infty$-categories which have all the required structure such as $\infty$-toposes.
One problem of ordinary homotopy type theory is that every construction must be stable under pullbacks.
For example, we will define orthogonal factorization systems in section~\ref{sec:refl-fib}.
In ordinary homotopy type theory only the stable factorization systems can be defined (which was done in \cite{modality-hott}).

Similar problem occurs when we try to describe certain universes.
For example, we could try to add a universe $\mathcal{U}_\mathrm{cov}$ of discrete Segal types and covariant maps between them to the theory described in \cite{riehl-dhott}.
A naive definition of such a universe postulates that we have an equivalence between the type of functions $X \to \mathcal{U}_\mathrm{cov}$ and the type of covariant fibrations over $X$.
This is not a correct definition since this condition is too strong.
The correct definition requires only an equivalence between the space of maps from $X$ to $\mathcal{U}_\mathrm{cov}$ and the space of covariant fibrations over $X$.
It is impossible to formulate this condition in ordinary homotopy type theory, but it is easy to do in an indexed type theory as we will see in section~\ref{sec:class}.

We will define two kinds of indexed type theories: unary and dependent.
Indexed unary type theories allow us to discuss properties of arbitrary categories.
Such a theory has two levels: the base theory and the indexed theory.
The base theory has the usual type-theoretic syntax and the language of the indexed theory can be described informally as the language of type-enriched categories.
Indexed dependent type theories also have two levels.
The first one is the same as for unary theories, but the second has type-theoretic syntax, so it is more convenient to use such theories, but it applies only to finitely complete theories.

We will define several categorical constructions in both styles and prove that they are equivalent:
\begin{enumerate}
\item Finite limits in unary theories and $\Sigma$-types, unit types, and identity types in dependent theories.
\item Finite colimits in unary theories and finite higher inductive types in dependent theories.
\item Exponents in unary theories and $\Pi$-types in dependent theories.
\item Object classifiers in unary theories and univalent universes in dependent theories.
\item Orthogonal factorization systems in unary theories and higher modalities in dependent theories.
\end{enumerate}

The paper is organized as follows.
In section~\ref{sec:unary}, we define indexed unary type theories.
In section~\ref{sec:equivalence}, we define the notion of an equivalence in unary theories.
In section~\ref{sec:colimits}, we define limits and colimits in unary theories.
In section~\ref{sec:dependent}, we define indexed dependent type theories.
In section~\ref{sec:lccc}, we define exponent and $\Pi$-types in indexed theories.
In section~\ref{sec:colimits-dep}, we define limits and colimits in dependent theories.
In section~\ref{sec:initial}, we prove the initial type theorem, which is the first step in the proof of the general adjoint functor theorem.
In section~\ref{sec:class}, we defined object classifiers.
In section~\ref{sec:refl-fib}, we define orthogonal factorization systems.

\section{Indexed unary type theories}
\label{sec:unary}

We can think about an indexed type theory as a syntactic representation of indexed $\infty$-categories, that is a functor $F$ from an $\infty$-category $\mathcal{B}$ to the large $\infty$-category of $\infty$-categories.
An indexed type theory consists of two levels.
The first level is just an ordinary type theory and it represents $\mathcal{B}$
Since we are mostly interested in the case when $\mathcal{B}$ is the $\infty$-category of spaces,
we can assume that the first level has all usual constructions such as identity types, $\Sigma$-types, $\Pi$-types, (univalent) universes, and (higher) inductive types.
Nevertheless, in general, we will assume that the base theory has only identity types and $\Sigma$-types; all additional assumptions will be explicitly specified.
We will often talk about functions, but this is only for notational convenience and does not assume that function types exist.
Terms of type $A \to B$ correspond to terms of type $B$ in context $x : A$, so we can talk about functions $A \to B$ as long as this type does not appear inside other types.

The second level of the theory represents $\infty$-categories $F(\Gamma)$ for various objects $\Gamma$ of $\mathcal{B}$.
In this section, we will discuss \emph{indexed unary type theories}, that is indexed type theories in which the second level consists of unary type theories.
A unary type theory is a non-dependent type theory in which contexts consist of exactly one type.
Such theories represent arbitrary 1-categories.
We do not know whether indexed unary type theories represent all indexed $\infty$-categories over a given base, but it seems that this should be true at least for locally small indexed $\infty$-categories.

Indexed unary type theories have four kinds of judgments:
\[ \Gamma \vdash A \type \qquad \Gamma \vdash a : A \qquad \Gamma \mid \cdot \vdash B \ob \qquad \Gamma \mid x : A \vdash b : B \]

In each of these judgments, $\Gamma$ is a context, that is a sequence of the form $x_1 : A_1, \ldots x_n : A_n$, where $A_1$, \ldots $A_n$ are types and $x_1$, \ldots $x_n$ are pairwise distinct variables.
Judgments $\Gamma \vdash A \type$ and $\Gamma \vdash a : A$ represent types and terms of the first level of the theory.
We will call such types and terms \emph{base types} and \emph{base terms}, respectively.
The collection of rules that involve only judgments for base types and base terms will be called the base (sub)theory.
When we say that the base theory has some construction such as $\Pi$-types or universes, this means that there are usual rules for these constructions formulated in terms of these judgments.

Judgments $\Gamma \mid \cdot \vdash A \ob$ represent types of the second level of the theory.
We will call these types \emph{indexed types} to distinguish them from base types.
In a judgment $\Gamma \mid x : A \vdash b : B$, $x$ is a variable which is distinct from the variables in $\Gamma$, $A$ and $B$ are indexed types, and $b$ is a term of the second level of the theory.
We will call such terms \emph{indexed terms}.
Indexed types represent objects indexed by $\Gamma$ and indexed an indexed term $\Gamma \mid x : A \vdash b : B$ represents a morphism between $A$ and $B$.

We have the usual rules for variables and substitutions for the base theory:
\begin{center}
\AxiomC{}
\UnaryInfC{$x_1 : A_1, \ldots x_n : A_n \vdash x_i : A_i$}
\DisplayProof
\end{center}

\begin{center}
\def\extraVskip{1pt}
\Axiom$\fCenter \Gamma \vdash b_1 : B_1$
\noLine
\UnaryInf$\fCenter \ldots$
\noLine
\UnaryInf$\fCenter \Gamma \vdash b_k : B_k[b_1/y_1, \ldots b_{k-1}/y_{k-1}]$
\Axiom$\fCenter \Gamma, y_1 : B_1, \ldots y_k : B_k \vdash C \type$
\def\extraVskip{2pt}
\BinaryInfC{$\Gamma \vdash C[b_1/y_1, \ldots b_k/y_k] \type$}
\DisplayProof
\end{center}

\begin{center}
\def\extraVskip{1pt}
\Axiom$\fCenter \Gamma \vdash b_1 : B_1$
\noLine
\UnaryInf$\fCenter \ldots$
\noLine
\UnaryInf$\fCenter \Gamma \vdash b_k : B_k[b_1/y_1, \ldots b_{k-1}/y_{k-1}]$
\Axiom$\fCenter \Gamma, y_1 : B_1, \ldots y_k : B_k \vdash c : C$
\def\extraVskip{2pt}
\BinaryInfC{$\Gamma \vdash c[b_1/y_1, \ldots b_k/y_k] : C[b_1/y_1, \ldots b_k/y_k]$}
\DisplayProof
\end{center}

We also have the usual equations for substitution:
\begin{align*}
y_i[b_1/y_1, \ldots b_k/y_k] & = b_i \\
c[y_1/y_1, \ldots y_k/y_k] & = c \\
d[c_1/z_1, \ldots c_n/z_n][b_1/y_1, \ldots b_k/y_k] & = d[c_1'/z_1, \ldots c_n'/z_n],
\end{align*}
where $c_i' = c_i[b_1/y_1, \ldots b_k/y_k]$.

For every construction $\sigma(\overline{z_1}.\,c_1, \ldots \overline{z_n}.\,c_n)$ in the base theory, we have the following equation whenever variables $\overline{z_1}$, \ldots $\overline{z_n}$ are not free in $b_1$, \ldots $b_k$:
\[ \sigma(\ldots, \overline{z_i}.\,c_i, \ldots)[b_1/y_1, \ldots b_k/y_k] = \sigma(\ldots, \overline{z_i}.\,c_i[b_1/y_1, \ldots b_k/y_k], \ldots) \]
We also have the weakening operation which is left implicit as usual.
This concludes the description of basic rules of the base theory.
They are the usual rules of a dependent type theory which we include here so that they can be compared to the rules of the indexed theory.

Variables of the indexed theory represent identity morphisms and substitution represents composition:
\begin{center}
\AxiomC{}
\UnaryInfC{$\Gamma \mid x : A \vdash x : A$}
\DisplayProof
\qquad
\AxiomC{$\Gamma \mid \Delta \vdash b : B$}
\AxiomC{$\Gamma \mid y : B \vdash c : C$}
\BinaryInfC{$\Gamma \mid \Delta \vdash c[b/y] : C$}
\DisplayProof
\end{center}

These operations satisfy the obvious equations:
\begin{align*}
y[b/y] & = b \\
b[x/x] & = b \\
d[c/z][b/y] & = d[c[b/y]/z]
\end{align*}

We can also substitute base terms into indexed types and terms:
\begin{center}
\def\extraVskip{1pt}
\Axiom$\fCenter \Gamma \vdash b_1 : B_1$
\noLine
\UnaryInf$\fCenter \ldots$
\noLine
\UnaryInf$\fCenter \Gamma \vdash b_k : B_k[b_1/y_1, \ldots b_{k-1}/y_{k-1}]$
\Axiom$\fCenter \Gamma, y_1 : B_1, \ldots y_k : B_k \mid \cdot \vdash C \ob$
\def\extraVskip{2pt}
\BinaryInfC{$\Gamma \mid \cdot \vdash C[b_1/y_1, \ldots b_k/y_k] \ob$}
\DisplayProof
\end{center}

\begin{center}
\def\extraVskip{1pt}
\Axiom$\fCenter \Gamma \vdash b_1 : B_1$
\noLine
\UnaryInf$\fCenter \ldots$
\noLine
\UnaryInf$\fCenter \Gamma \vdash b_k : B_k[b_1/y_1, \ldots b_{k-1}/y_{k-1}]$
\Axiom$\fCenter \Gamma, y_1 : B_1, \ldots y_k : B_k \mid z : C \vdash d : D$
\def\extraVskip{2pt}
\BinaryInfC{$\Gamma \mid z : C[b_1/y_1, \ldots b_k/y_k] \vdash d[b_1/y_1, \ldots b_k/y_k] : D[b_1/y_1, \ldots b_k/y_k]$}
\DisplayProof
\end{center}

These operations represent reindexing along a morphism in the base category.
They satisfy the following equations:
\begin{align*}
x[b_1/y_1, \ldots b_k/y_k] & = x \\
d[c/z][b_1/y_1, \ldots b_k/y_k] & = d[b_1/y_1, \ldots b_k/y_k][c[b_1/y_1, \ldots b_k/y_k]/z] \\
c[y_1/y_1, \ldots y_k/y_k] & = c \\
d[c_1/z_1, \ldots c_n/z_n][b_1/y_1, \ldots b_k/y_k] & = d[c_1'/z_1, \ldots c_n'/z_n],
\end{align*}
where $c_i' = c_i[b_1/y_1, \ldots b_k/y_k]$.
The first two equations correspond to the fact that reindexing preserves identity morphisms and composition in the indexed theory.
The last two equations correspond to the fact that reindexing is functorial, that is it preserves identity morphisms and composition in the base theory.

For every construction $\sigma(\overline{z_1}.\,c_1, \ldots \overline{z_n}.\,c_k)$ in the indexed theory, we have the following equation whenever variables $\overline{z_1}$, \ldots $\overline{z_n}$ are not free in $b_1$, \ldots $b_k$:
\[ \sigma(\ldots, \overline{z_i}.\,c_i, \ldots)[b_1/y_1, \ldots b_k/y_k] = \sigma(\ldots, \overline{z_i}.\,c_i[b_1/y_1, \ldots b_k/y_k], \ldots) \]
We also have the weakening operation which is left implicit as usual.
This equation corresponds to the fact that all constructions in the indexed category must be stable under reindexing.

As we noted before, we assume that the base theory has identity types:
\begin{center}
\AxiomC{$\Gamma \vdash a : A$}
\AxiomC{$\Gamma \vdash a' : A$}
\BinaryInfC{$\Gamma \vdash \Id_A(a,a') \type$}
\DisplayProof
\qquad
\AxiomC{$\Gamma \vdash a : A$}
\UnaryInfC{$\Gamma \vdash \refl(a) : \Id_A(a,a)$}
\DisplayProof
\end{center}
\medskip

\begin{center}
\def\extraVskip{1pt}
\Axiom$\fCenter \Gamma \vdash a : A$
\noLine
\UnaryInf$\fCenter \Gamma \vdash a' : A$
\noLine
\UnaryInf$\fCenter \Gamma \vdash t : \Id_A(a,a')$
\Axiom$\fCenter \Gamma, x : A, p : \Id_A(a,x), \Delta \vdash D \type$
\noLine
\UnaryInf$\fCenter \Gamma, \Delta[a/x,\refl(a)/p] \vdash d : D[a/x,\refl(a)/p]$
\def\extraVskip{2pt}
\BinaryInfC{$\Gamma, \Delta[a'/x,t/p] \vdash J(a, x p \Delta.\,D, \Delta.\,d, a', t) : D[a'/x,t/p]$}
\DisplayProof
\end{center}

\[ J(a, x p \Delta.\,D, \Delta.\,d, a, \refl(a)) = d \]

We will sometimes omit the type in the notation $\Id_A(a,a')$.
The fact that the type $\Id(a,a')$ is inhabited will be denoted by $a \sim a'$.
If $\Gamma \vdash p : \Id_A(a,a')$, $\Gamma, x : A \vdash B \type$, and $\Gamma \vdash b : B[a/x]$, then we will write $\Gamma \vdash p_*(b) : B[a'/x]$ for the usual transport operation defined in terms of $J$.
Operation $\pmap$ is defined in terms of $J$ and has the following type:
if $\Gamma \vdash B \type$, $\Gamma, x : A \vdash b : B$, and $\Gamma \vdash p : \Id_A(a,a')$, then $\Gamma \vdash \pmap(x.b, p) : \Id_B(b[a/x],b[a'/x])$.

An indexed type theory is \emph{locally small} if there is a type of its morphisms.
That is, it must contain the following rules and equations:
\begin{center}
\AxiomC{$\Gamma \mid \cdot \vdash A \ob$}
\AxiomC{$\Gamma \mid \cdot \vdash B \ob$}
\BinaryInfC{$\Gamma \vdash \Hom(A,B) \type$}
\DisplayProof
\qquad
\AxiomC{$\Gamma \mid x : A \vdash b : B$}
\UnaryInfC{$\Gamma \vdash \lambda x.\,b : \Hom(A,B)$}
\DisplayProof
\end{center}
\medskip

\begin{center}
\AxiomC{$\Gamma \vdash f : \Hom(A,B)$}
\AxiomC{$\Gamma \mid \Delta \vdash a : A$}
\BinaryInfC{$\Gamma \mid \Delta \vdash f\,a : B$}
\DisplayProof
\end{center}

\begin{align*}
(\lambda x.\,b)\,a & = b[a/x] \\
\lambda x.\,f\,x & = f
\end{align*}

We might also use notation $A \to B$ for $\Hom(A,B)$, but we prefer the latter notation since the former may be confusing.
Indeed, $A \to B$ might also denote the indexed type of functions if the indexed theory is Cartesian closed or the base type of function if $A$ and $B$ are base types.

If the indexed theory is locally small, then indexed types must carry the structure of an $\infty$-category.
We cannot construct this structure internally due to coherence issues, but we can at least construct lower levels of this structure.
Morphisms between indexed types $A$ and $B$ are terms of type $\Hom(A,B)$.
The identity morphism $\id_A$ on an indexed type $A$ is $\lambda x.\,x : \Hom(A,A)$.
Composition of morphisms $f : \Hom(A,B)$ and $g : \Hom(B,C)$ is defined as $\lambda x.\,g\,(f\,x) : \Hom(A,C)$ and denoted by $g \circ f$.
Composition is strictly associative and identity morphisms are strictly unital.

If $f,g : \Hom(A,B)$ are morphisms, then a 2-morphism between them is a term $p : \Id_{\Hom(A,B)}(f,g)$.
Vertical composition $p \ct q$ of 2-morphisms $p$ and $q$ is defined as the usual operation of path concatenation.
The identity 2-morphism on $f : \Hom(A,B)$ is $\refl(f)$.
Vertical composition is associative, identity 2-morphisms are unital, and every 2-morphism is invertible.
These facts are true in a weak sense, that is up to a 3-morphism.
Let $f,g : \Hom(A,B)$ and $h,i : \Hom(B,C)$ be morphisms and let $p : \Id_{\Hom(A,B)}(f,g)$ and $q : \Id_{\Hom(B,C)}(h,i)$ be 2-morphisms.
The horizontal composition of $p$ and $q$ is a term $p * q$ of type $\Id_{\Hom(A,C)}(\lambda x.\,h\,(f\,x), \lambda x.\,i\,(g\,x))$.
To define $p * q$, we just need to eliminate $p$ and $q$ and then define $\refl(f) * \refl(h)$ as $\refl(\lambda x.\,h\,(f\,x))$.
It is easy to prove that usual properties of this operation hold.
Expressions $\refl(f) * q$, $p * \refl(g)$, and $\refl(f) * \refl(g)$ will be denoted by $f * q$, $p * g$, and $f * g$, respectively.

\begin{defn}
An equivalence between indexed types $A$ and $B$ is a morphism $f : \Hom(A,B)$ such that there is a morphism $g : \Hom(B,A)$ such that $g \circ f \sim \id_A$ and $f \circ g \sim \id_B$.
\end{defn}

If the indexed theory is locally small, then not only base morphisms act on indexed types and terms, but also homotopies between them.
Let $\Gamma \vdash a : A$ and $\Gamma \vdash a' : A$ be two base terms.
If $\Gamma, x : A \mid \cdot \vdash B \ob$ be an indexed type, then we have indexed types $\Gamma \mid \cdot \vdash B[a/x] \ob$ and $\Gamma \mid \cdot \vdash B[a'/x] \ob$.
Let $\Gamma \vdash h : \Id_A(a,a')$ be homotopy between $a$ and $a'$.
Then we can construct an equivalence between $B[a/x]$ and $B[a'/x]$.
A map $f : \Hom(B[a/x],B[a'/x])$ is defined as $J(a, x p.\,\Hom(B[a/x],B), \id_{B[a/x]}, a', h)$.
A map $g : \Hom(B[a'/x],B[a/x])$ is constructed similarly: $\lambda y.\,J(a, x p.\,\Hom(B,B[a/x]), \id_{B[a/x]}, a', h)$.
To prove that $g \circ f$ and $f \circ g$ are homotopic to identity morphisms, it is enough to eliminate $h$ using $J$ and then both $g \circ f$ and $f \circ g$ become identity morphisms.

\section{Equivalences}
\label{sec:equivalence}

In this section, we define types that express the property of a map $f : \Hom(A,B)$ of being an equivalence and prove that they are equivalent.
We also prove a few simple properties of equivalences.
These questions were studied in \cite[Section~4]{hottbook} for ordinary homotopy type theory.
Most of the theorems in this section also hold in the framework of indexed unary type theories, but the proofs must be modified.

\subsection{Bi-invertible maps}

Let $f : \Hom(A,B)$ be a morphism.
We will say that $f$ is \emph{bi-invertible} if the following type is inhabited:
\[ \fs{biinv}(f) = \fs{linv}(f) \times \fs{rinv}(f), \]
where $\fs{linv}(f)$ and $\fs{rinv}(f)$ are types of left and right inverses of $f$, respectively:
\begin{align*}
\fs{linv}(f) & = \sum_{g : \Hom(B,A)} \Id(g \circ f, \id_A) \\
\fs{rinv}(f) & = \sum_{g : \Hom(B,A)} \Id(f \circ g, \id_B)
\end{align*}

\begin{prop}[biinv-equiv]
A map is bi-invertible if and only if it is an equivalence.
\end{prop}
\begin{proof}
Obviously, if a map is an equivalence, then it is bi-invertible.
Let us prove the converse.
Let $g : \Hom(B,A)$, $p : \Id(g \circ f, \id_A)$ be a left inverse of $f$ and let $g' : \Hom(B,A)$, $p' : \Id(f \circ g', \id_B)$ be a right inverse of $f$.
Then $g' * p : \Id(g \circ f \circ g', g')$.
Since $f \circ g' \sim \id_B$, there is a term of type $\Id(g,g')$.
It follows that $g$ is an inverse of $f$.
\end{proof}

\begin{lem}[lrinv-contr]
If $f$ is an equivalence, then types $\fs{linv}(f)$ and $\fs{rinv}(f)$ are contractible.
\end{lem}
\begin{proof}
If $f$ is an equivalence, then precomposition with $f$ is an equivalence between types $\Hom(B,C)$ and $\Hom(A,C)$.
Similarly, postcomposition with $f$ is an equivalence between types $\Hom(C,A)$ and $\Hom(C,B)$.
Since $\fs{linv}(f)$ and $\fs{rinv}(f)$ are fibres of these maps over the identity morphisms, \cite[Theorem~4.2.3]{hottbook} and \cite[Theorem~4.2.6]{hottbook} imply that these types are contractible.
Note that the proofs of these theorems work even if we do not have $\Pi$-types.
\end{proof}

\begin{prop}
The type $\fs{biinv}(f)$ is a proposition.
\end{prop}
\begin{proof}
This follows from \rprop{biinv-equiv} and \rlem{lrinv-contr}.
\end{proof}

\subsection{Half adjoint equivalences}

Let $f : \Hom(A,B)$ be a morphism.
We will say that $f$ is a \emph{half adjoint equivalence} if the following type is inhabited:
\[ \fs{ishae}(f) = \sum_{g : \Hom(B,A)} \sum_{\eta : \Id(g \circ f, \id_A)} \sum_{\epsilon : \Id(f \circ g, \id_B)} \Id(\eta * f, f * \epsilon). \]

\begin{prop}
A map is a half adjoint equivalence if and only if it is an equivalence.
\end{prop}
\begin{proof}
Obviously, if a map is a half adjoint equivalence, then it is an equivalence.
Let us prove the converse.
Let $g : \Hom(B,A)$, $\eta : \Id(g \circ f, \id_A)$, $\epsilon : \Id(f \circ g, \id_B)$ be an inverse of $f$.
Then we define $\epsilon' : \Id(f \circ g, \id_B)$ as concatenation of paths
$g * f * \sym{\epsilon} : \Id(f \circ g, f \circ g \circ f \circ g)$, $g * \eta * f : \Id(f \circ g \circ f \circ g, f \circ g)$, and $\epsilon : \Id(f \circ g, \id_B)$.
We need to prove that $f * \epsilon' \sim \eta * f$.

First, note that $\eta * f * g \sim f * g * \eta$.
Indeed, $(\eta * f * g) \ct \eta \sim \eta * \eta \sim (f * g * \eta) \ct \eta$.
Thus, if we cancel $\eta$, this gives us a homotopy between the original paths.
Now, we can finish the proof:
\begin{align*}
f * \epsilon' & \sim \\
(f * g * f * \sym{\epsilon}) \ct (f * g * \eta * f) \ct (f * \epsilon) & \sim \\
(f * g * f * \sym{\epsilon}) \ct (\eta * f * g * f) \ct (f * \epsilon) & \sim \\
(\eta * f * \sym{\epsilon}) \ct (f * \epsilon) & \sim \\
(\eta * f) \ct (f * \sym{\epsilon}) \ct (f * \epsilon) & \sim \\
\eta * f & .
\end{align*}
\end{proof}

\begin{prop}
The type $\fs{ishae}(f)$ is a proposition.
\end{prop}
\begin{proof}
We can assume that $f$ is an equivalence and prove that $\fs{ishae}(f)$ is contractible.
By \rlem{lrinv-contr}, the type $\Sigma_{g : \Hom(B,A)} \Id(g \circ f, \id_A)$ is contractible.
Thus, we just need to prove that, for every $g : \Hom(B,A)$ and $\eta : \Id(g \circ f, \id_A)$, the type $\Sigma_{\epsilon : \Id(f \circ g, \id_B)} \Id(\eta * f, f * \epsilon)$ is also contractible.

Since $f$ is an equivalence, the function $f * -$ is also an equivalence.
It follows that the type $\Id(\eta * f, f * \epsilon)$ is equivalent to the type $\Id(h(\eta * f), \epsilon)$, where $h$ is the inverse of $f * -$.
Thus, the type $\Sigma_{\epsilon : \Id(f \circ g, \id_B)} \Id(\eta * f, f * \epsilon)$ is equivalent to the type $\Sigma_{\epsilon : \Id(f \circ g, \id_B)} \Id(h(\eta * f), \epsilon)$, which is contractible by \cite[Lemma~3.11.8]{hottbook}.
\end{proof}

\subsection{Properties of equivalences}

\begin{prop}
Equivalences satisfy the 2-out-of-6 property.
That is, if $f : \Hom(A,B)$, $g : \Hom(B,C)$, and $h : \Hom(C,D)$ are maps such that $g \circ f$ and $h \circ g$ are equivalences, then so are the maps $f$, $g$, $h$, and $h \circ g \circ f$.
\end{prop}
\begin{proof}
Let $i : \Hom(C,A)$ be an inverse of $g \circ f$ and let $k : \Hom(D,B)$ be an inverse of $h \circ g$.
Since $g \circ f \circ i \sim \id_C$ and $h \circ g \circ k \sim \id_D$, \rprop{biinv-equiv} implies that $g$ is an equivalence.
The map $i \circ g$ is an inverse of $f$.
Indeed, $i \circ g \circ f \sim \id_A$ since $i$ is an inverse of $g \circ f$.
Since $g \circ f \circ i \sim \id_C$, it follows that $g \circ f \circ i \circ g \sim g$.
Since $g$ is an equivalence, this implies that $f \circ i \circ g \sim \id_B$.
Similarly, $g \circ k$ is an inverse of $h$.
The map $h \circ g \circ f$ is an equivalence since equivalences are closed under composition.
\end{proof}

A map $f : \Hom(A,B)$ is a \emph{quasi-retract} of a map $g : \Hom(C,D)$ if there is a commutative diagram of the form
\[ \xymatrix{ A \ar[r]^i \ar[d]_f & C \ar[r]^j \ar[d]_g & A \ar[d]_f \\
              B \ar[r]_k          & D \ar[r]_m          & B
            } \]
such that $j \circ i \sim \id_A$ and $m \circ k \sim \id_B$.

\begin{prop}
Equivalences are closed under quasi-retracts.
\end{prop}
\begin{proof}
Let $f : \Hom(A,B)$ be a retract of $g : \Hom(C,D)$ and let $i,j,k,m$ be maps as in the diagram above.
Let $h : \Hom(D,C)$ be an inverse of $g$.
Then $j \circ h \circ k$ is an inverse of $f$.
Indeed, $j \circ h \circ k \circ f \sim j \circ h \circ g \circ i \sim j \circ \id_C \circ i = j \circ i \sim \id_B$ and
$f \circ j \circ h \circ k \sim m \circ g \circ h \circ k \sim m \circ \id_D \circ k = m \circ k \sim \id_B$.
\end{proof}

\section{Limits and colimits}
\label{sec:colimits}

In this section, we will work in a locally small indexed unary type theory.
We will define specific finite (co)limits and arbitrary (co)products.

\subsection{Finite (co)limits}

An indexed type $T$ is \emph{terminal} if, for every indexed type $X$, the type $\Hom(X,T)$ is contractible.
Dually, an indexed type $T$ is \emph{initial} if, for every indexed type $X$, the type $\Hom(T,X)$ is contractible.
Terminal and initial types are unique up to unique equivalence, that is the type of equivalences between a pair of terminal or initial types is contractible.
We will say that an indexed unary type theory \emph{has terminal (resp., initial) types} if, for every context $\Gamma$, there is a terminal (resp., initial) type $\Gamma \mid \cdot \vdash T \ob$.

A \emph{binary product} of indexed types $A$ and $B$ is an indexed type $A \times B$ together with a pair of maps $\pi_1 : \Hom(A \times B, A)$ and $\pi_2 : \Hom(A \times B, B)$
such that the following function is an equivalence for every indexed type $C$:
\[ \lambda h.\,(\pi_1 \circ h, \pi_2 \circ h) : \Hom(C, A \times B) \to \Hom(C,A) \times \Hom(C,B). \]
The inverse of this function will be denoted by $\langle -, - \rangle$.
An indexed unary type theory \emph{has binary products} if a binary product exists for every pair of types in every context.
\emph{Binary coproducts} $A \amalg B$ are defined dually.

An \emph{equalizer} of a pair of maps $f,g : \Hom(A,B)$ is a map $e : \Hom(E,A)$ together with a homotopy $p : \Id(f \circ e, g \circ e)$
such that the following function is an equivalence for every indexed type $E'$:
\[ \lambda h.\,(e \circ h, h * p) : \Hom(E', E) \to \sum_{e' : \Hom(E',A)} \Id(f \circ e', g \circ e'). \]
An indexed unary type theory \emph{has equalizers} if an equalizer exists for every parallel pair of maps in every context.
\emph{Coequalizers} are defined dually.

A \emph{pullback} of a pair of maps $f : \Hom(A,C)$ and $g : \Hom(B,C)$ is a triple $\pi_1 : \Hom(A \times_C B, A)$, $\pi_2 : \Hom(A \times_C B, B)$, $\pi_3 : \Id(f \circ \pi_1, g \circ \pi_2)$
such that the following function is an equivalence for every indexed type $P'$:
\[ \lambda h.\,(\pi_1 \circ h, \pi_2 \circ h, h * \pi_3) : \Hom(P, A \times_C B) \to \Hom(P,A) \times_{\Hom(P,C)} \Hom(P,B), \]
where the pullback of types $\Hom(P,A) \times_{\Hom(P,C)} \Hom(P,B)$ is defined as usual:
\[ \sum_{\pi_1' : \Hom(P,A)} \sum_{\pi_2' : \Hom(P,B)} \Id(f \circ \pi_1', g \circ \pi_2'). \]
An indexed unary type theory \emph{has pullbacks} if a pullback exists for every pair of maps with a common codomain in every context.
\emph{Pushouts} $A \amalg_C B$ are defined dually.

\begin{remark}
Binary (co)products, (co)equalizers, pullbacks, and pushouts are unique up to unique equivalence.
\end{remark}

\begin{remark}
The function $\Hom(C,-)$ preserves binary products, equalizers, and pullbacks.
\end{remark}

We have the following standard proposition:

\begin{prop}[fin-lim]
An indexed unary type theory with terminal types has pullbacks if and only if it has equalizers and binary products.
\end{prop}
\begin{proof}
First, suppose that the theory has a terminal type $1$ and pullbacks.
Then we can define a product of types $A$ and $B$ as the pullback of unique maps $!_A : \Hom(A,1)$ and $!_B : \Hom(B,1)$.
Since $\Hom(P,1)$ is contractible, the obvious projection $\Hom(P,A) \times_{\Hom(P,1)} \Hom(P,B) \to \Hom(P,A) \times \Hom(P,B)$ is an equivalence.
This implies that $A \times_1 B$ is a product of $A$ and $B$.

An equalizer of maps $f,g : \Hom(A,B)$ can be defined as the pullback of $\langle \id_B, \id_B \rangle : \Hom(B, B \times B)$ along $\langle f, g \rangle : \Hom(A, B \times B)$:
\[ \xymatrix{ E \ar[r]^s \ar[d]_e & B \ar[d]^{\langle \id_B, \id_B \rangle} \\
              A \ar[r]_-{\langle f, g \rangle} & B \times B
            } \]
By the definition of products, the type of homotopies $\Id_{\Hom(P, B \times B)}(r,r')$ is equivalent to the type $\Id_{\Hom(P,B)}(\pi_1 \circ r, \pi_1 \circ r') \times \Id_{\Hom(P,B)}(\pi_2 \circ r, \pi_2 \circ r')$.
Thus, by the definition of pushouts, we have the following equivalence:
\begin{align*}
& \Hom(P,E) \to \sum_{a : \Hom(P,A)} \sum_{b : \Hom(P,B)} \Id(f \circ a, b) \times \Id(g \circ a, b) \\
& \lambda q.\,(e \circ q, s \circ q, q * h_1, q * h_2),
\end{align*}
where $h_1 : \Id(f \circ e, s)$ and $h_2 : \Id(g \circ e, s)$ are certain homotopies.
The codomain of this function is equivalent to $\Sigma_{a : \Hom(P,A)} \Id(f \circ a, g \circ a)$.
This implies that we have the following equivalence:
\[ \lambda q.\,(e \circ q, q * (h_1 \ct \sym{h_2})) : \Hom(P,E) \to \sum_{a : \Hom(P,A)} \Id(f \circ a, g \circ a). \]
Thus, we can define an equalizer of maps $f$ and $g$ as the triple $E$, $e$, $h_1 \ct \sym{h_2}$.

Now, suppose that the theory has binary products and equalizers.
Let $f : \Hom(A,C)$ and $g : \Hom(B,C)$ be a pair of maps.
Let $e : P \to A \times B$, $h : \Id(f \circ \pi_1 \circ e, g \circ \pi_2 \circ e)$ be the equalizer of the maps $f \circ \pi_1, g \circ \pi_2 : A \times B \to C$.
Then we can define a pullback of $f$ and $g$ as the triple $\pi_1 \circ e$, $\pi_2 \circ e$, $h$.
The universal property of equalizers implies the universal property of pullbacks.
\end{proof}

\begin{defn}[fin-lim]
An indexed unary type theory \emph{has finite limits} if equivalent conditions of \rprop{fin-lim} hold.
\end{defn}

\begin{example}
If an indexed unary type theory has a terminal type $\term$ with a point $\unit : \term$, then the \emph{loop space type} of a pointed type $Y$, $y_0 : \Hom(\term,Y)$ is the pullback of $y_0$ and $y_0$.
Equivalently, the loop space type is the equalizer of $y_0$ and $y_0$.
Thus, the loop space type is a type $\Omega(Y,y_0)$ together with a homotopy $\Id_{\Hom(\Omega(Y,y_0), Y)}(\lambda s.\,y_0\,\unit, \lambda s.\,y_0\,\unit)$ satisfying the universal property.
Since $\Hom(X,-)$ preserves terminal types and pullbacks, we have the following equivalence:
\[ \Hom(X, \Omega(Y,y_0)) \simeq \Omega(\Hom(X,Y), \lambda x.\,y_0\,\unit), \]
where the second $\Omega$ is the usual loop space base type: $\Omega(S,s_0) = \Id_S(s_0,s_0)$.
\end{example}

\begin{example}
The \emph{suspension} $\Sigma X$ of a type $X$ is the pushout of the maps $\lambda x.\,\unit, \lambda x.\,\unit : \Hom(X,\term)$.
The \emph{$0$-sphere} $S^0$ is the coproduct $\term \amalg \term$.
The \emph{$(n+1)$-sphere} $S^{n+1}$ is the suspension $\Sigma S^n$.
\end{example}

\subsection{(Co)products}
\label{sec:products}

A \emph{product} of an indexed type $\Gamma, i : I \mid \cdot \vdash B \ob$ is an indexed type $\Gamma \mid \cdot \vdash P \ob$ together with a term $\Gamma, i : I \vdash \pi : \Hom(P,B)$
such that the function $\pi \circ -$ has an inverse in the sense that there is a rule of the form
\begin{center}
\AxiomC{$\Gamma \mid \cdot \vdash P' \ob$}
\AxiomC{$\Gamma, i : I \vdash f : \Hom(P',B)$}
\BinaryInfC{$\Gamma \vdash \langle f \rangle_{i : I} : \Hom(P',P)$}
\DisplayProof
\end{center}
and the following types are inhabited:
\begin{align*}
& \Id(\pi \circ \langle f \rangle_{i : I}, f) \\
& \Id(\langle \pi \circ f \rangle_{i : I}, f).
\end{align*}

The theory of coproducts is defined dually.
A \emph{coproduct} of an indexed type $\Gamma, i : I \mid \cdot \vdash B \ob$ is an indexed type $\Gamma \mid \cdot \vdash C \ob$ together with a term $\Gamma, i : I \vdash \fs{in} : \Hom(B,C)$
such that the function $- \circ \fs{in}$ has an inverse in the sense that there is a rule of the form
\begin{center}
\AxiomC{$\Gamma \mid \cdot \vdash C' \ob$}
\AxiomC{$\Gamma, i : I \vdash f : \Hom(B,C')$}
\BinaryInfC{$\Gamma \vdash [ f ]_{i : I} : \Hom(C,C')$}
\DisplayProof
\end{center}
and the following types are inhabited:
\begin{align*}
& \Id([ f ]_{i : I} \circ \fs{in}, f) \\
& \Id([ f \circ \fs{in} ]_{i : I}, f).
\end{align*}

If the $\Pi$-type $\Pi_{i : I} \Hom(P',B)$ exists for all indexed types $\Gamma \mid \cdot \vdash P' \ob$, then a pair $P$, $\pi$ is a product of a family $B$ if and only if the following function is an equivalence for every indexed type $P'$:
\[ \lambda h.\,\lambda i.\,\pi \circ h : \Hom(P',P) \to \prod_{i : I} \Hom(P',B). \]
Dually, if the $\Pi$-type $\Pi_{i : I} \Hom(B,C')$ exists for all indexed types $\Gamma \mid \cdot \vdash C' \ob$, then a pair $C$, $\fs{in}$ is a product of a family $B$ if and only if the following function is an equivalence for every indexed type $C'$:
\[ \lambda h.\,\lambda i.\,h \circ \fs{in}(i) : \Hom(C,C') \to \prod_{i : I} \Hom(B,C'). \]

Products and coproducts are unique up to unique equivalence.
We will denote the product and the coproduct of a family $\Gamma, i : I \mid \cdot \vdash B \type$ by $\prod_{i : I} B$ and $\coprod_{i : I} B$, respectively.
We will say that the product $\prod_{i : I} B$ is \emph{extensional} if the type $\Hom(P, \prod_{i : I} B)$ satisfies functional extensionality as a weak $\Pi$-type $\Pi_{i : I} \Hom(P,B)$ for all indexed types $P$.
Similarly, we will say that the coproduct $\coprod_{i : I} B$ is \emph{extensional} if the type $\Hom(\coprod_{i : I} B, C)$ satisfies functional extensionality as a weak $\Pi$-type $\Pi_{i : I} \Hom(B,C)$ for all indexed types $C$.

\begin{example}
If $\Gamma \vdash I \type$ is a base type and $\Gamma \mid \cdot \vdash X \type$ is an indexed type, then the \emph{power} (or \emph{cotensor}) of $X$ by $I$ is the product $\prod_{i : I} X$.
The \emph{copower} (or \emph{tensor}) of $X$ by $I$ is the coproduct $\coprod_{i : I} X$.
The power will be denoted by $X^I$ and the tensor by $I \cdot X$.
\end{example}

We can also define the theory of \emph{strict products}:
\begin{center}
\AxiomC{$\Gamma, i : I \mid \cdot \vdash B \ob$}
\UnaryInfC{$\Gamma \mid \cdot \vdash \prod_{i : I} B \ob$}
\DisplayProof
\qquad
\AxiomC{$\Gamma, i : I \mid \Delta \vdash b : B$}
\RightLabel{, $i \notin \mathrm{FV}(\Delta)$}
\UnaryInfC{$\Gamma \mid \Delta \vdash \lambda i.\,b : \prod_{i : I} B$}
\DisplayProof
\end{center}
\medskip

\begin{center}
\AxiomC{$\Gamma \mid \Delta \vdash f : \prod_{i : I} B$}
\AxiomC{$\Gamma \vdash j : I$}
\BinaryInfC{$\Gamma \mid \Delta \vdash f\,j : B[j/i]$}
\DisplayProof
\end{center}

\begin{align*}
(\lambda i.\,b)\,j & = b[j/i] \\
\lambda i.\,f\,i & = f
\end{align*}

The difference between weak and strict products is that the former requires types $\Hom(P',\prod_{i : I} B)$ and $\Pi_{i : I} \Hom(P',B)$ to be equivalent while the latter requires them to be \emph{isomorphic}.
The theory of \emph{weak products} is define in the same way as the theory of strict products with the difference that the last two equations hold only propositionally.

\begin{prop}
A type $\Gamma \mid \cdot \vdash \prod_{i : I} B$ is a product of a family $\Gamma, i : I \mid \cdot \vdash B \ob$ if and only if it is a weak product of this family.
\end{prop}
\begin{proof}
First, suppose that $\prod_{i : I} B$ is a product.
If $\Gamma, i : I \mid x : A \vdash b : B$, then we define $\lambda i.\,b$ as $\langle \lambda x.\,b \rangle_{i : I}\,x$.
If $\Gamma \mid x : A \vdash f : \prod_{i : I} B$ and $\Gamma \vdash j : I$, then we define $f\,j$ as $\pi[j/i]\,f$.
The $\beta$ and $\eta$ equations hold for these definitions:
\begin{align*}
& \lambda x.\,(\lambda i.b)\,j = \lambda x.\,\pi[j/i]\,(\langle \lambda x.b \rangle_{i : I}\,x) = (\pi \circ \langle \lambda x.b \rangle_{i : I})[j/i] \sim (\lambda x.b)[j/i] = \lambda x.\,b[j/i] \\
& \lambda x.\,\lambda i.\,f\,i = \lambda x.\,\langle \lambda x.\,\pi\,f \rangle_{i : I}\,x = \langle \lambda x.\,\pi\,f \rangle_{i : I} = \langle \pi \circ (\lambda x.f) \rangle_{i : I} \sim \lambda x.f
\end{align*}

Now, suppose that $\prod_{i : I} B$ is a weak product.
We define
\[ \Gamma, i : I \vdash \pi : \Hom(\prod_{i : I} B, B) \]
as $\lambda f.\,f\,i$.
If $\Gamma, i : I \vdash g : \Hom(P,B)$, then we define $\Gamma \vdash \langle g \rangle_{i : I} : \Hom(P, \prod_{i : I} B)$ as $\lambda x.\,\lambda i.\,g\,x$.
The required homotopies can be constructed as follows:
\begin{align*}
& \pi \circ \langle g \rangle_{i : I} = \lambda x.\,(\lambda i.\,g\,x)\,i \sim \lambda x.\,g\,x = g \\
& \langle \pi \circ g \rangle_{i : I} = \lambda x.\,\lambda i.\,g\,x\,i \sim \lambda x.\,g\,x = g
\end{align*}
\end{proof}

Let us prove a few properties of products and coproducts.
To simplify the notation, we will assume that the base theory of an indexed theory with a product $\prod_{i : I} B$ has $\Pi$-types $\Pi_{i : I} \Hom(P,B)$ for all $P$.
The following proposition shows that the (co)product of a contractible family of types is any type of this family:

\begin{prop}
Let $I$ be a contractible type and let $i_0$ be a point of $I$.
Then types $\prod_{i : I} B$, $\coprod_{i : I} B$, and $B[i_0/i]$ are equivalent.
\end{prop}
\begin{proof}
Let $p(i)$ be a path between $i_0$ and $i : I$.
Then the pair $B[i_0/i], \pi_i = \lambda x.\,p(i)_*(x)$ is a product of $B$.
We can define $\langle f \rangle_{i : I}$ as $\lambda x.\,\sym{p(i)}_*(f\,x)$.
Clearly, this is an inverse to $\pi \circ -$.
Since $B[i_0/i],\pi$ is a product and products are unique up to equivalence, it follows that $B[i_0/i]$ is equivalent to $\prod_{i : I} B$.
Similar argument shows that it is also equivalent to $\coprod_{i : I} B$.
\end{proof}

The following proposition shows how to compute products and coproducts indexed by $\Sigma$-types:

\begin{prop}
Let $\Gamma \vdash I \type$ and $\Gamma, i : I \vdash J \type$ be base types and let $\Gamma, i : I, j : J \mid \cdot \vdash B \ob$ be an indexed type.
Then types $\prod_{(p : \Sigma_{i : I} J)} B[\pi_1(p)/i, \pi_2(p)/j]$ and $\prod_{i : I} \prod_{j : J} B$ are equivalent.
Dually, types $\coprod_{(p : \Sigma_{i : I} J)} B[\pi_1(p)/i, \pi_2(p)/j]$ and $\coprod_{i : I} \coprod_{j : J} B$ are equivalent.
\end{prop}
\begin{proof}
We will prove this statement for products; the case of coproducts is dual.
To do this, it is enough to show that $\prod_{i : I} \prod_{j : J} B$ is a product of $\Gamma , p : \Sigma_{i : I} J \mid \cdot \vdash B[\pi_1(p)/i, \pi_2(p)/j] \ob$.
We define projections as follows:
\[ \lambda f.\,f\,(\pi_1(p))\,(\pi_2(p)) : \Hom(\prod_{i : I} \prod_{j : J} B, B[\pi_1(p)/i,\pi_2(p)/j]). \]
We need to show that the following map is an equivalence:
\begin{align*}
& \Hom(X, \prod_{i : I} \prod_{j : J} B) \to \prod_{p : \sum_{i : I} J} \Hom(X, B[\pi_1(p)/i,\pi_2(p)/j]) \\
& \lambda g.\,\lambda p.\,\lambda x.\,g\,x\,(\pi_1(p))\,(\pi_2(p)).
\end{align*}
Note that this map factors through the following maps:
\begin{align*}
\lambda g.\,\lambda i j.\,\lambda x.\,g\,x\,i\,j & : \Hom(X, \prod_{i : I} \prod_{j : J} B) \to \prod_{i : I} \prod_{j : J} \Hom(X,B) \\
\lambda h p.\,h\,(\pi_1(p))\,(\pi_2(p)) & : (\prod_{i : I} \prod_{j : J} \Hom(X,B)) \to \prod_{p : \sum_{i : I} J} \Hom(X, B[\pi_1(p)/i,\pi_2(p)/j]).
\end{align*}
The first map is an equivalence since $\prod_{i : I} \prod_{j : J} B$ is a product and the fact that the second map is an equivalence is an easy exercise in the ordinary type theory.
\end{proof}

The following proposition shows that the product of an empty family of types is the terminal object and the coproduct of such a family is initial:

\begin{prop}
Suppose that the base theory has the empty type $\bot$.
Let $\Gamma, i : \bot \mid \cdot \vdash B \ob$ be an indexed type.
Then $\prod_{i : \bot} B$ is terminal and $\coprod_{i : \bot} B$ is initial.
\end{prop}
\begin{proof}
Since $\Hom(P, \prod_{i : \bot} B)$ is equivalent to $\Pi_{i : \bot} \Hom(P,B)$ and $\Hom(\coprod_{i : \bot} B, P)$ is equivalent $\Pi_{i : \bot} \Hom(B,P)$,
the statement follows from the fact that $\Pi_{i : \bot} X$ is contractible for every base type $X$.
\end{proof}

The following proposition shows how to compute products and coproducts indexed by pushouts:

\begin{prop}
Suppose that the base theory has the following pushout:
\[ \xymatrix{ K \ar[r]^-g \ar[d]_-f & J \ar[d]^-{f'} \\
              I \ar[r]_-{g'}        & \po I \amalg_K J.
            } \]
Let $\Gamma, s : I \amalg_K J \mid \cdot \vdash B \ob$ be an indexed type.
Then we have the following canonical equivalences:
\begin{align*}
\prod_{s : I \amalg_K J} B & \simeq (\prod_{i : I} B[g' i/s]) \times_{(\prod_{k : K} B[g' (f k) / s])} (\prod_{j : J} B[f' j/s]) \\
\coprod_{s : I \amalg_K J} B & \simeq (\coprod_{i : I} B[g' i/s]) \amalg_{(\coprod_{k : K} B[f' (g k) / s])} (\coprod_{j : J} B[f' j/s]).
\end{align*}
\end{prop}
\begin{proof}
We will construct the first equivalence; the second is its dual.
First, let us define maps that appears in the pullback in the statement of this proposition.
The map $\Hom(\prod_{i : I} B[g' i/s], \prod_{k : K} B[g' (f k) / s])$ is defined as $\lambda p.\,\lambda k.\,p\,(f\,k)$.
One of the constructors of the pushout $I \amalg_K J$ gives us a map $h : \Pi_{k : K} \Id(f'\,(g\,k),g'\,(f\,k))$.
The map $\Hom(\prod_{j : J} B[f' j/s], \prod_{k : K} B[g' (f k) / s])$ is defined as $\lambda p.\,\lambda k.\,(h\,k)_*(p\,(g\,k))$.

Now, we need to prove that, for every type $P$, the type $\Hom(P, \prod_{s : I \amalg_K J} B)$ is the weak $\Pi$-type $\Pi_{s : I \amalg_K J} \Hom(P,B)$.
By the universal property of pullbacks, we have an equivalence between $\Hom(P, \prod_{s : I \amalg_K J} B)$ and the following type:
\[ \Hom(P, \prod_{i : I} B[g' i/s]) \times_{\Hom(P, \prod_{k : K} B[g' (f k) / s])} \Hom(P, \prod_{j : J} B[f' j/s]). \]
By the universal property of products, this type is equivalent to the following one:
\[ (\prod_{i : I} \Hom(P, B[g' i/s])) \times_{(\prod_{k : K} \Hom(P, B[g' (f k) / s]))} (\prod_{j : J} \Hom(P, B[f' j/s])). \]
The standard argument about pushouts shows that this type is equivalent to the weak $\Pi$-type $\Pi_{s : I \amalg_K J} \Hom(P,B)$.
\end{proof}

The following corollary shows how to compute tensoring by different type-theoretic constructions:

\begin{cor}
We have the following canonical equivalences:
\begin{align*}
\bot \cdot X & \simeq 0 \\
(I \amalg_K J) \cdot X & \simeq I \cdot X \amalg_{K \cdot X} J \cdot X \\
\top \cdot X & \simeq X \\
(\sum_{i : I} J) \cdot X & \simeq \coprod_{i : I} J \cdot X \\
(\Sigma I) \cdot 1 & \simeq \Sigma (I \cdot 1) \\
S^n \cdot 1 & \simeq S^n
\end{align*}
\end{cor}
\begin{proof}
The first four equivalences follow from previous propositions.
The equivalence for suspension follows from previous equations since suspension is defined in terms of pushouts and terminal types.
The last equivalence follows from previous since spheres are defined in terms of suspensions, coproducts, and terminal types.
\end{proof}

\begin{example}
Let $S^1$ be the pushout of $1 \amalg 1$ in the base theory, that is a higher inductive type with two point constructors $N,S : S^1$ and two path constructors $L,R : \Id_{S^1}(N,S)$.
Then the product of a family $\Gamma, x : S^1 \mid \cdot \vdash B$ is the equalizer of the maps $L_*(-),R_*(-) : B[N/x] \to B[S/x]$.
\end{example}

\section{Indexed dependent type theories}
\label{sec:dependent}

In this section, we define the dependent version of indexed type theories.

\subsection{Basic rules}

\emph{Indexed dependent type theories} have four kinds of judgments:
\[ \Gamma \vdash A \type \qquad \Gamma \vdash a : A \qquad \Gamma \mid \Delta \vdash B \ob \qquad \Gamma \mid \Delta \vdash b : B \]

In each of these judgments, $\Delta$ is an indexed context, that is a sequence of the form $y_1 : B_1, \ldots y_k : B_k$, where $B_1$, \ldots $B_k$ are indexed types and $y_1$, \ldots $y_k$ are pairwise distinct variables.
The base theory has the same rules as the base theory in indexed unary type theories.
The indexed theory has the following rules:
\begin{center}
\AxiomC{}
\UnaryInfC{$\Gamma \mid x_1 : A_1, \ldots x_n : A_n \vdash x_i : A_i$}
\DisplayProof
\end{center}

\begin{center}
\def\extraVskip{1pt}
\Axiom$\fCenter \Gamma \mid \Delta \vdash b_1 : B_1$
\noLine
\UnaryInf$\fCenter \ldots$
\noLine
\UnaryInf$\fCenter \Gamma \mid \Delta \vdash b_k : B_k[b_1/y_1, \ldots b_{k-1}/y_{k-1}]$
\Axiom$\fCenter \Gamma \mid \Delta, y_1 : B_1, \ldots y_k : B_k \vdash C \type$
\def\extraVskip{2pt}
\BinaryInfC{$\Gamma \mid \Delta \vdash C[b_1/y_1, \ldots b_k/y_k] \type$}
\DisplayProof
\end{center}

\begin{center}
\def\extraVskip{1pt}
\Axiom$\fCenter \Gamma \mid \Delta \vdash b_1 : B_1$
\noLine
\UnaryInf$\fCenter \ldots$
\noLine
\UnaryInf$\fCenter \Gamma \mid \Delta \vdash b_k : B_k[b_1/y_1, \ldots b_{k-1}/y_{k-1}]$
\Axiom$\fCenter \Gamma \mid \Delta, y_1 : B_1, \ldots y_k : B_k \vdash c : C$
\def\extraVskip{2pt}
\BinaryInfC{$\Gamma \mid \Delta \vdash c[b_1/y_1, \ldots b_k/y_k] : C[b_1/y_1, \ldots b_k/y_k]$}
\DisplayProof
\end{center}

We also have the usual equations for substitution:
\begin{align*}
y_i[b_1/y_1, \ldots b_k/y_k] & = b_i \\
c[y_1/y_1, \ldots y_k/y_k] & = c \\
d[c_1/z_1, \ldots c_n/z_n][b_1/y_1, \ldots b_k/y_k] & = d[c_1'/z_1, \ldots c_n'/z_n],
\end{align*}
where $c_i' = c_i[b_1/y_1, \ldots b_k/y_k]$.

We often assume that, for every construction $\sigma(\overline{z_1}.\,c_1, \ldots \overline{z_n}.\,c_n)$ in the indexed theory, the following equation holds whenever variables $\overline{z_1}$, \ldots $\overline{z_n}$ are not free in $b_1$, \ldots $b_k$:
\[ \sigma(\ldots, \overline{z_i}.\,c_i, \ldots)[b_1/y_1, \ldots b_k/y_k] = \sigma(\ldots, \overline{z_i}.\,c_i[b_1/y_1, \ldots b_k/y_k], \ldots) \]
We also have the weakening operation which is left implicit as usual.
We will call a theory or a construction \emph{unstable} if this equation does not hold.

We can also substitute base terms into indexed types and terms:
\begin{center}
\def\extraVskip{1pt}
\Axiom$\fCenter \Gamma \vdash b_1 : B_1$
\noLine
\UnaryInf$\fCenter \ldots$
\noLine
\UnaryInf$\fCenter \Gamma \vdash b_k : B_k[b_1/y_1, \ldots b_{k-1}/y_{k-1}]$
\Axiom$\fCenter \Gamma, y_1 : B_1, \ldots y_k : B_k \mid \Delta \vdash C \ob$
\def\extraVskip{2pt}
\BinaryInfC{$\Gamma \mid \Delta[b_1/y_1, \ldots b_k/y_k] \vdash C[b_1/y_1, \ldots b_k/y_k] \ob$}
\DisplayProof
\end{center}

\begin{center}
\def\extraVskip{1pt}
\Axiom$\fCenter \Gamma \vdash b_1 : B_1$
\noLine
\UnaryInf$\fCenter \ldots$
\noLine
\UnaryInf$\fCenter \Gamma \vdash b_k : B_k[b_1/y_1, \ldots b_{k-1}/y_{k-1}]$
\Axiom$\fCenter \Gamma, y_1 : B_1, \ldots y_k : B_k \mid \Delta \vdash c : C$
\def\extraVskip{2pt}
\BinaryInfC{$\Gamma \mid \Delta[b_1/y_1, \ldots b_k/y_k] \vdash c[b_1/y_1, \ldots b_k/y_k] : C[b_1/y_1, \ldots b_k/y_k]$}
\DisplayProof
\end{center}

These operations satisfy the following equations for all base terms $b_1$, \ldots $b_k$ and indexed terms $c_1$, \ldots $c_n$:
\begin{align*}
x[b_1/y_1, \ldots b_k/y_k] & = x \\
d[c_1/z_1, \ldots c_n/z_n][b_1/y_1, \ldots b_k/y_k] & = d[c_1'/z_1, \ldots c_n'/z_n],
\end{align*}
where $c_i' = c_i[b_1/y_1, \ldots b_k/y_k]$.

These operations satisfy the following equations for all base terms $b_1$, \ldots $b_k$, $c_1$, \ldots $c_n$:
\begin{align*}
c[y_1/y_1, \ldots y_k/y_k] & = c \\
d[c_1/z_1, \ldots c_n/z_n][b_1/y_1, \ldots b_k/y_k] & = d[c_1'/z_1, \ldots c_n'/z_n],
\end{align*}
where $c_i' = c_i[b_1/y_1, \ldots b_k/y_k]$.

Since the second level of an indexed dependent type theory is also a dependent type theory,
we can add standard type-theoretic construction to it.
When we add such a construction, we always assume that it is defined in every base context.

Indexed dependent type theories have the same rules as indexed unary type theories.
This means that indexed unary type theories can be interpreted in indexed dependent type theories.
This implies that every model of an indexed dependent type theory is a model of corresponding unary theory
(that is, there is a forgetful functor from the category of models of an indexed dependent theory to the category of models of an indexed unary theory).
Every model of an ordinary dependent type theory (as defined in \cite{alg-tt}) is a model of an indexed dependent type theory.
This follows from the fact that indexed type theories can be interpreted in ordinary dependent type theories.

Judgments $\Gamma \mid \Delta \vdash A \ob$ and $\Gamma \mid \Delta \vdash a : A$ are interpreted as $\Gamma, \Delta \vdash A \type$ and $\Gamma, \Delta \vdash a : A$, respectively.
All the rules of indexed dependent type theories correspond to some rules of ordinary dependent type theories.
This interpretation will be called \emph{the canonical indexing of a dependent type theory over itself}.
It is analogous to the canonical indexing of a Cartesian category over itself.

\subsection{Dependent $\Hom$-types}

Since every indexed dependent type theory is an indexed unary type theory, the extensions that we discussed in the context of unary theories also applies to dependent versions.
Note that these constructions apply only to closed indexed types.
Sometimes we can extend the notion, so that it applies to indexed types in a non-empty context.

For example, there is a notion of locally small indexed dependent type theory.
We can also define the following dependent version of $\Hom$-types:
\begin{center}
\AxiomC{$\Gamma \mid \Delta \vdash B \ob$}
\UnaryInfC{$\Gamma \vdash \Hom(\Delta.B) \type$}
\DisplayProof
\qquad
\AxiomC{$\Gamma \mid \Delta \vdash b : B$}
\UnaryInfC{$\Gamma \vdash \lambda \Delta.\,b : \Hom(\Delta.B)$}
\DisplayProof
\end{center}
\medskip

\begin{center}
\AxiomC{$\Gamma \vdash f : \Hom(\Delta.B)$}
\AxiomC{$\Gamma \mid E \vdash a_1 : A_1\ \ldots\ \Gamma \mid E \vdash a_k : A_k[a_1/x_1, \ldots a_{k-1}/x_{k-1}]$}
\BinaryInfC{$\Gamma \mid E \vdash f\,a_1\,\ldots\,a_k : B[a_1/x_1, \ldots a_k/x_k]$}
\DisplayProof
\end{center}
where $\Delta = x_1 : A_1, \ldots x_k : A_k$.

\begin{align*}
(\lambda x_1 \ldots x_k.\,b)\,a_1\,\ldots\,a_k & = b[a_1/x_1, \ldots a_k/x_k] \\
\lambda x_1 \ldots x_k.\,f\,x_1\,\ldots\,x_k & = f
\end{align*}
If $\Delta$ is empty, then we will write $\Hom(\Delta.B)$ as $\Hom(B)$, abstraction as $\lambda(b)$, and the application operation as $f\,()$.

If we have such dependent $\Hom$-types, then we can define the following operation:
\begin{center}
\def\extraVskip{1pt}
\Axiom$\fCenter \Gamma \vdash a : A$
\noLine
\UnaryInf$\fCenter \Gamma \vdash a' : A$
\noLine
\UnaryInf$\fCenter \Gamma \vdash t : \Id_A(a,a')$
\Axiom$\fCenter \Gamma, x : A, p : \Id_A(a,x), \Delta \mid E \vdash D \ob$
\noLine
\UnaryInf$\fCenter \Gamma, \Delta[a/x,\refl(a)/p] \mid E[a/x,\refl(a)/p] \vdash d : D[a/x,\refl(a)/p]$
\def\extraVskip{2pt}
\BinaryInfC{$\Gamma, \Delta[a'/x,t/p] \mid E[a'/x,t/p] \vdash J(a, x p \Delta E.\,D, \Delta E.\,d, a', t) : D[a'/x,t/p]$}
\DisplayProof
\end{center}

\[ J(a, x p \Delta E.\,D, \Delta E.\,d, a, \refl(a)) = d \]

Indeed, we define $J(a, x p \Delta E.\,D, \Delta E.\,d, a', t)$ as follows:
\[ J(a, x p \Delta.\,\Hom(E.D), \Delta.\,\lambda E.\,d, a', t)\,y_1\,\ldots\,y_k, \]
where $E = y_1 : B_1, \ldots y_k : B_k$.
If we do not assume the existence dependent $\Hom$-types, then we need to add this operation explicitly.
If the indexed theory has identity types, then we can define the following operation:
\begin{center}
\AxiomC{$\Gamma \vdash p : \Id_{\Hom(A,B)}(f,g)$}
\AxiomC{$\Gamma \mid \Delta \vdash a : A$}
\BinaryInfC{$\Gamma \mid \Delta \vdash \fs{hap}(p,a) : \Id_B(f\,a,g\,a)$}
\DisplayProof
\end{center}
It is defined as follows:
\[ \fs{hap}(p,a) = \pmap(f.\,f\,a,p). \]

We will say that identity types are \emph{extensional} if $\fs{hap}$ is an equivalence.
More precisely, the theory of extensional identity types has the following constructions:
\begin{center}
\AxiomC{$\Gamma \mid x : A \vdash p : \Id_{B}(f\,x,g\,x)$}
\UnaryInfC{$\Gamma \vdash \fs{Idext}(x.\,p) : \Id_{\Hom(A,B)}(f,g)$}
\DisplayProof
\end{center}
\medskip

\begin{center}
\AxiomC{$\Gamma \mid x : A \vdash p : \Id_{B}(f\,x,g\,x)$}
\AxiomC{$\Gamma \mid \Delta \vdash a : A$}
\BinaryInfC{$\Gamma \mid \Delta \vdash \fs{hap_h}(x.\,p,a) : \Id_B(\fs{hap}(\fs{Idext}(x.\,p),a),p[a/x])$}
\DisplayProof
\end{center}
\medskip
The standard argument implies that we also have the following homotopy:
\begin{center}
\AxiomC{$\Gamma \vdash p : \Id_{\Hom(A,B)}(f,g)$}
\UnaryInfC{$\Gamma \mid \Delta \vdash \fs{hap_h'}(p) : \Id(\fs{Idext}(x.\,\fs{hap}(p,x)),p)$}
\DisplayProof
\end{center}
If we have dependent $\Hom$-types, then identity types are extensional if and only if the following function is an equivalence:
\[ \lambda p.\,\lambda x.\,\fs{hap}(p,x) : \Id_{\Hom(A,B)}(f,g) \to \Hom(A, x.\,\Id_B(f\,x,g\,x)). \]

\begin{example}
The canonical indexing of a dependent type theory over itself is locally small if and only if it has non-dependent function types.
It has dependent $\Hom$-types if and only if it has $\Pi$-types.
It has extensional identity types if and only if it has identity types and the functional extensionality.
\end{example}

\subsection{Weak dependent $\Hom$-types}

If the indexed theory has extensional identity types, $\Sigma$-types, and the unit type, then we can define a weak version of dependent $\Hom$-types:
\[ \Hom(\Delta.B) = \sum_{f : \Hom(\Sigma(\Delta),\Sigma_{p : \Sigma(\Delta)} B[\pi_1(p)/x_1, \ldots \pi_k(p)/x_k])} \Id(\pi_1 \circ f, \id_{\Sigma(\Delta)}), \]
where $\Delta = x_1 : A_1, \ldots x_k : A_k$ and $\Sigma(\Delta)$ is defined inductively:
\begin{align*}
\Sigma(\cdot) & = \top \\
\Sigma(x : A, \Delta) & = \sum_{x : A} \Sigma(\Delta).
\end{align*}
The abstraction is defined as follows:
\[ \lambda x_1 \ldots x_k.\,b = (\lambda p.\,(p, b[\pi_1(p)/x_1, \ldots \pi_k(p)/x_k]), \refl(\id_{\Sigma(\Delta)})). \]
The application is defined as follows:
\[ f\,a_1\,\ldots\,a_k = \fs{hap}(\pi_2(f),(a_1, \ldots a_k))_*(\pi_2(\pi_1(f)\,(a_1, \ldots a_k))). \]
The beta rule holds judgmentally, but the eta rule holds only propositionally.
Indeed, $\lambda x_1 \ldots x_k.\,f\,x_1\,\ldots\,x_k$ equals to
\[ (\lambda p.\,(p,\fs{hap}(\pi_2(f),p')_*(\pi_2(\pi_1(f)\,p'))), \refl(\id_{\Sigma(\Delta)})), \]
where $p' = (\pi_1(p), \ldots \pi_k(p))$.
To prove that it is homotopic to $f$, we need to construct a homotopy of the following type:
\[ h : \Id(\pi_1(f), \lambda p.\,(p,\fs{hap}(\pi_2(f),p')_*(\pi_2(\pi_1(f)\,p')))) \]
such that $h * \pi_1$ is homotopic to $\pi_2(f)$.
To construct such a homotopy, we can use $\fs{Idext}$.
Then we need to define two homotopies for every $p : \Sigma(\Delta)$:
\begin{align*}
h_1 & : \Id(\pi_1(\pi_1(f)\,p),p) \\
h_2 & : \Id((h_1)_*(\pi_2(\pi_1(f)\,p)),\fs{hap}(\pi_2(f),p')_*(\pi_2(\pi_1(f)\,p')))
\end{align*}
The condition that $h * \pi_1$ is homotopic to $\pi_2(f)$ is satisfied if we put $h_1 = \fs{hap}(\pi_2(f),p)$.
Finally, to construct $h_2$, it is enough to note that $p'$ is homotopic to $p$.

Thus, the theory of dependent $\Hom$-types is a slightly stricter version of the theory of $\Hom$-types.
This is similar to the theory of $\Pi$-types being a strict version of the theory of non-dependent function types.

We can define a dependent version of $\fs{hap}$:
\begin{center}
\AxiomC{$\Gamma \vdash p : \Id_{\Hom(\Delta,B)}(f,g)$}
\def\extraVskip{1pt}
\Axiom$\fCenter \Gamma \mid E \vdash a_1 : A_1$
\noLine
\UnaryInf$\fCenter \ldots$
\noLine
\UnaryInf$\fCenter \Gamma \mid E \vdash a_k : A_k[a_1/x_1, \ldots a_{k-1}/x_{k-1}]$
\def\extraVskip{2pt}
\BinaryInfC{$\Gamma \mid E \vdash \fs{hap}(p, a_1, \ldots a_k) : \Id_{B[a_1/x_1, \ldots a_k/x_k]}(f\,a_1\,\ldots\,a_k,g\,a_1\,\ldots\,a_k)$}
\DisplayProof
\end{center}
It is defined as follows:
\[ \fs{hap}(p, a_1, \ldots a_k) = J(f, h q.\,\Id(f\,a_1\,\ldots\,a_k,h\,a_1\,\ldots\,a_k), \refl(f\,a_1\,\ldots\,a_k), g, p). \]
It is straightforward to check that if identity types are extensional, then the following function is an equivalence:
\begin{align*}
& \Id_{\Hom(x_1 \ldots x_k.B)}(f,g) \to \Hom(x_1 \ldots x_k.\,\Id_B(f\,x_1\,\ldots\,x_k,g\,x_1\,\ldots\,x_k)) \\
& \lambda p.\,\lambda x_1 \ldots x_k.\,\fs{hap}(p, x_1, \ldots x_k).
\end{align*}

\section{Locally Cartesian closed indexed theories}
\label{sec:lccc}

In this section, we will define locally Cartesian closed unary type theories and discuss the relationship between them and indexed dependent type theories with $\Pi$-types.

\subsection{Types over contexts}

Let $p_A : \Hom(A,\Delta)$ and $p_B : \Hom(B,\Delta)$ be a pair of maps with the same codomain.
We will write $\Hom_\Delta(A,B)$ for the type $\Sigma_{f : \Hom(A,B)} \Id(p_B \circ f, p_A)$.
If we think of maps $\Hom(X,\Delta)$ as types over $\Delta$, then $\Hom_\Delta(A,B)$ is the type of morphisms between such types.
We will identify elements of $\Hom_\Delta(A,B)$ with underlying morphisms $\Hom(A,B)$ and we will often omit the homotopy in $\Hom_\Delta(A,B)$ when constructing an element of this type.

If the theory has $\Sigma$-types and unit types, then, for every dependent types $\Delta \vdash B \ob$, we can define the following map:
\[ \pi_1 : \Hom(\sum_{p : \Sigma(\Delta)} B[\pi_1(p)/x_1, \ldots \pi_k(p)/x_k]), \Sigma(\Delta)), \]
where $\Delta = x_1 : A_1, \ldots x_k : A_k$.
Conversely, if the theory also has identity types, then, for every map $f : \Hom(A,B)$, we can define the following dependent type:
\[ y : B \vdash \sum_{x : A} \Id_B(f\,x,y). \]
These constructions are mutually inverse.
This implies that every dependent type in an arbitrary context is equivalent to a type in a context of size 1.

If $A$ and $B$ are dependent types in a context $\Delta$ in an indexed dependent type theory, then we will write $\Hom_\Delta(A,B)$ for $\Hom((\Delta, x : A). B)$.
This type corresponds to the previously defined type of maps over $\Delta$ through the equivalence between dependent types and types over $\Sigma(\Delta)$.

\subsection{Exponentiable maps}

Let $p_A : \Hom(A,D)$ be a morphism in an indexed unary type theory such that its pullbacks along any map exist.
The \emph{exponent over $D$} of $p_A$ and a map $p_B : \Hom(B,D)$ is a map $p : \Hom(B^A,D)$ together with a map $\fs{ev} : \Hom_D(B^A \times_D A, B)$ such that the following function is an equivalence for every indexed type $X$:
\[ \lambda f.\, \fs{ev} \circ (f \times_D A) : \Hom_D(X, B^A) \to \Hom_D(X \times_D A, B). \]
We will say that $p_A$ is \emph{exponentiable} if the exponent $B^A$ exists for all maps $p_B$.
We will say that the theory is \emph{locally Cartesian closed} if all maps are exponentiable.

Non-dependent function types in an indexed dependent type theory are defined as usual.
If $\Gamma \mid \Delta \vdash A \ob$ and $\Gamma \mid \Delta \vdash B \ob$ is a pair of indexed types, then the \emph{function type} $\Gamma \mid \Delta \vdash A \to B \ob$ is defined as follows:
\begin{center}
\AxiomC{$\Gamma \mid \Delta, x : A \vdash b : B$}
\UnaryInfC{$\Gamma \mid \Delta \vdash \lambda x.\,b : A \to B$}
\DisplayProof
\qquad
\AxiomC{$\Gamma \mid \Delta \vdash f : A \to B$}
\AxiomC{$\Gamma \mid \Delta \vdash a : A$}
\BinaryInfC{$\Gamma \mid \Delta \vdash f\,a : B$}
\DisplayProof
\end{center}

\begin{align*}
(\lambda x.\,b)\,a & = b[a/x] \\
\lambda x.\,f\,x & = f
\end{align*}
The \emph{$\Pi$-type} $\Pi_{x : A} B$ of a family $\Gamma \vdash \Delta, x : A \vdash B \ob$ is defined similarly:
\begin{center}
\AxiomC{$\Gamma \mid \Delta, x : A \vdash b : B$}
\UnaryInfC{$\Gamma \mid \Delta \vdash \lambda x.\,b : \Pi_{x : A} B$}
\DisplayProof
\qquad
\AxiomC{$\Gamma \mid \Delta \vdash f : \Pi_{x : A} B$}
\AxiomC{$\Gamma \mid \Delta \vdash a : A$}
\BinaryInfC{$\Gamma \mid \Delta \vdash f\,a : B[a/x]$}
\DisplayProof
\end{center}

\begin{align*}
(\lambda x.\,b)\,a & = b[a/x] \\
\lambda x.\,f\,x & = f
\end{align*}

The last two equations are called $\beta$ and $\eta$ rules.
\emph{Weak function types} and \emph{weak $\Pi$-types} are defined in the same way except for the $\beta$ and $\eta$ rules which hold only propositionally.
We will denote them by $\beta(x.b,a)$ and $\eta(f)$:
\begin{align*}
\beta(x.b, a) & : \Id((\lambda x.\,b)\,a, b[a/x]) \\
\eta(f) & : \Id(\lambda x.\,f\,x, f)
\end{align*}

\emph{Functional extensionality} for (weak) $\Pi$-types is defined as usual:
\begin{center}
\AxiomC{$\Gamma \mid \Delta, x : A \vdash p : \Id_B(f\,x,g\,x)$}
\UnaryInfC{$\Gamma \mid \Delta \vdash \fs{funext}(x.p) : \Id_{\Pi_{x : A} B}(f,g)$}
\DisplayProof
\end{center}
\medskip

\begin{center}
\AxiomC{$\Gamma \mid \Delta \vdash p : \Id_{\Pi_{x : A} B}(f,g)$}
\UnaryInfC{$\Gamma \mid \Delta \vdash \fs{funext_h}(p) : \Id(\fs{funext}(x.\pmap(f.\,f\,x,p)), p)$}
\DisplayProof
\end{center}
\medskip

\begin{center}
\AxiomC{$\Gamma \mid \Delta, x : A \vdash p : \Id_B(f\,x,g\,x)$}
\AxiomC{$\Gamma \mid \Delta \vdash a : A$}
\BinaryInfC{$\Gamma \mid \Delta \vdash \fs{funext_h'}(x.p,a) : \Id(\pmap(f.\,f\,a,\fs{funext}(x.p)), p[a/x])$}
\DisplayProof
\end{center}

\begin{lem}[dep-exp]
Let $\Delta = y_1 : A_1, \ldots y_k : A_k$ be a context and let $\Gamma \mid \Delta \vdash \Pi_{x : A} B \ob$ be a $\Pi$-type in this context.
Then the following function is an equivalence:
\[ \lambda f.\,\lambda \overline{y} x.\,f\,\overline{y}\,x : \Hom(\Delta. \Pi_{x : A} B) \to \Hom((\Delta, x : A). B). \]
\end{lem}
\begin{proof}
The inverse of this function is given by $\lambda g.\,\lambda \overline{y}.\lambda x.\,g\,\overline{y}\,x$:
\begin{align*}
(\lambda f.\,\lambda \overline{y} x.\,f\,\overline{y}\,x) \circ (\lambda g.\,\lambda \overline{y}.\lambda x.\,g\,\overline{y}\,x) & = \\
\lambda g.\,\lambda \overline{y} x.\,(\lambda \overline{y}.\lambda x.\,g\,\overline{y}\,x)\,\overline{y}\,x & = \\
\lambda g.\,\lambda \overline{y} x.\,(\lambda x.\,g\,\overline{y}\,x)\,x & \sim \\
\lambda g.\,\lambda \overline{y} x.\,g\,\overline{y}\,x & \sim \\
\lambda g.g &
\end{align*}
\begin{align*}
(\lambda g.\,\lambda \overline{y}.\lambda x.\,g\,\overline{y}\,x) \circ (\lambda f.\,\lambda \overline{y} x.\,f\,\overline{y}\,x) & = \\
\lambda f.\,\lambda \overline{y}.\lambda x.\,(\lambda \overline{y} x.\,f\,\overline{y}\,x)\,\overline{y}\,x & = \\
\lambda f.\,\lambda \overline{y}.\lambda x.\,f\,\overline{y}\,x & \sim \\
\lambda f.\,\lambda \overline{y}.\,f\,\overline{y} & \sim \\
\lambda f.f &
\end{align*}
\end{proof}

\begin{prop}
Functional extensionality holds for all $\Pi$-types that exist in an indexed dependent type theory with extensional identity types.
\end{prop}
\begin{proof}
Let $\Delta = y_1 : A_1, \ldots y_k : A_k$ be a context and let $\Gamma \mid \Delta \vdash f' : \Pi_{x : A} B$ and $\Gamma \mid \Delta \vdash g' : \Pi_{x : A} B$ be a pair of functions in this context.
Consider the following square:
\[ \xymatrix{ \Id_{\Hom(\Delta. \Pi_{x : A} B)}(f',g') \ar[r] \ar[d] & \Id_{\Hom((\Delta, x : A). B)}(\lambda \overline{y} x.\,f'\,\overline{y}\,x, \lambda \overline{y} x.\,g'\,\overline{y}\,x) \ar[d] \\
              \Hom(\Delta.\,\Id_{\Pi_{x : A} B}(f'\,\overline{y}, g'\,\overline{y})) \ar[r] & \Hom((\Delta, x : A).\,\Id_B(f'\,\overline{y}\,x, g'\,\overline{y}\,x))
            } \]
The bottom map is defined as $\lambda s.\,\lambda \overline{y} x.\,\pmap(f.\,f\,x,s\,\overline{y})$.
The left and right maps are defined as $\fs{hap}$.
These functions are equivalences since identity maps are extensional.
The top map is defined as $\lambda p.\,\pmap(f.\,\lambda \overline{y} x.\,f\,\overline{y}\,x,p)$.
\Rlem{dep-exp} implies that this function is also an equivalence.
Since the functions appearing in this square preserve $\refl$, path induction implies that it commutes.
Thus, the bottom map is an equivalence.
Let $r(f',g')$ be its inverse.

If $\Gamma \mid \Delta, x : A \vdash p : \Id(f\,x,g\,x)$, then we define $\fs{funext}(x.p)$ as follows:
\[ \Gamma \mid \Delta \vdash r(\lambda \overline{y}.f,\lambda \overline{y}.g)\,(\lambda \overline{y} x.p)\,\overline{y} : \Id_{\Pi_{x : A} B}(f,g). \]
It is easy to define $\fs{funext_h}$ and $\fs{funext_h'}$ using the fact that $r$ is an inverse of $\lambda s.\,\lambda \overline{y} x.\,\pmap(f.\,f\,x,s\,\overline{y})$.
\end{proof}

\begin{prop}[lccc]
Let $\Gamma \mid \Delta \vdash A \ob$ be a type in an indexed dependent type theory with $\Sigma$-types, unit types, and extensional identity types.
Then the following conditions are equivalent:
\begin{enumerate}
\item \label{it:exp-dep} Weak unstable $\Pi$-types $\Pi_{x : A} B$ exist for all dependent types $B$.
\item \label{it:exp-non-dep} Weak unstable function types $A \to B$ exist for all types $B$.
\item \label{it:exp-exp} $A$ is exponentiable.
\end{enumerate}
\end{prop}
\begin{proof}
\eqref{it:exp-dep} $\Rightarrow$ \eqref{it:exp-non-dep} Obvious.

\eqref{it:exp-non-dep} $\Rightarrow$ \eqref{it:exp-dep} 
We define $\Pi_{x : A} B$ as follows:
\[ \sum_{f : A \to \Sigma_{x : A} B} \Id_{A \to A}(\lambda x.\,\pi_1\,(f\,x), \lambda x.\,x). \]
If $\Gamma \mid \Delta, x : A \vdash b : B$, then we define $\lambda x.b$ as follows:
\[ (\lambda x.(x,b), \fs{funext}(x. h(b,x))), \]
where $h(b,x) = \beta(x.\,\pi_1\,((\lambda x.(x,b))\,x),x) \ct \pmap(p.\,\pi_1(p), \beta(x.(x,b),x)) \ct \sym{\beta(x.x,x)}$.
We will denote this pair by $\lambda^d x.b$ to distinguish it from the non-dependent $\lambda$.
If $\Gamma \mid \Delta \vdash p : \Pi_{x : A} B$ and $\Gamma \mid \Delta \vdash a : A$, then we define $p\,a$ as follows:
\[ h'(p,a)_*(\pi_2(\pi_1(p)\,a)), \]
where $h'(p,a) = \sym{\beta(x.\,\pi_1\,(\pi_1(p)\,x),a)} \ct \pmap(f.\,f\,a,\pi_2(p)) \ct \beta(x.x,a)$.

Let us prove that $(\lambda^d x.b)\,a = h'(\lambda^d x.b,a)_*((\lambda x.(x,b))\,a)$ is homotopic to $b[a/x]$.
First, note that we have the following homotopy:
\[ \beta(x.(x,b),a) : \Id_{\Sigma_{x : A} B}((\lambda x.(x,b))\,a,(a,b[a/x])). \]
A standard argument implies that $\pmap(p.\,\pi_1(p),\beta(x.(x,b),a))_*((\lambda x.(x,b))\,a)$ is homotopic to $b[a/x]$.
Thus, we just need to prove that $\pmap(p.\,\pi_1(p),\beta(x.(x,b),a))$ is homotopic to $h'(\lambda^d x.b,a)$:
\begin{align*}
h'(\lambda^d x.b,a) & = \\
\sym{\beta(x.\,\pi_1\,((\lambda x.(x,b))\,x),a)} \ct \pmap(f.\,f\,a,\fs{funext}(x. h(b,x))) \ct \beta(x.x,a) & \sim \\
\sym{\beta(x.\,\pi_1\,((\lambda x.(x,b))\,x),a)} \ct h(b,a) \ct \beta(x.x,a) & \sim \\
\pmap(p.\,\pi_1(p), \beta(x.(x,b),a)) & .
\end{align*}
The first homotopy follows from $\fs{funext_h'}$ and the second one from the definition of $h(b,a)$.

Now, let $(f,q)$ be a term of type $\Pi_{x : A} B$.
Let us prove that $\lambda^d x.\,(f,q)\,x$ is homotopic to $(f,q)$.
Let $b = h'((f,q),x)_*(\pi_2(f\,x))$.
Then
\[ \lambda^d x.\,(f,q)\,x = (\lambda x.(x,b),\fs{funext}(x.h(b,x))). \]
By $\Sigma$-extensionality, it is enough to construct a homotopy $t$ between $f$ and $\lambda x.(x,b)$ such that $t_*(q)$ is homotopic to $\fs{funext}(x.h(b,x))$.
We can define $t$ as follows:
\[ \fs{funext}(x. \Sigma \fs{ext}(h'((f,q),x),\refl) \ct \sym{\beta(x.(x,b),x)}). \]
Note that $t_*(q)$ is homotopic $\sym{\fs{funext}(x.s)} \ct q$, where
\[ s = \beta(x.\,\pi_1(f\,x),x) \ct \pmap(p.\,\pi_1(p\,x),t) \ct \sym{\beta(x.\,\pi_1((\lambda x.(x,b))\,x),x)}. \]
This is easy to prove by the path induction on $t$.
Thus, we just need to prove that $q$ is homotopic to $\fs{funext}(x.\,s \ct h(b,x))$.
By $\fs{funext_h}$, it is enough to prove that $\pmap(f.\,f\,x,q)$ is homotopic to $s \ct h(b,x)$.

We have the following sequence of homotopies:
\begin{align*}
\pmap(p.\,\pi_1(p\,x),t) & \sim \\
\pmap(p.\,\pi_1(p),\pmap(f.\,f\,x,t)) & = \\
\pmap(p.\,\pi_1(p),\pmap(f.\,f\,x, \fs{funext}(x. \Sigma \fs{ext}(h'((f,q),x),\refl) \ct \sym{\beta(x.(x,b),x)}))) & \sim \\
\pmap(p.\,\pi_1(p), \Sigma \fs{ext}(h'((f,q),x),\refl) \ct \sym{\beta(x.(x,b),x)}) & \sim \\
h'((f,q),x) \ct \sym{\pmap(p.\,\pi_1(p), \beta(x.(x,b),x))} & .
\end{align*}
This implies the existence of the required homotopy:
\begin{align*}
s \ct h(b,x) & \sim \\
\beta(x.\,\pi_1(f\,x),x) \ct \pmap(p.\,\pi_1(p\,x),t) \ct \pmap(p.\,\pi_1(p), \beta(x.(x,b),x)) \ct \sym{\beta(x.x,x)} & \sim \\
\beta(x.\,\pi_1(f\,x),x) \ct h'((f,q),x) \ct \sym{\beta(x.x,x)} & \sim \\
\pmap(f.\,f\,x,q) & .
\end{align*}
The first homotopy easily follows from the definitions of $s$ and $h(b,x)$.
The second homotopy follows from the sequence above.
The last homotopy follows from the definition of $h'((f,q),x)$.

\eqref{it:exp-non-dep} $\Leftrightarrow$ \eqref{it:exp-exp}.
A map $\fs{ev} : \Hom_D(B^A \times_D A, B)$ corresponds to a map $\fs{ev} : \Hom((d : D, f : B^A, a : A). B)$.
Then $\lambda f.\,\fs{ev} \circ (f \times_D A)$ is an equivalence if and only if the following function is an equivalence:
\[ \lambda f. \lambda d x a.\,\fs{ev}\,d\,(f\,d\,x)\,a : \Hom((d : D, x : X).\,B^A) \to \Hom((d : D, x : X, a : A). B). \]
If the function type $\Gamma \mid d : D \vdash A \to B \ob$ exists, then we define $B^A$ as $A \to B$ and $\fs{ev}$ as $\lambda d f a.\,f\,a$.
The function $\lambda f. \lambda d x a.\,\fs{ev}\,d\,(f\,d\,x)\,a = \lambda f.\,\lambda d x a.\,f\,d\,x\,a$ is an equivalence by \rlem{dep-exp}.

Conversely, suppose that the exponent $B^A$ over $D$ exists for all dependent types $\Gamma \mid d : D \vdash A \ob$ and $\Gamma \mid d : D \vdash B \ob$.
Let $\Gamma \mid \Delta \vdash A \ob$ and $\Gamma \mid \Delta \vdash B \ob$ be a pair of dependent types.
Since the theory has $\Sigma$-types and unit types, we may assume that $\Delta$ consists of a single type, that is $\Delta = (d : D)$.
Then we can define $A \to B$ as $B^A$.
If $\Gamma \mid d : D \vdash f : A \to B$ and $\Gamma \mid d : D \vdash a : A$, then we define $\Gamma \mid d : D \vdash f\,a : B$ as $\fs{ev}\,d\,f\,a$.
Let $r$ be the inverse of $\lambda f. \lambda d x a.\,\fs{ev}\,d\,(f\,d\,x)\,a$ for $X = \top$.
If $\Gamma \mid d : D, x : A \vdash b : B$, then we define $\Gamma \mid d : D \vdash \lambda x.b : A \to B$ as $r\,(\lambda d x a.\,b[a/x])\,d\,\fs{tt}$, where $\fs{tt}$ is the unique element of $\top$.
Then $\beta$ and $\eta$ equivalences follow from the fact that $r$ is the inverse of $\lambda f. \lambda d x a.\,\fs{ev}\,d\,(f\,d\,x)\,a$.
\end{proof}

\section{Limits and colimits in dependent theories}
\label{sec:colimits-dep}

In this section, we discuss the concepts of limits and colimits in indexed dependent type theories.

\subsection{Finite limits}

Clearly, an indexed dependent type theory has a terminal type if and only if it has a closed contractible type.
For example, this is true when it has the unit type.
If an indexed dependent type theory has extensional identity types and $\Sigma$-types, then it has pullbacks.
Indeed, we can define a pullback of maps $f : \Hom(A,C)$ and $g : \Hom(B,C)$ as $A \times_C B = \Sigma_{x : A} \Sigma_{y : B} \Id_C(f\,x,g\,y)$
with the obvious projections $\pi_1 : \Hom(A \times_C B, A)$, $\pi_2 : \Hom(A \times_C B, B)$ and the obvious homotopy between $f \circ \pi_1$ and $g \circ \pi_2$: namely, $\fs{Idext}(p.\,\pi_3(p))$.
We need to show that the following map is an equivalence:
\begin{align*}
& \Hom(P, A \times_C B) \to \sum_{F : \Hom(P,A)} \sum_{G : \Hom(P,B)} \Id(f \circ F, g \circ G) \\
& \lambda s.\,(\pi_1 \circ s, \pi_2 \circ s, s * \fs{Idext}(p.\,\pi_3(p))).
\end{align*}
The inverse of this map is defined as follows:
\[ \lambda t.\,\lambda p.\,(\pi_1(t)\,p, \pi_2(t)\,p, \fs{hap}(\pi_3(t),p)). \]
It is easy to see that these functions are inverse of each other using the fact that $s * \fs{Idext}(p.\,\pi_3(p)) = \fs{Idext}(p.\,\pi_3(s\,p))$.

\begin{remark}
The existence of $\Sigma$-types also implies the existence of products.
Thus, equalizers also can be constructed from $\Sigma$-types.
\end{remark}

The following proposition shows that $\Hom(x_1 \ldots x_n.\,-)$ commutes with $\Sigma$-types and identity types:
\begin{prop}
If an indexed dependent type theory has extensional identity types and $\Sigma$-types, then it has the following canonical equivalences:
\begin{align*}
\Hom(x_1 \ldots x_n.\,\Id_A(a,a')) & \simeq \Id_{\Hom(x_1 \ldots x_n. A)}(\lambda x_1 \ldots x_n.\,a, \lambda x_1 \ldots x_n.\,a') \\
\Hom(x_1 \ldots x_n.\,\sum_{y : A} B) & \simeq \sum_{f : \Hom(x_1 \ldots x_n. A)} \Hom(x_1 \ldots x_n.\,B[f\,x_1\,\ldots\,x_n / y]).
\end{align*}
\end{prop}
\begin{proof}
The first equivalence is simply the extensionality for identity types.
The second equivalence is defined as follows:
\begin{align*}
& \lambda g.\,(\lambda x_1 \ldots x_n.\,\pi_1(g\,x_1\,\ldots\,x_n), \lambda x_1 \ldots x_n.\,\pi_2(g\,x_1\,\ldots\,x_n)) \\
& \lambda p.\,\lambda x_1 \ldots x_n.\,(\pi_1(p)\,x_1\,\ldots\,x_n, \pi_2(p)\,x_1\,\ldots\,x_n)
\end{align*}
It is easy to see that these functions are mutually inverse.
\end{proof}

\subsection{Initial types}

Finite colimits can be defined in an indexed dependent type theory as higher inductive types.
We need to be careful since the usual definition gives us finite colimits which are stable under pullbacks.
If we want a definition of general finite colimits, then we need to modify the rules for higher inductive types slightly.

\begin{prop}
Let $0$ be a type in an indexed dependent type theory has extensional identity types and $\Sigma$-types.
Then $0$ is initial if and only if the theory has the following rule:
\begin{center}
\AxiomC{$\Gamma \mid x : 0 \vdash D \type$}
\AxiomC{$\Gamma \mid \Delta \vdash a : 0$}
\BinaryInfC{$\Gamma \mid \Delta \vdash 0\text{-}\fs{elim}(x.D,a) : D[a/x]$}
\DisplayProof
\end{center}
\end{prop}
\begin{proof}
If $0$ is initial, then $\Gamma \mid x : 0 \vdash \fs{hap}(h,x)_*(\pi_2(f\,x)) : D$, where $f$ is the unique map $\Hom(0,\Sigma_{x : 0} D)$ and $h$ is the unique homotopy between $\pi_1 \circ f$ and $\id_0$.
Now, we can define $0\text{-}\fs{elim}(x.D,a)$ as $\fs{hap}(h,x)_*(\pi_2(f\,x))[a/x]$.

Conversely, suppose that the theory has the the rule $0\text{-}\fs{elim}$.
Then, for every type $D$, we can define a map $\lambda x.\,0\text{-}\fs{elim}(y.D,x) : \Hom(0,D)$.
For all maps $f,g : \Hom(0,D)$, we can define a homotopy between them as follows:
\[ \fs{Idext}(x.\,0\text{-}\fs{elim}(y.\,\Id_D(f\,y,g\,y),x)) : \Id(f,g). \]
\end{proof}

We will say that the initial type $0$ in an indexed unary type theory is \emph{stable under pullbacks} if, for every type $B$ such that the product $B \times 0$ exists, this product is initial.

\begin{prop}
Let $0$ be a type in an indexed dependent type theory with extensional identity types and $\Sigma$-types.
Then the following conditions are equivalent:
\begin{enumerate}
\item \label{it:init-first} The theory has the following unstable rule:
\begin{center}
\AxiomC{$\Gamma \mid \Delta, x : 0, E \vdash D \type$}
\AxiomC{$\Gamma \mid \Delta \vdash a : 0$}
\BinaryInfC{$\Gamma \mid \Delta, E[a/x] \vdash 0\text{-}\fs{elim''}(x E.D,a) : D[a/x]$}
\DisplayProof
\end{center}
\item \label{it:init-second} The theory has the following unstable rule:
\begin{center}
\AxiomC{$\Gamma \mid \Delta \vdash D \type$}
\AxiomC{$\Gamma \mid \Delta \vdash a : 0$}
\BinaryInfC{$\Gamma \mid \Delta \vdash 0\text{-}\fs{elim'}(D,a) : D$}
\DisplayProof
\end{center}
\item \label{it:init-third} For every map $f : \Hom(B,0)$, the type $B$ is initial.
\item \label{it:init-fourth} The type $0$ is a stable under pullbacks initial type.
\end{enumerate}
\end{prop}
\begin{proof}
\eqref{it:init-first} $\Rightarrow$ \eqref{it:init-second}
This is obvious since $0\text{-}\fs{elim'}$ is a special case of $0\text{-}\fs{elim''}$.

\eqref{it:init-second} $\Rightarrow$ \eqref{it:init-third}
For every type $C$, we can define a map $\lambda x.\,0\text{-}\fs{elim'}(C,f\,x) : \Hom(B,C)$.
Let $g_1,g_2 : \Hom(B,C)$ be a pair of maps.
Then we can define a homotopy between them as follows:
\[ \fs{Idext}(x.\,0\text{-}\fs{elim'}(\Id_C(g_1\,x,g_2\,x),f\,x)) : \Id(g_1,g_2). \]

\eqref{it:init-third} $\Rightarrow$ \eqref{it:init-fourth}
Since we have the map $\id_0 : \Hom(0,0)$, the type $0$ is initial.
Moreover, since we have the projection $\pi_2 : \Hom(B \times 0, 0)$, every product $B \times 0$ is also initial.

\eqref{it:init-fourth} $\Rightarrow$ \eqref{it:init-third}
It is enough to prove that $B$ is equivalent to $B \times 0$ since the latter type is initial.
We have maps $\langle \id_B, f \rangle : \Hom(B, B \times 0)$ and $\pi_1 : B \times 0 \to B$.
It is clear that $\pi_1 \circ \langle id_B, f \rangle$ is homotopic to $\id_B$.
The map $\langle \id_B, f \rangle \circ \pi_1$ is homotopic to $\id_{B \times 0}$ since $B \times 0$ is initial.

\eqref{it:init-third} $\Rightarrow$ \eqref{it:init-first}
Let $\Delta = x_1 : A_1, \ldots x_m : A_m$ and $E = y_1 : C_1, \ldots y_n : C_n$.
Let $B = \Sigma(\Delta, x : 0, E)$.
Since we have the projection $\pi_{m+1} : \Hom(B,0)$, the type $B$ is initial.
Thus, we have a map $f : \Hom(B, \Sigma_{b : B} D')$, where
\[ D' = D[\pi_1(b)/x_1, \ldots \pi_m(b)/x_m, \pi_{m+1}(b)/x, \pi_{m+2}(b)/y_1, \ldots \pi_{m+n+1}(b)/y_n]. \]
We also have a homotopy $h : \Id(\pi_1 \circ f, \id_B)$.
Now, we can define $0\text{-}\fs{elim''}(x E.D, a)$ as $\fs{hap}(h,p)_*(\pi_2(f\,p))$, where $p = (x_1, \ldots x_m, a, y_1, \ldots y_n)$.
\end{proof}

\subsection{Dependent products}

We defined strict products in subsection~\ref{sec:products}.
We can define even stricter version of products which we call \emph{strict dependent products} in indexed dependent type theories:
\begin{center}
\AxiomC{$\Gamma, i : I \mid \Delta \vdash B \ob$}
\RightLabel{, $i \notin \mathrm{FV}(\Delta)$}
\UnaryInfC{$\Gamma \mid \Delta \vdash \prod_{i : I} B \ob$}
\DisplayProof
\qquad
\AxiomC{$\Gamma, i : I \mid \Delta \vdash b : B$}
\RightLabel{, $i \notin \mathrm{FV}(\Delta)$}
\UnaryInfC{$\Gamma \mid \Delta \vdash \lambda i.\,b : \prod_{i : I} B$}
\DisplayProof
\end{center}
\medskip

\begin{center}
\AxiomC{$\Gamma \mid \Delta \vdash f : \prod_{i : I} B$}
\AxiomC{$\Gamma \vdash j : I$}
\BinaryInfC{$\Gamma \mid \Delta \vdash f\,j : B[j/i]$}
\DisplayProof
\end{center}

\begin{align*}
(\lambda i.\,b)\,j & = b[j/i] \\
\lambda i.\,f\,i & = f
\end{align*}

\begin{example}
The canonical indexing of a dependent type theory over itself has strict dependent products if and only if it has $\Pi$-types.
\end{example}

\emph{Weak dependent products} are defined in the same way except for the last two equations which hold only propositionally.
Obviously, if a theory has strict (resp., weak) dependent products, then it also has strict (resp., weak) products.
We can think of dependent products as wide pullbacks.
Thus, to construct them from ordinary products, we also need to assume the existence of finite limits:

\begin{prop}
If a theory has $\Sigma$-types, unit types, extensional identity types, and weak extensional products, then it also has weak unstable dependent products.
\end{prop}
\begin{proof}
Let $\Gamma, i : I \mid \Delta \vdash B_i \ob$ be an indexed type.
Since we have $\Sigma$-types and unit types, we may assume that $\Delta$ consists of exactly one type, that is $\Delta = (x : A)$.
We define $\Gamma \mid x : A \vdash \prod_{i : I} B_i$ as follows:
\[ \Gamma \mid x : A \vdash \sum_{f : \prod_{i : I} \Sigma_{x : A} B_i} \Id_{A^I}(\lambda i.\,\pi_1(f\,i), \lambda i.\,x). \]
The construction of the abstraction and the application and proofs of $\beta$ and $\eta$ equivalences are the same as in \rprop{lccc}.
\end{proof}

\subsection{Dependent coproducts}

Coproducts can be defined in a more type-theoretic way.
The theory of \emph{(unstable) dependent coproducts} consists of the following unstable rules:
\begin{center}
\AxiomC{$\Gamma, i : I \mid \Delta \vdash B_i \ob$}
\RightLabel{, $i \notin \mathrm{FV}(\Delta)$}
\UnaryInfC{$\Gamma \mid \Delta \vdash \coprod_{i : I} B_i \ob$}
\DisplayProof
\qquad
\AxiomC{$\Gamma \vdash j : I$}
\UnaryInfC{$\Gamma \mid \Delta, x : B_j \vdash \fs{in}_j(x) : \coprod_{i : I} B_i$}
\DisplayProof
\end{center}
\medskip

\begin{center}
\AxiomC{$\Gamma \mid \Delta, z : \coprod_{i : I} B \vdash D \ob$}
\AxiomC{$\Gamma, i : I \mid \Delta, x : B \vdash d : D[\fs{in}_i(x)/z]$}
\BinaryInfC{$\Gamma \mid \Delta, z : \coprod_{i : I} B \vdash \coprod\text{-}\fs{elim}(z.D, i x.d) : D$}
\DisplayProof
\end{center}
\medskip

\[ \coprod\text{-}\fs{elim}(z.D, i x.d)[\fs{in}_j(x)[b/x]/z] = d[j/i,b/x] \]

The theory of \emph{weak (unstable) dependent coproducts} has the same rules except for the last equality which holds only propositionally.

\begin{example}
The canonical indexing of a dependent type theory over itself always has dependent coproducts since we always assume that the base theory has $\Sigma$-types.
\end{example}

\begin{prop}[coprod-sigma]
Dependent coproducts commute with $\Sigma$-types.
More precisely, the following map is an equivalence for every dependent type $\Gamma, i : I \mid \Delta, d : D \vdash B \ob$:
\[ \lambda z. \coprod\text{-}\fs{elim}(z. \sum_{d : D} \coprod_{i : I} B, i p. (\pi_1(p),\fs{in}_i(x)[\pi_2(p)/x])) : \Hom_\Delta(\coprod_{i : I} \sum_{d : D} B, \sum_{d : D} \coprod_{i : I} B). \]
\end{prop}
\begin{proof}
The inverse of this map is defined as follows:
\[ \lambda p. \coprod\text{-}\fs{elim}(z. \coprod_{i : I} \sum_{d : D} B, i b.\,\fs{in}_i(x)[(\pi_1(p),b)/x])[\pi_2(p)/z]. \]
It is easy to show that these maps are mutually inverse using the eliminator for coproducts.
\end{proof}

\begin{lem}[dep-coprod-coprod]
Suppose that an indexed dependent type theory has $\Sigma$-types, extensional identity types, and dependent coproducts.
Then the map $\Gamma, i : I \vdash \fs{in}_i' : \Hom(\Sigma(\Delta,B_i), \Sigma(\Delta, \coprod_{i : I} B_i))$ induced by $\fs{in}_i$ makes $\Sigma(\Delta, \coprod_{i : I} B_i)$ into a coproduct of the family $\Gamma, i : I \mid \cdot \vdash \Sigma(\Delta,B_i)$.
\end{lem}
\begin{proof}
Let $\Delta = x_1 : A_1, \ldots x_n : A_n$.
Then $\fs{in}_i'$ is defined as follows:
\[ \fs{in}_i' = \lambda p.(\pi_1(p), \ldots \pi_n(p), \fs{in}_i(x)[\pi_{n+1}(p)/x]). \]
Let $f$ be a map of the following form:
\[ \Gamma, i : I \vdash f : \Hom(\Sigma(\Delta,B_i),C). \]
Then we define $\Gamma \vdash [f]_{i : I} : \Hom(\Sigma(\Delta, \coprod_{i : I} B_i), C)$ as follows:
\[ [f]_{i : I} = \lambda p.\,\coprod\text{-}\fs{elim}(z.C, i x.\,f\,(\pi_1(p), \ldots \pi_n(p), x))[\pi_{n+1}(p)/z]. \]
Let us prove that $[f]_{i : I} \circ \fs{in}_i' \sim f$:
\begin{align*}
[f]_{i : I} \circ \fs{in}_i' & = \\
\lambda p.\,\coprod\text{-}\fs{elim}(z.C, i x.\,f\,(\pi_1(p), \ldots \pi_n(p), x))[\fs{in}_i(x)[\pi_{n+1}(p)/x]/z] & \sim \\
\lambda p.\,f\,(\pi_1(p), \ldots \pi_{n+1}(p)) & \sim \\
\lambda p.\,f\,p & \sim \\
f & .
\end{align*}
Finally, for every $g : \Hom(\coprod_{i : I} B_i, C)$, we need to prove that $[g \circ \fs{in}_i']_{i : I} \sim g$.
It is enough to prove that, for all $x_1 : A_1, \ldots x_n : A_n, z : \coprod_{i : I} B_i$, there is a homotopy between $[g \circ \fs{in}_i']_{i : I}\,(x_1, \ldots x_n, z)$ and $g\,(x_1, \ldots x_n, z)$.
To do this, we can apply $\coprod\text{-}\fs{elim}$ to $z$.
Then we just need to construct a homotopy between $[g \circ \fs{in}_i']_{i : I}\,(x_1, \ldots x_n, \fs{in}_i(x))$ and $g\,(x_1, \ldots x_n, \fs{in}_i(x))$:
\begin{align*}
[g \circ \fs{in}_i']_{i : I}\,(x_1, \ldots x_n, \fs{in}_i(x)) & = \\
\lambda p.\,\coprod\text{-}\fs{elim}(z.C, i x.\,g\,(\fs{in}_i\,(x_1, \ldots x_n, x)))[\fs{in}_i(x)/z] & \sim \\
g\,(x_1, \ldots x_n, \fs{in}_i(x)) & .
\end{align*}
\end{proof}

This lemma implies that $\Hom_\Delta(\coprod_{i : I} B_i, C)$ is a weak $\Pi$-type $\Pi_{i : I} \Hom_\Delta(B_i,C)$.
We will say that dependent coproducts are \emph{extensional} if this $\Pi$-type satisfies functional extensionality.

\begin{lem}[coprod-dep-coprod]
Suppose that an indexed dependent type theory has $\Sigma$-types and identity types.
If $\coprod$ and $\fs{in}$ are unstable constructions defined above and the map $\Gamma, i : I \vdash \fs{in}_i' : \Hom(\Sigma(\Delta,B_i), \Sigma(\Delta, \coprod_{i : I} B_i))$ induced by $\fs{in}_i$ makes $\Sigma(\Delta, \coprod_{i : I} B_i)$ into an extensional coproduct of the family $\Gamma, i : I \mid \cdot \vdash \Sigma(\Delta,B_i)$,
then the eliminator is definable and this dependent coproduct is extensional.
\end{lem}
\begin{proof}
Let $\Delta = x_1 : A_1, \ldots x_n : A_n$ and let $\Gamma, i : I \mid \Delta, x : B_i \vdash d : D[\fs{in}_i(x)/z]$ be a term.
Since $\Sigma(\Delta, \coprod_{i : I} B_i)$ is a coproduct of $\Sigma(\Delta, B_i)$, we have the following map:
\begin{align*}
t & : \Hom(\Sigma(\Delta, \coprod_{i : I} B_i), \Sigma(\Delta, z : \coprod_{i : I} B_i, D)) \\
t & = [\lambda p.(x_1, \ldots x_n, \fs{in}_i(x), d)[\pi_1(p)/x_1, \ldots \pi_n(p)/x_n, \pi_{n+1}(p)/x]]_{i : I}
\end{align*}
Let $\pi_0 : \Hom(\Sigma(\Delta, z : \coprod_{i : I} B_i, D), \Sigma(\Delta, \coprod_{i : I} B_i))$ be the obvious projection map.
Let $\eta(q)$ be the following homotopy:
\[ \eta(q) : \Id(t\,q, (x_1, \ldots x_n, \fs{in}_i(x), d)[\pi_1(q)/x_1, \ldots \pi_n(q)/x_n, \pi_{n+1}(q)/x]). \]
Then $\eta(\fs{in}_i'\,p) * \pi_0$ is a homotopy between $\pi_0\,(t\,(\fs{in}_i'\,p))$ and $\fs{in}_i'\,p$.
By the universal property of coproducts, there exists (a unique) homotopy $h(q)$ between $\pi_0\,(t\,q)$ and $q$ such that $h(\fs{in}_i'\,p)$ is homotopic to $\eta(\fs{in}_i'\,p) * \pi_0$.
Now, we can define the eliminator $\coprod\text{-}\fs{elim}(z.D, i x.d)$ as $h(x_1, \ldots x_n, z)_*(\pi_{n+2}(t\,(x_1, \ldots x_n, z)))$.

We need to construct the following homotopy:
\[ h(x_1, \ldots x_n, \fs{in}_j(x)[b/x])_*(\pi_{n+2}(t\,(x_1, \ldots x_n, \fs{in}_j(x)[b/x]))) \sim d[j/i,b/x]. \]
Since $\eta(q)$ is a homotopy between $\Sigma$-types, the standard argument about such homotopies shows that we have the following homotopy:
\[ (\eta(\fs{in}_j'(b)) * \pi_0)_*(\pi_{n+2}(t\,(x_1, \ldots x_n, \fs{in}_j(x)[b/x]))) \sim d[j/i,b/x]. \]
Thus, we just need to prove that $h(x_1, \ldots x_n, \fs{in}_j(x)[b/x])$ is homotopic to $\eta(\fs{in}_j'(b)) * \pi_0$, but this is true by the definition of $h$.

The extensionality of dependent coproducts is equivalent to the extensionality of the ordinary coproduct $\Sigma(\Delta, \coprod_{i : I} B_i)$.
\end{proof}

\begin{prop}
If an indexed dependent type theory has $\Sigma$-types and extensional identity types, then it has weak extensional dependent coproducts if and only if it has extensional coproducts.
\end{prop}
\begin{proof}
If the theory has weak dependent coproducts, then it has coproducts by \rlem{dep-coprod-coprod}.
If dependent coproducts are extensional, then coproducts are also extensional since the latter is the special case of the former.
Conversely, suppose that the theory has extensional coproducts.
Let $\Gamma, i : I \mid \Delta \vdash B_i \ob$ be a dependent type.
If $\Delta$ is empty, then an extensional dependent coproduct of $B_i$ exists by \rlem{coprod-dep-coprod}.
Thus, we may assume that $\Delta$ is not empty.
Since we have $\Sigma$-types, we may also assume that it consists of a single type, that is $\Delta = (x : A)$.

We define $\coprod'_{i : I} B_i$ as $\Gamma \mid \Delta \vdash \Sigma_{(p : \coprod_{i : I} \Sigma_{x : A} B_i)} \Id_A([\pi_1]_{i : I}\,p, x) \ob$.
If $\Gamma \vdash j : I$, then we define $\fs{in}_j$ as $\Gamma \mid \Delta, y : B_j \vdash (\fs{in}_j(x,y), \refl(x)) : \coprod'_{i : I} B_i$.
Let $\fs{in}_i'$ be the following map:
\[ \lambda p. (\pi_1(p), \fs{in}_i\,p, \refl(x)) : \Hom(\Sigma(\Delta, B_i), \Sigma(\Delta, p : \coprod_{i : I} \Sigma(\Delta, B_i), \Id_A([\pi_1]_{i : I}\,p,x))). \]
By \rlem{coprod-dep-coprod}, we just need to prove that $\fs{in}_i'$ makes its codomain into an extensional coproduct.
Since the type $\Sigma(\Delta, \Id_A([\pi_1]_{i : I}\,p,x))$ is contractible, the following map is an equivalence:
\[ \pi_2 : \Hom(\Sigma(\Delta, p : \coprod_{i : I} \Sigma(\Delta, B_i), \Id_A([\pi_1]_{i : I}\,p,x)), \coprod_{i : I} \Sigma(\Delta, B_i)). \]
Since $\pi_2 \circ \fs{in}_i'$ makes $\coprod_{i : I} \Sigma(\Delta, B_i)$ into an extensional coproduct of $\Sigma(\Delta, B_i)$ and $\pi_2$ is an equivalence, this is also true for $\fs{in}_i'$.
\end{proof}

If we assume the stability condition for $\fs{in}$, then we can replace it with the following stable rule:
\begin{center}
\AxiomC{$\Gamma \vdash j : I$}
\AxiomC{$\Gamma \mid \Delta \vdash b : B_j$}
\BinaryInfC{$\Gamma \mid \Delta \vdash (j,b) : \coprod_{i : I} B_i$}
\DisplayProof
\end{center}
\medskip
Indeed, $\fs{in}_j(x)$ can be defined as $(j,x)$.
Conversely, $(j,b)$ can be defined in terms of $\fs{in}$ as $\fs{in}_j(x)[b/x]$.
Since $(-,-)$ is stable, these constructions are mutually inverse.

We will say that dependent coproducts are \emph{stable} if $\coprod$ and $(-,-)$ are stable.
We will say that coproducts in an indexed unary type theory are \emph{stable under pullbacks} if, for all maps $p : \Hom(\coprod_{i : I} B_i, D)$ and $r : \Hom(E,D)$,
the canonical map from $\coprod_{i : I} r^*(B_i)$ to $r^*(\coprod_{i : I} B_i)$ is an equivalence, where $r^*(X)$ is the pullback of $X$ along $r$.

If dependent coproducts are stable, then we can define the local version of their eliminator:
\begin{center}
\AxiomC{$\Gamma \mid \Delta, z : \coprod_{i : I} B_i, E \vdash D \ob$}
\AxiomC{$\Gamma, i : I \mid \Delta, x : B_i, E[(i,x)/z] \vdash d : D[(i,x)/z]$}
\BinaryInfC{$\Gamma \mid \Delta, z : \coprod_{i : I} B_i, E \vdash \coprod\text{-}\fs{elim}(z E. D, i x E. d) : D$}
\DisplayProof
\end{center}
\medskip

\[ \coprod\text{-}\fs{elim}(z E. D, i x E. d)[(j,b)/z] = d[j/i,b/x] \]

This is a strict version of the local eliminator.
We can define its weak version as usual by replacing the judgmental equality with a propositional one.

\begin{prop}
Suppose that an indexed dependent type theory has $\Sigma$-types, identity types, and dependent coproducts.
If $\coprod$ and $(-,-)$ are stable, then the local weak unstable eliminator is definable and coproducts are stable under pullbacks.
\end{prop}
\begin{proof}
First, let us prove that coproducts are stable under pullbacks.
Let $\Gamma, i : I \mid \cdot \vdash B_i \ob$ be a dependent type and let $p : \Hom(\coprod_{i : I} B_i, D)$ and $r : \Hom(E,D)$ be maps.
We define $\Gamma, i : I \mid d : D \vdash B_i' \ob$ as the fiber of $B_i$ over $d$.
Then we have the following pullback squares:
\[ \xymatrix{ \coprod_{i : I} r^*(B_i) \ar[r] \ar[d]_\simeq \pb & \coprod_{i : I} B_i \ar[d]^\simeq \\
              \coprod_{i : I} \sum_{e : E} B_i'[r\,e/d] \ar[r] \ar[d]_\simeq \pb & \coprod_{i : I} \sum_{d : D} B_i' \ar[d]^\simeq \\
              \sum_{e : E} \coprod_{i : I} B_i'[r\,e/d] \ar[r] \ar[d] \pb & \sum_{d : D} \coprod_{i : I} B_i' \ar[d] \\
              E \ar[r] & D
            } \]
The bottom square is a pullback since substitutions correspond to pullbacks.
Vertical maps in the second row are equivalences by \rprop{coprod-sigma}.
Vertical maps in the first row are equivalences since $B_i \simeq \Sigma_{d : D} B_i'$ and $r^*(B_i) \simeq B_i'[r\,e/d]$.

Now, let us prove the existence of the local weak unstable eliminator.
By \rlem{coprod-dep-coprod}, it is enough to show that the following map makes its codomain into a coproduct:
\[ \fs{in}_i : \Hom(\Sigma(x : B_i, E[(i,x)/z]), \Sigma(z : \coprod_{i : I} B_i, E)). \]
This follows from the fact that coproducts are stable under pullbacks since the codomain of this map is a pullback of $\id_{\coprod_{i : I} B_i}$ along the obvious projection from $\Sigma(z : \coprod_{i : I} B_i, E)$ to $\coprod_{i : I} B_i$.
\end{proof}

\subsection{Pushouts}

Pushouts can be defined in a more type-theoretic way.
The theory of \emph{(unstable) weak dependent pushouts} consists of the following unstable rules:
\begin{center}
\AxiomC{$\Gamma \mid \Delta \vdash f : \Hom(A,B)$}
\AxiomC{$\Gamma \mid \Delta \vdash g : \Hom(A,C)$}
\BinaryInfC{$\Gamma \mid \Delta \vdash B \amalg_A C \ob$}
\DisplayProof
\end{center}
\medskip

\begin{center}
\AxiomC{}
\UnaryInfC{$\Gamma \mid \Delta, y : B \vdash \fs{inl}(y) : B \amalg_A C$}
\DisplayProof
\qquad
\AxiomC{}
\UnaryInfC{$\Gamma \mid \Delta, z : C \vdash \fs{inr}(z) : B \amalg_A C$}
\DisplayProof
\end{center}
\medskip

\begin{center}
\AxiomC{}
\UnaryInfC{$\Gamma \mid \Delta, x : A \vdash \fs{glue}(x) : \Id(\fs{inl}(y)[f\,x/y],\fs{inr}(z)[g\,x/z])$}
\DisplayProof
\end{center}
\medskip

To simplify the notation, we will write $\fs{inl}(b)$, $\fs{inr}(c)$, and $\fs{glue}(a)$ instead of $\fs{inl}(y)[b/y]$, $\fs{inr}(z)[c/z]$, and $\fs{glue}(x)[a/x]$, respectively.

\begin{center}
\def\extraVskip{1pt}
\Axiom$\fCenter \Gamma \mid \Delta, w : B \amalg_A C \vdash D \ob$
\noLine
\UnaryInf$\fCenter \Gamma \mid \Delta, y : B \vdash d_1 : D[\fs{inl}(y)/w]$
\noLine
\UnaryInf$\fCenter \Gamma \mid \Delta, z : C \vdash d_2 : D[\fs{inr}(z)/w]$
\noLine
\UnaryInf$\fCenter \Gamma \mid \Delta, x : A \vdash d_3 : \Id(\fs{glue}(x)_*(d_1[f\,x/y]), d_2[g\,x/z])$
\def\extraVskip{2pt}
\UnaryInf$\fCenter \Gamma \mid \Delta, w : B \amalg_A C \vdash \amalg\text{-}\fs{elim}(w.D, y.d_1, z.d_2, x.d_3) : D$
\DisplayProof
\end{center}
\medskip

\begin{align*}
h_1(b) & : \Id(\amalg\text{-}\fs{elim}(w.D,y.d_1,z.d_2,x.d_3)[\fs{inl}(b)/w], d_1[b/y]) \\
h_2(c) & : \Id(\amalg\text{-}\fs{elim}(w.D,y.d_1,z.d_2,x.d_3)[\fs{inr}(c)/w], d_2[c/z])
\end{align*}

\[ \xymatrix{ \fs{glue}(a)_*(\amalg\text{-}\fs{elim}(w.D,y.d_1,z.d_2,x.d_3)[\fs{inl}(f\,a)/w]) \ar@{=}[r] \ar@{=}[d] & \fs{glue}(a)_*(d_1[f\,a/y]) \ar@{=}[d]^{d_3[a/x]} \\
              \amalg\text{-}\fs{elim}(w.D,y.d_1,z.d_2,x.d_3)[\fs{inr}(g\,a)/w] \ar@{=}[r]_-{h_2(g\,a)} & d_2[g\,a/z]
            } \]
The last square must commute up to a homotopy $h_3(a)$.
The top arrow in this square is $\pmap(x.\,\fs{glue}(a)_*(x), h_1(f\,a))$ and the left arrow is defined by path induction on $\fs{glue}(a)$.

\begin{prop}[pushout-sigma]
Dependent pushouts commute with $\Sigma$-types.
More precisely, the following map is an equivalence for all maps $f : \Hom_\Delta(A,B)$ and $g : \Hom_\Delta(A,C)$:
\begin{align*}
& \Hom_\Delta((\sum_{d : D} B) \amalg_{(\sum_{d : D} A)} (\sum_{d : D} C), \sum_{d : D} B \amalg_A C) \\
& \lambda w. \amalg\text{-}\fs{elim}(w. \sum_{d : D} B \amalg_A C, (d,y).(d,\fs{inl}(y)), (d,z).(d,\fs{inr}(z)), (d,x).d_3),
\end{align*}
where $d_3 = (d,\pmap(w.(d,w),\fs{glue}(x)))$.
\end{prop}
\begin{proof}
The inverse of this map is defined as follows:
\[ \lambda (d,w). \amalg\text{-}\fs{elim}(w. (\sum_{d : D} B) \amalg_{(\sum_{d : D} A)} (\sum_{d : D} C), y.(d,y), z.(d.z), x.\pmap(w.(d,w),\fs{glue}(x))). \]
It is easy to show that these maps are mutually inverse using the eliminator for pushouts.
\end{proof}

\begin{lem}[dep-pushouts]
Suppose that an indexed dependent type theory has $\Sigma$-types and extensional identity types.
Let $B \amalg_A C$, $\fs{inl}$, $\fs{inr}$, and $\fs{glue}$ be unstable constructions as defined above.
Then the eliminator is definable if and only if the following square is a pushout:
\[ \xymatrix{ \Sigma(\Delta,A) \ar[rr]^-{T(f)} \ar[d]_{T(g)}        & & \Sigma(\Delta,B) \ar[d]^{T(\lambda y.\fs{inl}(y))} \\
              \Sigma(\Delta,C) \ar[rr]_-{T(\lambda z.\fs{inr}(z))}  & & \Sigma(\Delta, B \amalg_A C)
            } \]
where $T(h) = \lambda x_1 \ldots x_n x.\,(x_1, \ldots x_n, h\,x)$.
\end{lem}
\begin{proof}
First, suppose that the eliminator is definable.
Then we need to show that the following canonical map is an equivalence:
\[ \Hom(\Sigma(\Delta, B \amalg_A C), D) \to \Hom(\Sigma(\Delta,B),D) \times_{\Hom(\Sigma(\Delta,A),D)} \Hom(\Sigma(\Delta,C),D). \]
The domain of this map is equivalent to $\Hom_\Delta(B \amalg_A C, D)$ and the codomain is equivalent to $\Sigma_{b : \Hom_\Delta(B,D)} \Sigma_{c : \Hom_\Delta(C,D)} \Id_{\Hom_\Delta(A,D)}(\lambda \overline{x} x.\,b\,\overline{x}\,(f\,x), \lambda \overline{x} x.\,c\,\overline{x}\,(g\,x))$.
It follows that this map is an equivalence if and only if the following one is:
\begin{align*}
F & : \Hom_\Delta(B \amalg_A C, D) \to \sum_{b : \Hom_\Delta(B,D)} \sum_{c : \Hom_\Delta(C,D)} \Id(\lambda \overline{x} x.\,b\,\overline{x}\,(f\,x), \lambda \overline{x} x.\,c\,\overline{x}\,(g\,x)) \\
F & = \lambda h.\,(\lambda \overline{x} y.\,h\,\overline{x}\,(\fs{inl}(y)), \lambda \overline{x} z.\,h\,\overline{x}\,(\fs{inr}(z)), \fs{Idext}(\overline{x} x.\,\pmap(h\,\overline{x},\fs{glue}(x))))
\end{align*}
The inverse of this function is defined as follows:
\begin{align*}
G & : \sum_{b : \Hom_\Delta(B,D)} \sum_{c : \Hom_\Delta(C,D)} \Id(\lambda \overline{x} x.\,b\,\overline{x}\,(f\,x), \lambda \overline{x} x.\,c\,\overline{x}\,(g\,x)) \to \Hom_\Delta(B \amalg_A C, D) \\
G & = \lambda (b,c,p).\,\lambda \overline{x} w.\,\amalg\text{-}\fs{elim}(w.D, y.\,b\,\overline{x}\,y, z.\,c\,\overline{x}\,z, x.\,\fs{hap}(p,\overline{x},x))
\end{align*}
For every triple $(b,c,p)$, we need to show that $F\,(G\,(b,c,p)) \sim (b,c,p)$.
It is easy to construct this homotopy using $h_1$, $h_2$, and $h_3$.
For every map $h : \Hom_\Delta(B \amalg_A C, D)$, we need to show that $G\,(F\,h) \sim h$.
We have the following homotopy:
\[ G\,(F\,h) \sim \lambda \overline{x} w.\,\amalg\text{-}\fs{elim}(w.D, y.\,h\,\overline{x}\,(\fs{inl}(y)), z.\,h\,\overline{x}\,(\fs{inr}(z)), x.\,\pmap(h\,\overline{x}, \fs{glue}(x))). \]
Let us denote the latter function by $h'$.
We can define a homotopy $h'$ and $h$ as follows:
\[ \fs{Idext}(\overline{x} w.\,\amalg\text{-}\fs{elim}(w.\,\Id(h'\,\overline{x}\,w,h\,\overline{x}\,w), y.\,h_1(y), z.\,h_2(z), x.\,h_3'(x))), \]
where $h_3'(x)$ is a homotopy between $\fs{glue}(x)_*(h_1(f\,x))$ and $h_2(g\,x)$.
Since the former term is homotopic to $\sym{\pmap(h'\,\overline{x},\fs{glue}(x))} \ct h_1(f\,x) \ct \pmap(h\,\overline{x},\fs{glue}(x))$, we can construct $h_3'(x)$ using $h_3(x)$.

Now, let us prove the converse.
Let $D$, $d_1$, $d_2$, and $d_3$ be arguments of $\amalg\text{-}\fs{elim}$.
Then we have the following commutative square:
\[ \xymatrix{ \Sigma(\Delta,A) \ar[rrr]^-{T(f)} \ar[d]_{T(g)}               & & & \Sigma(\Delta,B) \ar[d]^{T(\lambda y.(\fs{inl}(y),d_1))} \\
              \Sigma(\Delta,C) \ar[rrr]_-{T(\lambda z.(\fs{inr}(z),d_2))}   & & & \Sigma(\Delta, w : B \amalg_A C, D)
            } \]
The commutativity of the square is witnessed by the following term:
\[ \fs{Idext}((x_1, \ldots x_n, x).\,\pmap(p.(x_1, \ldots x_n, p), \Sigma\fs{ext}(\fs{glue}(x),d_3))). \]
By the universal property of pushouts, we have the following terms:
\begin{align*}
\Delta, w : B \amalg_A C & \vdash s(w) : \Sigma(\Delta, w : B \amalg_A C, D) \\
\Delta, y : B & \vdash h_1'(y) : \Id(s(\fs{inl}(y)),(\overline{x},\fs{inl}(y),d_1)) \\
\Delta, z : C & \vdash h_2'(z) : \Id(s(\fs{inr}(z)),(\overline{x},\fs{inr}(z),d_2))
\end{align*}
and term $h_3'(x)$ which witnesses the commutativity of the following square:
\[ \xymatrix{ s(\fs{inl}(f\,x)) \ar@{=}[r]^-{h_1'(f\,x)} \ar@{=}[d]_{\pmap(w.s(w),\fs{glue}(x))} & (\overline{x},\fs{inl}(f\,x),d_1[f\,x/y]) \ar@{=}[d]^{\pmap(p.(\overline{x},p),\Sigma\fs{ext}(\fs{glue}(x),d_3))} \\
              s(\fs{inr}(g\,x)) \ar@{=}[r]_-{h_2'(g\,x)} & (\overline{x},\fs{inr}(g\,x),d_2[g\,x/z])
            } \]
Let $\pi_0$ be the obvious projection $\Hom(\Sigma(\Delta, w : B \amalg_A C, D), \Sigma(\Delta, B \amalg_A C))$.
By the uniqueness, we have the following terms:
\begin{align*}
\Delta, w : B \amalg_A C & \vdash q(w) : \Id(\pi_0(s(w)),(\overline{x},w)) \\
\Delta, y : B & \vdash q_1(y) : \Id(h(\fs{inl}(y)),\pmap(\pi_0,h_1'(y))) \\
\Delta, z : C & \vdash q_2(z) : \Id(h(\fs{inr}(z)),\pmap(\pi_0,h_2'(z)))
\end{align*}
and $q_3(x)$ which is a homotopy between two terms witnessing the commutativity of the following square:
\[ \xymatrix{ \pi_0(s(\fs{inl}(f\,x))) \ar@{=}[rr]^-{q(\fs{inl}(f\,x))} \ar@{=}[d]_{\pmap(w.\pi_0(s(w)),\fs{glue}(x))} & & (\overline{x},\fs{inl}(f\,x)) \ar@{=}[d]^{\pmap(p.(\overline{x},p),\fs{glue}(x))} \\
              \pi_0(s(\fs{inr}(g\,x))) \ar@{=}[rr]_-{q(\fs{inr}(g\,x))} & & (\overline{x},\fs{inr}(g\,x))
            } \]
One of this term is define by path induction on $\fs{glue}(x)$ and the other one is obtained from $h_3'(x)$, $q_1(f\,x)$, and $q_2(g\,x)$.
We define $\amalg\text{-}\fs{elim}(w.D, y.d_1, z.d_2, x.d_3)$ as $q_*(\pi_{n+2}(s))$.
Maps $h_1$, $h_2$, and $h_3$ can be defined using $h_1'$ and $q_1$, $h_2'$ and $q_2$, and $h_3'$ and $q_3$, respectively.
\end{proof}

\begin{prop}
If an indexed dependent type theory has $\Sigma$-types and extensional identity types, then it has weak dependent pushouts if and only if it has pushouts.
\end{prop}
\begin{proof}
\Rlem{dep-pushouts} implies that the theory with weak dependent pushouts has pushouts.
Conversely, suppose that the theory has pushouts.
Let $\Gamma \mid \Delta \vdash f : \Hom(A,B)$ and $\Gamma \mid \Delta \vdash g : \Hom(A,C)$ be a pair of maps.
If $\Delta$ is empty, then an extensional dependent pushout of $f$ and $g$ exists by \rlem{dep-pushouts}.
Thus, we may assume that $\Delta$ is not empty.
Since we have $\Sigma$-types, we may also assume that it consists of a single type, that is $\Delta = (u : U)$.

By the universal property of pushouts, there exists a map $p : \Hom(B \amalg_A C, U)$ such that $p \circ \fs{inl} \sim \pi_1$ and $p \circ \fs{inr} \sim \pi_1$.
We define $B \amalg_A' C$ as $\Gamma \mid \Delta \vdash \Sigma_{w : B \amalg_A C} \Id(p\,w,u)$.
By \rlem{dep-pushouts}, to prove that $B \amalg_A' C$ is a pushouts of $f$ and $g$, it is enough to show that $\Sigma(\Delta, B \amalg_A' C)$ is a pushout of $T(f)$ and $T(g)$.
This follows from the fact that $\Sigma(\Delta, B \amalg_A' C)$ is equivalent to $B \amalg_A C$.
\end{proof}

We will say that pushouts in an indexed unary type theory are \emph{stable under pullbacks} if, for all maps $f : \Hom(A,B)$, $g : \Hom(A,C)$, $p : \Hom(B \amalg_A C, D)$, and $r : \Hom(E,D)$,
the canonical map from $r^*(B) \amalg_{r^*(A)} r^*(C)$ to $r^*(B \amalg_A C)$ is an equivalence, where $r^*(X)$ is the pullback of $X$ along $r$.
If dependent pushouts are stable, then we can define the local version of their eliminator:
\begin{center}
\def\extraVskip{1pt}
\Axiom$\fCenter \Gamma \mid \Delta, w : B \amalg_A C, E \vdash D \ob$
\noLine
\UnaryInf$\fCenter \Gamma \mid \Delta, y : B, E[\fs{inl}(y)/w] \vdash d_1 : D[\fs{inl}(y)/w]$
\noLine
\UnaryInf$\fCenter \Gamma \mid \Delta, z : C, E[\fs{inr}(z)/w] \vdash d_2 : D[\fs{inr}(z)/w]$
\noLine
\UnaryInf$\fCenter \Gamma \mid \Delta, x : A, E[\fs{inr}(g\,x)/w] \vdash d_3 : \Id(d_1', d_2[g\,x/z])$
\def\extraVskip{2pt}
\UnaryInf$\fCenter \Gamma \mid \Delta, w : B \amalg_A C, E \vdash \amalg\text{-}\fs{elim}(z E.D, x E.d_1, y E.d_2, w E.d_3) : D$
\DisplayProof
\end{center}
where $E = z_1 : E_1, \ldots z_n : E_n$, $d_1' = \fs{glue}(x)_*(d_1[\rho(\fs{glue}(x)), f\,x/w])$, and $\rho(p) = \overline{\sym{p}_*(z_i)/z_i}$.
The definitions of $h_1$, $h_2$, and $h_3$ are modified appropriately.

\begin{prop}
Suppose that an indexed dependent type theory has $\Sigma$-types, identity types, and dependent pushouts.
If $B \amalg_A C$, $\fs{inl}$, $\fs{inr}$, and $\fs{glue}$ are stable, then the local weak unstable eliminator is definable and pushouts are stable under pullbacks.
\end{prop}
\begin{proof}
First, let us prove that pushouts are stable under pullbacks.
Suppose that we have the following maps: $f : \Hom(A,B)$, $g : \Hom(A,C)$, $p : \Hom(B \amalg_A C, D)$, and $r : \Hom(E,D)$.
We define $d : D \vdash A' \ob$ as the fiber of $A$ over $d$.
Types $B'$ and $C'$ and maps $f' : \Hom_{d : D}(A',B')$ and $g' : \Hom_{d : D}(A',C')$ are defined similarly.
Then we have the following pullback squares:
\[ \xymatrix{ r^*(B) \amalg_{r^*(A)} r^*(C) \ar[r] \ar[d]_\simeq \pb & B \amalg_A C \ar[d]^\simeq \\
              (\Sigma_{e : E} B'[r\,e/d]) \amalg_{(\Sigma_{e : E} A'[r\,e/d])} (\Sigma_{e : E} C'[r\,e/d]) \ar[r] \ar[d]_\simeq \pb & (\Sigma_{d : D} B') \amalg_{(\Sigma_{d : D} A')} (\Sigma_{d : D} C') \ar[d]^\simeq \\
              \Sigma_{e : E} (B' \amalg_{A'} C')[r\,e/d] \ar[r] \ar[d] \pb & \Sigma_{d : D} B' \amalg_{A'} C' \ar[d] \\
              E \ar[r] & D
            } \]
The bottom square is a pullback since substitutions correspond to pullbacks.
Vertical maps in the second row are equivalences by \rprop{pushout-sigma}.

Now, let us prove the existence of the local weak unstable eliminator.
By \rlem{dep-pushouts}, it is enough to prove that $\Sigma(w : B \amalg_A C, E)$ is a pushout of the following maps:
\begin{align*}
\lambda (x, \overline{z}).\,(f\,x, \overline{z}) & : \Hom_\Delta(\Sigma(x : A, E[\fs{inl}(f\,x)/w]), \Sigma(y : B, E[\fs{inl}(y)])) \\
\lambda (x, \overline{z}).\,(g\,x, \overline{\fs{glue}(x)_*(z)}) & : \Hom_\Delta(\Sigma(x : A, E[\fs{inl}(f\,x)/w]), \Sigma(z : C, E[\fs{inr}(z)]))
\end{align*}
This follows from the fact that pushouts are stable under pullbacks since $\Sigma(w : B \amalg_A C, E)$ is a pullback of $\id_{B \amalg_A C}$ along the obvious projection from $\Sigma(w : B \amalg_A C, E)$ to $B \amalg_A C$.
\end{proof}

\section{The initial type theorem}
\label{sec:initial}

The general adjoint functor theorem holds in the context of indexed categories \cite[IV.1]{indexed-cats} and in the context of $\infty$-categories \cite{infty-gaft}.
Thus, it is natural to assume that it also should hold in the context of indexed type theories.
To properly state this theorem, we need to define the notion of adjoint functors between models of such theories.
This paper focuses on internal properties of a single model of an indexed type theory.
So, we only consider the first step in the proof of the adjoint functor theorems, which is known as \emph{the initial object theorem} or \emph{the initial type theorem} in our case.
This theorem is proved in \cite[IV.1.1]{indexed-cats} for indexed categories and in \cite[Proposition~2.3.2]{infty-gaft} for $\infty$-categories.

\subsection{$h$-initial types}

In this subsection, we will prove an analogue of \cite[Proposition~2.2.2]{infty-gaft}.
This proposition states that $h$-initial objects are initial in finitely complete $\infty$-categories.
An object $Z$ is $h$-initial if the space $\Hom(Z,X)$ is connected (and inhabited) for all $X$.
This proposition has two problems in the context of indexed type theories.
The first one is that it seems that it is not enough to assume the existence of finite limits since this only implies that homotopy groups of $\Hom(Z,X)$ vanish which might be not enough to conclude that this type is contractible.
For this reason, we replace this condition with the condition that all powers exist.
The second problem is that the definition of $h$-initial objects involves the propositional truncation, but it might not exist in general.
We solve this problem by replacing the condition of connectedness by a weaker condition which can be formulated without the propositional truncation.

First, for every type $X$, we define a weakening of the condition that $X$ is inhabited.
Similar condition was defined in \cite[Definition~5]{gen-hedberg}: a type $X$ is \emph{populated}, written $\mathrm{isPop}(X)$, if every constant endofunction on $X$ has a fixed point.
We will say that $X$ is \emph{weakly populated}, written $\mathrm{isWPop}(X)$, if $\mathrm{isProp}(X) \to X$, that is if $X$ is inhabited whenever it is a proposition.
Clearly, $\mathrm{isWPop}(X)$ is a proposition.

\begin{prop}
We have the following sequence of implications:
\[ X \to \| X \| \to \mathrm{isPop}(X) \to \mathrm{isWPop}(X) \to \neg \neg X. \]
\end{prop}
\begin{proof}
The first implication is obvious and the second follows from the fact that $\mathrm{isPop}(X)$ is a proposition.
Suppose that $X$ is populated.
To prove that it is weakly populated, we may assume that it is a proposition.
Then the identity endofunction on $X$ is constant.
Thus, there exists a point in $X$ (namely, the fixed point of $\id_X$).
Finally, suppose that $X$ is weakly populated and let us prove that $\neg \neg X$.
Assume that $\neg X$.
This implies that $X$ is a proposition.
Since $X$ is weakly populated, it is inhabited, which is a contradiction.
\end{proof}

We will say that a type $X$ is \emph{weakly connected} if, for all $x, x' : X$, the type $\Id(x,x')$ is weakly populated.
We will say that an indexed type $Z$ is \emph{$h$-initial} if, for every indexed type $X$, the type $\Hom(Z,X)$ is weakly connected and inhabited.

\begin{prop}
Let $0$ be an indexed type in a locally small indexed unary type theory such that, for every indexed type $X$, the type $\Hom(0,X)$ is weakly connected.
If the theory has extensional powers, then $\Hom(0,X)$ is a proposition for every $X$.
\end{prop}
\begin{proof}
Since the theory has powers, there is a type $X^{\Hom(0,X)}$ such that the type $\Hom(0,X^{\Hom(0,X)})$ is equivalent to $\Hom(0,X) \to \Hom(0,X)$.
This implies that the type $\Hom(0,X) \to \Hom(0,X)$ is weakly populated.
Let $f,g : \Hom(0,X)$ be a pair of maps.
We need to construct a homotopy between them.

First, let us prove the type $\Id_{\Hom(0,X) \to \Hom(0,X)}(\id, \lambda x.g)$ is a proposition.
Since powers are extensional, the functional extensionality holds for the type $\Hom(0,X) \to \Hom(0,X)$.
Thus, we just need to prove that, for all $x : \Hom(0,X) \vdash p : \Id(x,g)$ and $x : \Hom(0,X) \vdash q : \Id(x,g)$, there is a homotopy $x : X \vdash h : \Id(p,q)$.
It is enough to prove that, for all $x$ and $q$, there is a homotopy between $q$ and $p \ct \sym{p[g/x]}$, which follows by path induction.

Now, since $\Hom(0,X) \to \Hom(0,X)$ is weakly populated and $\Id(\id_{\Hom(0,X)}, \lambda x.g)$ is a proposition, the latter type is inhabited.
If $h$ is a homotopy between $\id$ and $\lambda x.g$, then $\fs{hap}(h,f)$ is a homotopy between $f$ and $g$.
\end{proof}

\begin{cor}
Any $h$-initial type in a locally small indexed unary type theory with extensional powers is initial.
\end{cor}

\subsection{Split idempotents}

An idempotent in a 1-category is a map $h : B \to B$ such that $h \circ h = h$.
In the setting of $\infty$-categories an idempotent consists of a map and an infinite amount of coherence data.
Lurie proved in \cite[Lemma~7.3.5.14]{lurie-algebra} that if $h$ is a map such that there exists a homotopy between $h \circ h$ and $h$ satisfying one additional coherence condition, then $h$ can be extended to an idempotent.

An idempotent $h$ in a 1-category is split if there are maps $f : A \to B$ and $g : B \to A$ such that $g \circ f = \id_A$ and $f \circ g = h$.
If the category is finitely complete, then every idempotent is split since $f$ can be defined as the equalizer of $h$ and $\id_B$ and $g$ exists by the universal property.
It is no longer true that the splitting of an idempotent in an $\infty$-category can be constructed as a limit of a finite diagram, but it is a limit of a countable diagram.
This question in the context of ordinary homotopy type theory was discussed by Shulman in \cite{split-idemp}.
We can repeat this argument in the setting of indexed type theories.

\begin{defn}
A map $h : \Hom(B,B)$ in an indexed unary type theory consists is \emph{idempotent} if there are the following terms:
\begin{align*}
I & : \Id_{\Hom(B,B)}(h, h \circ h) \\
J & : \Id_{\Id(h \circ h, h \circ h \circ h)}(\pmap(- \circ h, I), \pmap(h \circ -, I))
\end{align*}
\end{defn}

\begin{defn}
An idempotent map $h : \Hom(B,B)$ is \emph{split} if there exist maps $f : \Hom(A,B)$ and $g : \Hom(B,A)$ such that $g \circ f \sim \id_A$ and $f \circ g \sim h$.
\end{defn}

\begin{prop}
If the base theory of a locally small indexed unary type theory has natural numbers $\mathbb{N}$ and the indexed theory has equalizers and extensional powers $B^\mathbb{N}$ for some type $B$, then every idempotent on $B$ is split.
\end{prop}
\begin{proof}
We define the splitting of an idempotent map $h : \Hom(B,B)$ as the limit of the following sequence:
\[ \ldots \xrightarrow{h} B \xrightarrow{h} B \xrightarrow{h} B. \]
More precisely, let $e : \Hom(A,B^\mathbb{N})$ be the equalizer of $\id, h \circ - \circ \mathrm{suc} : \Hom(B^\mathbb{N},B^\mathbb{N})$.
Let $d : \Id(e, h \circ e(-) \circ \mathrm{suc})$ be the witness of the fact that $e$ equalizes these maps.
Let $f : \Hom(A,B)$ be the following composite:
\[ A \xrightarrow{e} B^\mathbb{N} \xrightarrow{\lambda f.\,f\,0} B. \]
By the universal property of equalizers, to define a map $g : \Hom(B,A)$, it is enough to define a map $g' : \Hom(B,B^\mathbb{N})$ and a homotopy $p : \Id(g', h \circ g'(-) \circ \mathrm{suc})$.
Let $g' = \lambda b. \lambda n.\,h\,b$.
Since powers are extensional, to define $p$, it is enough to define a homotopy between $h$ and $h \circ h$, which we define as $I$.
It is easy to see that $f \circ g \sim h$.
Thus, we just need to prove that $g \circ f \sim \id_A$.

By the universal properties of equalizers, to construct this homotopy, it is enough to prove that for every $n : \mathbb{N}$, terms $(e,d)$ and $(e',d')$ of type $\Sigma_{r : \Hom(A,B^\mathbb{N})} \Id(r, h \circ r(-) \circ \mathrm{suc})$ are homotopic,
where $e' = \lambda a. \lambda n.\,h\,(e\,a\,0)$ and $d'$ is the homotopy between $e'$ and $\lambda a. \lambda n.\,h\,(h\,(e\,a\,0)))$ obtained from $I$.
First, we will construct a homotopy between $e$ and $e'$.
By the universal property of extensional powers, it is enough to prove that for every $n : \mathbb{N}$, maps $\lambda a.\,e\,a\,n$ and $\lambda a.\,h\,(e\,a\,0)$ are homotopic.
Let $s(n)$ be the homotopy between $\lambda a.e\,a\,n$ and $\lambda a.\,h\,(e\,a\,(n+1))$ obtained from $d$.
Then we define a homotopy between $\lambda a.\,e\,a\,n$ and $\lambda a.\,h\,(e\,a\,0)$ as $s(n) \ct T(n)$, where $T(n) : \Id(\lambda a.\,h\,(e\,a\,(n+1)), \lambda a.\,h\,(e\,a\,0))$ is defined by induction.
Let us write $I(a.t) : \Id(h \circ \lambda a.t, h \circ h \circ \lambda a.t)$ for $\pmap(\lambda f.\,\lambda a.\,f \circ t, I)$.
Then we define $T$ as follows:
\begin{align*}
T(0) & = I(a.\,e\,a\,1) \ct \pmap(h \circ -, \sym{s(0)}) \\
T(n+1) & = I(a.\,e\,a\,(n+2)) \ct \pmap(h \circ -, \sym{s(n+1)}) \ct T(n)
\end{align*}

Now, we need to construct a homotopy between the second components of pairs $(e,d)$ and $(e',d')$.
By the universal property of extensional powers, it is enough to prove that for every $n : \mathbb{N}$, the following square commutes:
\[ \xymatrix{ \lambda a.\,h\,(e\,a\,(n+1)) \ar@{=}[rr]^-{\pmap(h \circ -, s(n+1))} \ar@{=}[d]_{T(n)}    & & \lambda a.\,h\,(h\,(e\,a\,(n+2))) \ar@{=}[d]^{\pmap(h \circ -, T(n+1))} \\
              \lambda a.\,h\,(e\,a\,0) \ar@{=}[rr]_-{I(a.\,e\,a\,0)}                                    & & \lambda a.\,h\,(h\,(e\,a\,0))
            } \]
By the definition of $T(n+1)$, the top path is homotopic to $\pmap(h \circ -, s(n+1) \ct I(a.\,e\,a\,(n+2)) \ct \pmap(h \circ -, \sym{s(n+1)}) \ct T(n))$.
By $J$, we have a homotopy $\pmap(h \circ -, I(a.\,e\,a\,(n+2))) \sim I(a.\,h\,(e\,a\,(n+2)))$.
By path induction, we have a homotopy $\pmap(h \circ -, s(n+1)) \ct I(a.\,h\,(e\,a\,(n+2))) \ct \pmap(h \circ h \circ -, \sym{s(n+1)}) \sim I(a.\,e\,a\,(n+1))$.
Thus, we just need to prove that the following square is commutative:
\[ \xymatrix{ \lambda a.\,h\,(e\,a\,(n+1)) \ar@{=}[rr]^-{I(a.\,e\,a\,(n+1))} \ar@{=}[d]_{T(n)}  & & \lambda a.\,h\,(h\,(e\,a\,(n+1))) \ar@{=}[d]^{\pmap(h \circ -, T(n))} \\
              \lambda a.\,h\,(e\,a\,0) \ar@{=}[rr]_-{I(a.\,e\,a\,0)}                            & & \lambda a.\,h\,(h\,(e\,a\,0))
            } \]
We do this by induction on $n$.
If $n = 0$, then $I(a.\,e\,a\,1) \ct \pmap(h \circ -, T(0)) \sim I(a.\,e\,a\,1) \ct I(a.\,h\,(e\,a\,1)) \ct \pmap(h \circ -, \sym{s(0)})$ and $T(0) \ct I(a.\,e\,a\,0) = I(a.\,e\,a\,1) \ct \pmap(h \circ -, \sym{s(0)}) \ct I(a.\,e\,a\,0)$.
By path induction on $s(0)$, these terms are homotopic.
If $n = n' + 1$, then the top and the bottom paths in the diagram are homotopic to the following terms:
\begin{align*}
& I(a.\,e\,a\,(n+1)) \ct I(a.\,h\,(e\,a\,(n+1))) \ct \pmap(h \circ h \circ -, \sym{s(n)}) \ct \pmap(h \circ -, T(n')) \\
& I(a.\,e\,a\,(n+1)) \ct \pmap(h \circ -, \sym{s(n)}) \ct T(n') \ct I(a.\,e\,a\,0)
\end{align*}
By the induction hypothesis, we have a homotopy $T(n') \ct I(a.\,e\,a\,0) \sim I(a.\,e\,a\,n) \ct \pmap(h \circ -, T(n'))$.
Thus, we just need to prove that terms $I(a.\,h\,(e\,a\,(n+1))) \ct \pmap(h \circ h \circ -, \sym{s(n)})$ and $\pmap(h \circ -, \sym{s(n)}) \ct I(a.\,e\,a\,n)$ are homotopic, which follows by path induction on $s(n)$.
\end{proof}

\subsection{The initial type theorem}

A \emph{weakly initial family of indexed types} is a family $\Gamma, i : I \mid \cdot \vdash W_i \ob$ such that, for every indexed type $\Gamma \mid \cdot \vdash B \ob$, there exists an index $\Gamma \vdash i : I$ and a map $\Gamma \vdash b_i : \Hom(W_i,B)$.
We will prove the initial type theorem for indexed dependent type theories.
This is merely a technical convenience; the theorem should also be true in unary theories.

\begin{thm}
Suppose that a locally small indexed dependent type theory has $\Sigma$-types, extensional identity types, extensional dependent products, and split idempotents.
If it has a weakly initial family of indexed types, then it also has the initial type.
\end{thm}
\begin{proof}
Let $W$ be the product of a weakly initial family of indexed types.
Then $W$ is weakly initial in the sense that, for every indexed type $B$, there exists a map from $W$ to $B$.
Let $Z$ be the following type:
\[ \sum_{(x : W)} \sum_{(h : \prod_{f : \Hom(W,W)} \Id(x,f\,x))} \prod_{(f : \Hom(W,W))} \prod_{(p : \Id(f, f \circ f))} \Id(\fs{hap}(p,x),\pmap(f,h\,f)). \]
Let $e : \Hom(Z,W)$ be the first projection and let $r : \Hom(W,Z)$ be any map.
We will prove that $h = r \circ e$ is idempotent.
We define $I : \Id(h, h \circ h)$ as follows:
\[ I = \fs{Idext}(z.\,\pmap(r, \pi_2\,z\,(e \circ r))). \]
To construct a homotopy $J$, it is enough to define the following homotopy:
\[ z : Z \vdash J' : \Id(\pmap(r, \pi_2\,(h\,z)\,(e \circ r), \pmap(h \circ r, \pi_2\,z\,(e \circ r))). \]
We define $J'$ as $\pmap(r,-)$ applied to the following sequence:
\[ \pi_2\,(h\,z)\,(e \circ r) \sim \fs{hap}(\fs{Idext}(w.\,\pi_2\,(r\,w)\,(e \circ r)),e\,z) \sim \pmap(e \circ r, \pi_2\,z\,(e \circ r)). \]
The second homotopy here is $\pi_3\,z\,(e \circ r)\,(\fs{Idext}(w.\,\pi_2\,(r\,w)\,(e \circ r)),e\,z)$.

Since idempotents are split, there exist a type $0$ and maps $q : \Hom(0,Z)$ and $p : \Hom(Z,0)$ such that $p \circ q \sim \id_0$ and $q \circ p \sim r \circ e$.
Let us prove that $0$ is initial.
For every type $B$, there exists a map $0 \xrightarrow{q} Z \xrightarrow{e} W \to B$.
Thus, it is enough to construct a homotopy between any two maps $f,g : \Hom(0,B)$.
Since there exists a map from $W$ to the equalizer of $f \circ p$ and $g \circ p$, there is a map $r' : \Hom(W,Z)$ such that $f \circ p \circ r' \sim g \circ p \circ r'$.
Then $\fs{Idext}(z.\,\pi_2\,z\,(e \circ r'))$ is a homotopy between $e$ and $e \circ r' \circ e$.
Then we have the following sequence of homotopies:
\[ p \circ r' \circ e \sim p \circ q \circ p \circ r' \circ e \sim p \circ r \circ e \circ r' \circ e \sim p \circ r \circ e \sim p \circ q \circ p \sim p. \]
It follows that $f$ and $g$ are homotopic:
\[ f \sim f \circ p \circ q \sim f \circ p \circ r' \circ e \circ q \sim g \circ p \circ r' \circ e \circ q \sim g \circ p \circ q \sim g. \]
\end{proof}

\section{Classifying morphisms}
\label{sec:class}

In this section, we will define the notion of classifying morphisms, discuss its relationship to the notions of universes and factorization systems.

\subsection{Truncated maps}

Let $f : \Hom(A,B)$ be a map in an indexed unary type theory and let $n$ be an integer $\geq -2$.
We will say that $f$ is \emph{$n$-truncated} if the map $f \circ - : \Hom(X,A) \to \Hom(X,B)$ is $n$-truncated for all indexed types $X$.
This definition makes sense in models of indexed unary type theories, but the problem is that it is not algebraic since it quantifies over indexed types.
That is, we cannot define a predicate on $\Hom(A,B)$ which corresponds to the notion of $n$-truncated maps in models.
But we can fix this problem if the indexed theory has pullbacks.

First, let us prove a few technical lemmas:

\begin{lem}[trunc-pb]
In an indexed unary type theory, $n$-truncated maps are closed under pullbacks.
\end{lem}
\begin{proof}
Suppose that we have a pullback square in which the right arrow is $n$-truncated:
\[ \xymatrix{ A \ar[r] \ar[d] \pb   & C \ar[d] \\
              B \ar[r]              & D.
            } \]
Since $\Hom(X,-)$ preserves pullbacks, the following square is also pullback and the right arrow is $n$-truncated by the definition of $n$-truncated maps in indexed type theories:
\[ \xymatrix{ \Hom(X,A) \ar[r] \ar[d] \pb   & \Hom(X,C) \ar[d] \\
              \Hom(X,B) \ar[r]              & \Hom(X,D).
            } \]

Thus, we just need to prove that $n$-truncated maps are closed under pullbacks in the base theory.
Suppose that we have the following pullback square:
\[ \xymatrix{ \sum_{b : B} P(f(b)) \ar[r] \ar[d] \pb    & \sum_{d : D} P(d) \ar[d] \\
              B \ar[r]_f                                & D.
            } \]
If the right arrow is $n$-truncated, then its fibers $P(d)$ are $n$-types for all $d : D$, but this implies that fibers of the left arrow are also $n$-types; hence, it is also $n$-truncated.
\end{proof}

\begin{lem}[trunc-total]
Let $B$ and $C$ be base types over $x : A$.
Then a function $f : B \to C$ is $n$-truncated if and only if the induced function $f' : \Sigma_{x : A} B \to \Sigma_{x : A} C$ is $n$-truncated.
\end{lem}
\begin{proof}
By \cite[Theorem~4.7.6]{hottbook}, the fiber of $f'$ over a point $(x,c)$ is equivalent to the fiber of $f$ over $c$.
Thus, fibers of $f'$ are $n$-truncated if and only if fibers of $f$ are $n$-truncated.
\end{proof}

The following lemma is similar to \cite[5.5.6.15]{lurie-topos}, which is proved in the context of $\infty$-categories.

\begin{lem}[trunc-id]
A map $f : \Hom(A,B)$ is $n$-truncated if and only if the map $\langle \id_A, \id_A \rangle : \Hom(A, A \times_B A)$ is $(n-1)$-truncated.
\end{lem}
\begin{proof}
Since $\Hom(X,-)$ preserves and reflects $n$-truncatedness and pullbacks, it is enough to prove this fact for base types.
By \cite[Lemma~7.6.2]{hottbook}, a function $f : A \to B$ is $n$-truncated if and only if, for all $a,a' : A$, the function $\pmap(f,-) : \Id(a,a') \to \Id(f\,a,f\,a')$ is $(n-1)$-truncated.
By \rlem{trunc-total}, the latter function is $n$-truncated if and only if the induced function $\Sigma_{a : A} \Sigma_{a' : A} \Id(a,a') \to \Sigma_{a : A} \Sigma_{a' : A} \Id(f\,a,f\,a')$ is $(n-1)$-truncated.
The latter function is equivalent to $\lambda a.\,(a,a,\refl) : A \to \Sigma_{a : A} \Sigma_{a' : A} \Id(f\,a,f\,a')$, which is equivalent to $\langle \id_A, \id_A \rangle : \Hom(A, A \times_B A)$.
\end{proof}

The last lemma implies that if the theory has pullbacks, then we can define a predicate on maps that corresponds to the notion of $n$-truncated maps by induction on $n$.

\subsection{Fibrations}

Suppose that we have a class of families of propositions over all indexed morphisms:
\begin{center}
\AxiomC{$\Gamma \mid \cdot \vdash A \ob$}
\AxiomC{$\Gamma \mid \cdot \vdash B \ob$}
\AxiomC{$\Gamma \vdash f : \Hom(A,B)$}
\TrinaryInfC{$\Gamma \vdash \Fib(f) \type$}
\DisplayProof
\end{center}
We will call maps $f$ together with an element of $\Fib(f)$ \emph{fibrations} and denote them by $\twoheadrightarrow$.
We will assume that $\Fib$ is closed under equivalences, that is if $f : \Hom(A,B)$ is a fibration and $e_1 : \Hom(A',A)$ and $e_2 : \Hom(B,B')$ are equivalences, then $e_2 \circ f \circ e_1$ is a fibration.

Sometimes $\Fib(f)$ is not a type, but a finite number of judgments of the form $\Gamma, \Delta_i \vdash A_i \type$.
For example, we might want to define $\Fib(f)$ as $\fs{isEquiv}(C(f))$ for some morphism $C(f)$.
In general, this is not a type, but a collection of four judgments.
If the base theory has $\Pi$-types, then we can always replace such a collection of judgments with a single type.
Even if the base theory does not have $\Pi$-types, we still can work with such definitions of $\Fib(f)$;
we just need to replace judgments of the form $\Gamma \vdash b : \Fib(f)$ with a finite collection of judgments of the form $\Gamma, \Delta_i \vdash a_i : A_i$.
For notational convenience, we will always assume that $\Fib(f)$ is a single type.

We can also define the dependent version of classes of fibrations:
\begin{center}
\AxiomC{$\Gamma \mid \Delta \vdash B \ob$}
\UnaryInfC{$\Gamma \vdash \Fib(\Delta.B) \type$}
\DisplayProof
\end{center}
We will also call dependent types $B$ together with an element of $\Fib(\Delta.B)$ fibrations.
It is often more convenient to work with the dependent version of this definition.
If the indexed theory has $\Sigma$-types and unit types, then, for every class of fibrations $\Fib$, we can define its dependent version as follows:
\[ \Fib(\Delta.B) = \Fib(\pi_1 : \Hom(\sum_{p : \Sigma(\Delta)} B[\pi_1(p)/x_1, \ldots \pi_n(p)/x_n],\Sigma(\Delta))). \]
Conversely, if the indexed theory also has identity types, then, for every dependent class of fibrations $\Fib$, we can define its non-dependent version:
\[ \Fib(f : \Hom(A,B)) = \Fib((y : B).\,\sum_{x : A} \Id(f\,x,y)). \]

\begin{example}
If the indexed theory has finite limits, then we can define a class of fibrations consisting of $n$-truncated maps (or $n$-truncated indexed types) as was explained in the previous subsection.
\end{example}

For every dependent class of fibrations $\Fib$, we can add a new sort of dependent types $\Gamma \mid \Delta \vdash A \fib$ consisting of types satisfying the predicate $\Fib$:
\begin{center}
\AxiomC{$\Gamma \mid \Delta \vdash A \fib$}
\UnaryInfC{$\Gamma \mid \Delta \vdash \El(A) \ob$}
\DisplayProof
\qquad
\AxiomC{$\Gamma \mid \Delta \vdash A \fib$}
\UnaryInfC{$\Gamma \vdash \fs{fp}(\Delta.A) : \Fib(\Delta.\,\El(A))$}
\DisplayProof
\end{center}
\medskip

\begin{center}
\AxiomC{$\Gamma \mid \Delta \vdash A \ob$}
\AxiomC{$\Gamma \vdash p : \Fib(\Delta.A)$}
\AxiomC{$\Gamma \mid E \vdash b_i : B_i[b_1/x_1, \ldots b_{i-1}/x_{i-1}]$}
\TrinaryInfC{$\Gamma \mid E \vdash \fs{rf}(\Delta.A, p, b_1, \ldots b_k) \fib$}
\DisplayProof
\end{center}
where $\Delta = x_1 : B_1, \ldots x_k : B_k$.
\[ \El(\fs{rf}(\Delta.A, p, b_1, \ldots b_k)) = A[b_1/x_1, \ldots b_k/x_k]. \]
We define equivalences between fibrations $\Gamma \mid \Delta \vdash A \fib$ and $\Gamma \mid \Delta \vdash B \fib$ as equivalences between underlying types $\El(A)$ and $\El(B)$.
We will often omit the function symbol $\El$.

We can assume various closure conditions on the class of fibrations in the usual way.
For example, we can assume that $\Fib$ is closed under contractible types.
This is true if and only if it contains all identity morphisms.
Similarly, $\Fib$ is closed under $n$-types (as a dependent class) if and only if it contains all $n$-truncated maps (as a non-dependent class).

\begin{prop}[fib-sigma]
A class of fibrations is closed under $\Sigma$-types if and only if the corresponding non-dependent class is closed under compositions.
\end{prop}
\begin{proof}
First, suppose that $\Fib$ is closed under compositions.
If we have fibrations $\Gamma \mid \Delta \vdash A$ and $\Gamma \mid \Delta, x : A \vdash B$, then the type $\Sigma_{x : A} B$ corresponds to the following map:
\[ \Sigma(\Delta, A, B) \xrightarrow{\simeq} \sum_{p : \Sigma(\Delta, A)} B[\overline{\pi_i(p)/x_i}] \xrightarrow{\pi_1} \Sigma(\Delta, A) \xrightarrow{\simeq} \sum_{p : \Sigma(\Delta)} A[\overline{\pi_i(p)/x_i}] \xrightarrow{\pi_1} \Sigma(\Delta). \]
The first and the third maps are equivalences and the second and the fourth maps are fibrations by assumption.
Thus, the type $\Sigma_{x : A} B$ is also a fibration.

Now, suppose that $\Fib$ is closed under $\Sigma$-types.
Let $f : \Hom(A,B)$ and $g : \Hom(B,C)$ be fibrations.
These maps correspond to the types $y : B \vdash \Sigma_{x : A} \Id(f\,x,y)$ and $z : C \vdash \Sigma_{y : B} \Id(g\,y,z)$.
The first type is equivalent to the type $z : C, p : \Sigma_{y : B} \Id(g\,y,z) \vdash \Sigma_{x : A} \Id(f\,x,\pi_1(p))$.
Since fibrations are closed under $\Sigma$-types, the type $z : C \vdash \Sigma_{(p : \Sigma_{y : B} \Id(g y, z))} \sum_{x : A} \Id(f\,x,\pi_1(p))$ is a fibration.
This type corresponds to the following map:
\[ \pi_1 : \Hom(\sum_{z : C} \sum_{(p : \sum_{y : B} \Id(g y, z))} \sum_{x : A} \Id(f\,x,\pi_1(p)), C). \]
Since this map is equivalent to $g \circ f$, the composite is a fibration.
\end{proof}

\begin{prop}[fib-id]
Let $\Fib$ be a class of fibrations closed under pullbacks.
Then it is closed under identity types if and only if, for every fibration $p : \Hom(A,B)$, the map $\langle \id_A, \id_A \rangle : \Hom(A, A \times_B A)$ is also a fibration.
\end{prop}
\begin{proof}
First, suppose that $\Fib$ is closed under identity types.
Let $a_1 : A, a_2 : A, h : \Id(p\,a_1,p\,a_2)$ be an element of $A \times_B A$.
The fiber over this element is the following type:
\begin{align*}
\sum_{a : A} \sum_{h_1 : \Id(a_1,a)} \sum_{h_2 : \Id(a,a_2)} \Id(\pmap(p, h_1 \ct h_2), h) & \simeq \\
\sum_{h' : \Id(a_1,a_2)} \Id(\pmap(p,h'),h) & \simeq \\
\Id_{\sum_{a : A} \Id(p a_1, p a)}((a_1,\refl), (a_2,h)) & .
\end{align*}
Since $\Fib$ is closed under identity types, it is enough to show that $\sum_{a : A} \Id(p\,a_1,p\,a)$ is a fibration over $a_1,a_2,h$.
This follows from the fact that this type is a pullback of the type $\sum_{a : A} \Id(b,p\,a)$ over $b : B$ which is a fibration since it is the fiber of $p$ over $b$.

Now, let us prove the converse.
Let $\Gamma \mid \Delta \vdash c_1 : C$ and $\Gamma \mid \Delta \vdash c_2 : C$ be a pair of terms.
Then the following square is a pullback:
\[ \xymatrix{ \Sigma_{p : \Sigma(\Delta)} \Id(c_1',c_2') \ar[r]^-{c_1' \circ \pi_1} \ar@{->>}[d]_{\pi_1} \pb    & C' \ar@{->>}[d]^{\langle \id_{C'}, \id_{C'} \rangle} \\
              \Sigma(\Delta) \ar[r]_-{\langle c_1', c_2' \rangle}                                               & C' \times_{\Sigma(\Delta)} C'
            } \]
where $C' = \Sigma_{p : \Sigma(\Delta)} C[\overline{\pi_i(p)/x_i}]$ and $c_i' = c_i[\overline{\pi_i(p)/x_i}]$.
Since $\Gamma \mid \Delta \vdash \Id(c_1,c_2)$ is equivalent to $\Gamma \mid p : \Sigma(\Delta) \vdash \Id(c_1',c_2')$, it follows that $\Id(c_1,c_2)$ is a fibration.
\end{proof}

The following proposition shows that if the class of fibrations is closed under pullbacks and compositions, then there is another characterization of the condition that it is closed under identity types.

\begin{prop}[fib-id-comp]
Let $\Fib$ be a class of fibrations in an indexed unary type theory.
If $\Fib$ is closed under pullbacks and compositions, then the following conditions are equivalent:
\begin{enumerate}
\item \label{it:fib-pb} For every fibration $p : \Hom(A,B)$, the map $\langle \id_A, \id_A \rangle : \Hom(A, A \times_B A)$ is also a fibration.
\item \label{it:fib-over} For every commutative diagram as below in which $p$ and $q$ are fibrations, $f$ is also a fibration.
\[ \xymatrix{ A \ar[rr]^f \ar@{->>}[dr]_p &   & C \ar@{->>}[dl]^q \\
                                          & B &
            }\]
\end{enumerate}
\end{prop}
\begin{proof}
First, suppose that \eqref{it:fib-pb} holds.
Consider the following diagram:
\[ \xymatrix{ A \ar[r]^f \ar@{->>}[d]_r \pb             & C \ar@{->>}[d]^{\langle \id_C, \id_C \rangle} & \\
              A \times_B C \ar[r]^{f'} \ar@{->>}[d] \pb & C \times_B C \ar[r]^-{q'} \ar@{->>}[d] \pb    & C \ar@{->>}[d]^q \\
              A \ar[r]_f                                & C \ar[r]_q                                    & B
            } \]
The map $q' \circ f'$ is a pullback of $q \circ f = p$, so it is a fibration.
The map $r$ is a fibration since it is a pullback of $\langle \id_C, \id_C \rangle$, which is a fibration by \eqref{it:fib-pb}.
Since fibrations are closed under compositions, $q' \circ f' \circ r$ is also a fibration.
Finally, $f$ is a fibration since $f \sim q' \circ \langle \id_C, \id_C \rangle \circ f \sim q' \circ f' \circ r$.

Now, suppose that \eqref{it:fib-over} holds.
Let $f : \Hom(A,B)$ be a fibration.
Let $q$ be the composite $A \times_B A \xrightarrow{\pi_1} A \xrightarrow{f} B$.
Since fibrations are closed under pullbacks, the first map is a fibration.
Since they are closed under compositions, $q$ is also a fibration.
Since $q \circ \langle \id_A, \id_A \rangle = f$ is a fibration, $\langle \id_A, \id_A \rangle$ is also a fibration by \eqref{it:fib-over}.
\end{proof}

\begin{lem}[fib-id-idm]
Let $\Fib$ be a class of fibrations in an indexed unary type theory.
If $\Fib$ is closed under pullbacks and contains all identity morphisms, then \eqref{it:fib-over} implies \eqref{it:fib-pb}.
\end{lem}
\begin{proof}
Let $f : \Hom(A,B)$ be a fibration.
Since $\pi_1 : \Hom(A \times_B A, A)$ is a pullback of $f$, it is also a fibration.
Moreover, $\id_A$ is a fibration by assumption.
Now, the claim follows from the fact that $\pi_1 \circ \langle \id_A, \id_A \rangle \sim \id_A$.
\end{proof}

\begin{example}
The class of $n$-truncated maps is closed under compositions, pullbacks, $m$-types for $m \leq n$, $\Sigma$-types, and identity types.
This follows from \rlem{trunc-pb}, \rprop{fib-sigma}, \rprop{fib-id}, and \rlem{trunc-id}.
\end{example}

An \emph{external universe} is a base type $\mathcal{U}(\Delta)$ defined for all indexed context $\Delta$ together with an indexed type $\El(c)$ over $\Delta$ for all $c : \mathcal{U}(\Delta)$:
\begin{center}
\AxiomC{$\Gamma \vdash c : \mathcal{U}(\Delta)$}
\UnaryInfC{$\Gamma \mid \Delta \vdash \El(c) \ob$}
\DisplayProof
\end{center}
We have a function from $\Id_{\mathcal{U}(\Delta)}(c,c')$ to the type of equivalences between $\El(c)$ and $\El(c')$ over $\Delta$ defined as the transport of the identity morphism along the homotopy $\Id_{\mathcal{U}(\Delta)}(c,c')$.
We will say that the universe $\mathcal{U}(\Delta)$ is \emph{univalent} if this function is an equivalence.
We will say that a universe \emph{classifies} fibrations $\Fib$ if it is univalent, the family $\El(c)$ satisfies $\Fib$ for all $c$, and every fibration is equivalent to a fibration of the form $\El(c)$ for some $c$.

If a universe classifying $\Fib$ exists, then it is unique up to a canonical equivalence.
Indeed, let $\mathcal{U}$ and $\mathcal{U}'$ be two universes classifying $\Fib$.
Then, for every indexed context $\Delta$ over a base context $\Gamma$, we have a fibration $\El(x)$ over $\Delta$ in the context $\Gamma, x : \mathcal{U}$.
This fibration is equivalent to a fibration of the form $\El'(c')$ for some $\Gamma, x : \mathcal{U} \vdash c' : \mathcal{U}'$.
Thus, we have a map $\lambda x.\,c' : \mathcal{U} \to \mathcal{U}'$.
Similarly, we have a map $\lambda y.\,c : \mathcal{U}' \to \mathcal{U}$.
We need to construct a homotopy $\Gamma, x : \mathcal{U} \vdash h : \Id(c[c'/y],x)$.
By univalence, it is enough to prove that $\El(c[c'/y])$ is equivalent to $\El(x)$ over $\Delta$.
By definition of $c$, we have an equivalence $\Gamma, y : \mathcal{U}' \mid \Delta \vdash e : \El(c) \simeq \El'(y)$.
It follows that we have the following equivalence: $\Gamma, x : \mathcal{U} \mid \Delta \vdash e[c'/y] : \El(c[c'/y]) \simeq \El'(c')$.
By definition of $c'$, $\El'(c')$ is equivalent to $\El(x)$ over $\Delta$.
The composition of these two equivalences gives us the required equivalence between $\El(c[c'/y])$ and $\El(x)$ over $\Delta$.
A homotopy between $c'[c/x]$ and $y$ is constructed similarly.

\begin{defn}
We will say that a class $\Fib$ is \emph{locally small} if it is classified by an external universe.
\end{defn}

\begin{example}
An indexed type theory is called \emph{well-powered} if the class of monomorphisms (that is, $(-1)$-truncated maps) is locally small.
This definition is analogous to the definition of well-powered indexed categories (see \cite[Example~B1.3.14]{elephant}).
\end{example}

\subsection{Object classifiers}

Let $p : \Hom(\widehat{\mathcal{U}},\mathcal{U})$ be a map in an indexed unary type theory such that its pullbacks along all maps exist.
Then we can define a map from $\Id_{\Hom(\Delta,\mathcal{U})}(f,g)$ to the type of equivalences over $\Delta$ between pullbacks of $p$ along $f$ and $g$ as the transport of the identity map along the homotopy between $f$ and $g$.
An \emph{object classifier} is a map $p : \Hom(\widehat{\mathcal{U}},\mathcal{U})$ such that its pullbacks exist and the map defined above is an equivalence.
We will say that an object classifier \emph{classifies} a class of fibrations if this class is closed under pullbacks, $p$ is a fibration, and every fibration is a pullback of $p$.

\begin{example}
A subobject classifier is an object classifier for monomorphisms.
\end{example}

Let $p : \Hom(\widehat{\mathcal{U}},\mathcal{U})$ be any map.
We can think of such a map as a (non-univalent) universe.
We will say that the universe $\mathcal{U}$ \emph{contains a type $A$} if there is a map $a : \Hom(1,\mathcal{U})$ and an equivalence between $A$ and the pullback of $p$ along $a$.
We will say that $\mathcal{U}$ is \emph{closed under coproducts} if, for all maps $a,b : \Hom(\Delta,\mathcal{U})$, there is a map $a + b : \Hom(\Delta,\mathcal{U})$
and an equivalence over $\Delta$ between the pullback of $p$ along $a + b$ and the sum of pullbacks of $p$ along $a$ and $b$.
Similarly, we will say that $\mathcal{U}$ is \emph{closed under $\Sigma$-types} if, for every map $a : \Hom(\Delta,\mathcal{U})$
and every map $b : \Hom(\Delta \times_\mathcal{U} \widehat{\mathcal{U}}, \mathcal{U})$, there is a map $\Sigma(a,b) : \Hom(\Delta,\mathcal{U})$
together with an equivalence over $\Delta$ between the pullback of $p$ along $\Sigma(a,b)$ and the composition of pullbacks of $p$ along $b$ and $a$.
The closure under other constructions is defined similarly.

In general, being closed under different construction is not a property of a map but additional data on it.
The following proposition shows that it is a property if the map $p$ is an object classifier:

\begin{prop}
Let $p : \Hom(\widehat{\mathcal{U}},\mathcal{U})$ be an object classifier.
Then the types corresponding to the closure conditions listed above are propositions.
\end{prop}
\begin{proof}
Such types consist of a map $c : \Hom(\Delta,\mathcal{U})$ for some fixed type $\Delta$ together with an equivalence between the fiber of $p$ over $c$ and some fixed type over $\Delta$.
Let $(c_1,e_1)$ and $(c_2,e_2)$ be two such pairs.
Since $p$ is an object classifier, to define a homotopy between these pairs, it is enough to define a homotopy $e$ between fibers of $p$ over $c_1$ and $c_2$ together with a homotopy between $e \circ e_1$ and $e_2$.
We can define $e$ as $e_2 \circ e_1^{-1}$.
\end{proof}

An \emph{(internal) universe} in an indexed dependent type theory is an indexed type $\mathcal{U}$ together with an indexed type $\El(c)$ for all $c : \mathcal{U}$:
\begin{center}
\AxiomC{}
\UnaryInfC{$\Gamma \mid \Delta \vdash \mathcal{U} \ob$}
\DisplayProof
\qquad
\AxiomC{$\Gamma \mid \Delta \vdash c : \mathcal{U}$}
\UnaryInfC{$\Gamma \mid \Delta \vdash \El(c) \ob$}
\DisplayProof
\end{center}
An internal universe is \emph{univalent} if the obvious map from $\Id_{\mathcal{U}}(c,c')$ to the type of equivalences between $\El(c)$ and $\El(c')$ is an equivalence.
Internal universes in a theory with $\Sigma$-types and identity types correspond to maps $\Hom(\widehat{\mathcal{U}},\mathcal{U})$ via the construction $\pi_1 : \Hom(\Sigma_{x : \mathcal{U}} \El(x), \mathcal{U})$.
Such a map is an object classifier if and only if the corresponding universe is univalent.

Every internal universe $\mathcal{U}$ gives rise to an external one, namely $\Hom(\Delta.\mathcal{U})$.
An internal universe is univalent if and only if the corresponding external one is.
Moreover, a class of fibration is classified by an internal universe if and only if it is classified by the corresponding external one.
An internal universe classifying a given class of fibrations is also unique up to a canonical equivalence.
The proof is the same as for external universes.

A universe $\mathcal{U}$ is \emph{weakly (resp., strictly) contains a type $A$} if there is an element $a : \mathcal{U}$ such that $\El(a)$ is propositionally (resp., judgmentally) equivalent to $A$. 
A universe $\mathcal{U}$ is \emph{weakly (resp., strictly) closed under coproducts} if, for all elements $a,b : \mathcal{U}$, there exists an element $a + b : \mathcal{U}$ such that $\El(a + b)$ is propositionally (resp., judgmentally) equivalent to the coproduct of $\El(a)$ and $\El(b)$.
A universe $\mathcal{U}$ is \emph{weakly (resp., strictly) closed under $\Sigma$-types} if, for every element $a : \mathcal{U}$ and every function $b : \El(a) \to \mathcal{U}$, there exists an element $\Sigma(a,x.b(x)) : \mathcal{U}$ such that $\El(\Sigma(a,x.b(x)))$ is propositionally (resp., judgmentally) equivalent to $\Sigma_{x : \El(a)} \El(b(x))$.

\begin{prop}
A universe $\mathcal{U}$ is weakly closed under one of the constructions listed above if and only if the map $\pi_1 : \Hom(\Sigma_{A : \mathcal{U}} \El(A), \mathcal{U})$ is closed under this construction.
\end{prop}
\begin{proof}
This follows from the fact that elements of $\mathcal{U}$ in a context $\Delta$ correspond to maps $\Hom(\Delta, \mathcal{U})$ and
types of the form $\El(c)$ over $\Delta$ correspond to pullbacks of $\pi_1 : \Hom(\Sigma_{A : \mathcal{U}} \El(A), \mathcal{U})$ along $c$.
\end{proof}

The following propositions discuss $n$-truncated object classifiers.

\begin{prop}
If $p : \Hom(\widehat{\mathcal{U}},\mathcal{U})$ is an $n$-truncated map which is also an object classifier, then $\mathcal{U}$ and $\widehat{\mathcal{U}}$ are $(n+1)$-truncated.
\end{prop}
\begin{proof}
First, let us prove that $\mathcal{U}$ is $(n+1)$-truncated.
This is true if and only if $\Hom(B,\mathcal{U})$ is $(n+1)$-truncated for all $B$, which is true if and only if the type $\Id(f,f')$ is $n$-truncated for all $f,f' : \Hom(B,\mathcal{U})$.
Since $p$ is an object classifier, the type $\Id(f,f')$ is equivalent to the type of equivalences over $B$ between pullbacks of $p$ along $f$ and $f'$.
By \rlem{trunc-pb}, pullbacks of $p$ are $n$-truncated.
Since the type of equivalences over $B$ is embedded into the type of maps over $B$, we just need to prove that the type of such maps between $n$-truncated maps is $n$-truncated.

Let $s : \Hom(E,B)$ and $s' : \Hom(E',B)$ be $n$-truncated maps.
The type of maps over $B$ is defined as $\Sigma_{f : \Hom(E,E')} \Id(s' \circ f, s)$.
This type is the fiber of $s' \circ -$ over $s$, which is $n$-truncated since $s' \circ -$ is $n$-truncated.

Finally, since both $p$ and $\mathcal{U}$ are $(n+1)$-truncated, $\widehat{\mathcal{U}}$ is also $(n+1)$-truncated.
\end{proof}

\begin{prop}
If $p : \Hom(\widehat{\mathcal{U}},\mathcal{U})$ is an object classifier which is also a monomorphism, then $\widehat{\mathcal{U}}$ is subterminal.
It is terminal if and only if identity morphisms are classified by $p$.
\end{prop}
\begin{proof}
Any commutative square of the form
\[ \xymatrix{ B \ar[r]^s \ar[d]_{\id_B} & \widehat{\mathcal{U}} \ar[d]^p \\
              B \ar[r]_t                & \mathcal{U}
            } \]
is a pullback.
To prove this, we need to show that the canonical map
\[ r : \Hom(X,B) \to \sum_{f : \Hom(X,B)} \sum_{g : \Hom(X,\widehat{\mathcal{U}})} \Id(t \circ f, p \circ g). \]
is an equivalence for all $X$.
Since $t$ is homotopic to $p \circ s$, the type $\Id(t \circ f, p \circ g)$ is equivalent to $\Id(p \circ s \circ f, p \circ g)$.
Since $p$ is a monomorphism, it is equivalent to $\Id(s \circ f, g)$.
Since the type $\Sigma_{g : \Hom(X,\widehat{\mathcal{U}})} \Id(s \circ f, g)$ is contractible, $r$ is indeed an equivalence.

To prove that $\widehat{\mathcal{U}}$ is subterminal, we need to show that any two maps $f_1,f_2 : \Hom(B,\widehat{\mathcal{U}})$ are homotopic.
Since $p$ is a monomorphism, it is enough to construct a homotopy between $p \circ f_1$ and $p \circ f_2$.
We have two pullback squares as above with $s = f_i$ and $t = p \circ f_i$.
Since $p$ is an object classifier, we have an equivalence between $\Id(p \circ f_1, p \circ f_2)$ and the type of equivalences between pullbacks of $p$ along $p \circ f_1$ and $p \circ f_2$.
Since both pullbacks are just $\id_B$, they are equivalent; so, we have a homotopy between $p \circ f_1$ and $p \circ f_2$.

If $p$ classifies identity morphisms, then, for every $B$, the map $\id_B$ is a pullback of $p$.
In particular, there exists a map from $B$ to $\widehat{\mathcal{U}}$.
Thus, $\widehat{\mathcal{U}}$ is terminal.
Conversely, if $\widehat{\mathcal{U}}$ is terminal, then, for every type $B$, we have a commutative square as depicted at the beginning of the proof.
Since this square is a pullback, $p$ classifies $\id_B$.
\end{proof}

Finally, let us prove another simple but useful result.
Analogous result in the context of higher categories was proved in \cite[Theorem~3.28]{rasekh-eht}.

\begin{prop}
Let $p : \Hom(\widehat{\mathcal{U}},\mathcal{U})$ be an object classifier and let $f : \Hom(\mathcal{U}',\mathcal{U})$ be any map.
Then the pullback of $p$ along $f$ is an object classifier if and only if $f$ is a monomorphism.
\end{prop}
\begin{proof}
Let us denote the pullback of $p$ along $f$ by $p' : \Hom(\widehat{\mathcal{U}'},\mathcal{U}')$.
Then the pullback of $p'$ along a map $g : \Hom(\Delta,\mathcal{U}')$ is equivalent to the pullback of $p$ along $f \circ g$.
Thus, we have an equivalence between the type of equivalences between $g_1^*(p')$ and $g_2^*(p')$ and the type of equivalences between $(f \circ g_1)^*(p)$ and $(f \circ g_2)^*(p)$.
Since $p$ is an object classifier, the latter type is equivalent to $\Id(f \circ g_1, f \circ g_2)$.
Thus, $p'$ is an object classifier if and only if the canonical function $\Id(g_1,g_2) \to \Id(f \circ g_1, f \circ g_2)$ is an equivalence.
This function maps $\refl(g)$ to $\refl(f \circ g)$.
This implies that it is homotopic to $\lambda h.\,\pmap(f \circ -, h)$, but this map is an equivalence if and only if $f$ is a monomorphism.
\end{proof}

\section{Locally reflective classes of fibrations}
\label{sec:refl-fib}

In this section, we discuss the notion of modalities in indexed type theories.
Several equivalent definitions of modalities were defined in \cite{modality-hott} in the context of ordinary homotopy type theory.
We define the notion of locally reflective classes of fibrations which is similar to the notion of a reflective subuniverse.
This definition makes sense in an indexed unary type theory.
We also define a dependent version of this notion which is similar to the notion of a higher modality.

\subsection{Locally reflective classes in unary theories}

Let $\Fib$ be a class of fibrations in an indexed unary type theory as defined in the previous section.
We will say that it is \emph{locally reflective} if every map $f : \Hom(A,B)$ factors through a fibration $p : \Hom(C,B)$
such that, for every factorization of $f$ through any fibration $p' : \Hom(C',B)$, the type of lifts in the following square is contractible:
\[ \xymatrix{ A \ar[r] \ar[d]                   & C' \ar@{->>}[d]^{p'} \\
              C \ar@{->>}[r]_p \ar@{-->}[ur]    & B
            } \]
The factorization $A \to C \twoheadrightarrow B$ will be called \emph{the universal factorization} of $f$.
We will say that $\Fib$ is \emph{stably} locally reflective if the universal factorization of any map is stable under pullbacks.

\begin{lem}[fib-refl]
If $A \xrightarrow{i} C \overset{p}\twoheadrightarrow B$ is the universal factorization of $f$, then the following conditions are equivalent:
\begin{enumerate}
\item $i$ is an equivalence.
\item $f$ is a fibration.
\item $i$ has a retraction over $B$.
\end{enumerate}
\end{lem}
\begin{proof}
If $i$ is an equivalence, then $f$ is a fibration since fibrations are closed under equivalences and $p$ is a fibration.
If $f$ is a fibration, then the lift in the following square is a retraction of $i$ over $B$.
\[ \xymatrix{ A \ar@{=}[r] \ar[d]_i             & A \ar@{->>}[d]^{f} \\
              C \ar@{->>}[r]_p \ar@{-->}[ur]^r  & B
            } \]
Let $r$ be a retraction of $i$ over $B$.
Consider the following commutative square:
\[ \xymatrix{ A \ar[r]^i \ar[d]_i                                                           & C \ar@{->>}[d]^{p} \\
              C \ar@{->>}[r]_p \ar@{-->}@<-0.5ex>[ur]_{i \circ r} \ar@{-->}@<0.5ex>[ur]^\id & B
            } \]
It is easy to see that $\id_C$ and $i \circ r$ are lifts in this square.
Since lifts are unique up to a homotopy, these maps are homotopic.
Thus, $i$ is an equivalence.
\end{proof}

\begin{lem}[fib-idm]
Any locally reflective class of fibrations contains all identity morphisms
\end{lem}
\begin{proof}
Let $A \xrightarrow{i} B \overset{p}\twoheadrightarrow A$ be the universal factorization of $\id_A$.
Then $p$ is a retraction of $i$ over $A$.
By \rlem{fib-refl}, $\id_A$ is a fibration.
\end{proof}

\begin{lem}[fib-pullback]
Any stably locally reflective class of fibrations is closed under pullbacks.
\end{lem}
\begin{proof}
Let $f$ be a fibration and let $f = p \circ i$ be its universal factorization.
Since the class of fibrations is stably locally reflective, the pullbacks of $p$ and $i$ constitute the universal factorization of a pullback of $f$.
By \rlem{fib-refl}, $i$ is an equivalence.
Hence, its pullback is also an equivalence.
Since the pullback of $p$ is a fibration and fibrations are closed under equivalences, the pullback of $f$ is also a fibration.
\end{proof}

\begin{lem}[pullback-lift]
Suppose that we have the following diagram, where the right square is a pullback.
\[ \xymatrix{ A \ar[r]^{c} \ar[d]_i & C \ar[r]^e \ar[d]_p \pb   & E \ar[d]^q \\
              B \ar[r]_d            & D \ar[r]_f                & F
            } \]
Then the type of lifts in the left square is equivalent to the type of lifts in the outer rectangle.
\end{lem}
\begin{proof}
Let $H_1$ and $H_2$ be the homotopies witnessing the commutativity of the left and right square, respectively.
By the universal property of pullbacks, the type of lifts in the left square is equivalent to the following type:
\begin{align*}
& \sum_{r_1 : \Hom(B,D)} \sum_{r_2 : \Hom(B,E)} \sum_{r_3 : \Id(f \circ r_1, q \circ r_2)} \sum_{h : \Id(d,r_1)} \\
& \sum_{h_1 : \Id(r_1 \circ i, p \circ c)} \sum_{h_2 : \Id(e \circ c, r_2 \circ i)} \sum_{h_3 : \Id((h_1 * f) \ct (c * H_2) \ct (h_2 * q), i * r_3)} \Id(h_1,h_*(H_1)).
\end{align*}
After reducing $r_1$, $h$, $h_1$, and the last homotopy we get the following equivalent type:
\[ \sum_{r_2 : \Hom(B,E)} \sum_{r_3 : \Id(f \circ d, q \circ r_2)} \sum_{h_2 : \Id(e \circ c, r_2 \circ i)} \Id((H_1 * f) \ct (c * H_2) \ct (h_2 * q), i * r_3). \]
This type is equivalent to the type of lifts in the outer rectangle.
\end{proof}

\begin{lem}[fib-refl-lift]
Let $\Fib$ be a locally reflective class of fibrations closed under pullbacks.
Let $A \xrightarrow{i} C \overset{p}\twoheadrightarrow B$ be the universal factorization of a map $f : \Hom(A,B)$.
Then the type of lifts in every commutative square as below is contractible if $v$ factors through $p$.
\[ \xymatrix{ A \ar[r] \ar[d]_i & D \ar@{->>}[d] \\
              C \ar[r]_v        & E
            } \]
\end{lem}
\begin{proof}
By assumption, $v$ equals to $C \overset{p}\twoheadrightarrow B \xrightarrow{u} E$ for some map $u$.
Consider the following diagram:
\[ \xymatrix{ A \ar[r] \ar[d]_i & C' \ar[r] \ar@{->>}[d] \pb    & D \ar@{->>}[d] \\
              C \ar@{->>}[r]_p  & B \ar[r]_u                    & E
            } \]
The type of lift in the left square is contractible and \rlem{pullback-lift} implies that this type is equivalent to the type of lifts in the original square.
\end{proof}

\begin{prop}[fib-refl-id]
Any locally reflective class of fibrations closed under pullbacks satisfies the equivalent conditions of \rprop{fib-id}.
\end{prop}
\begin{proof}
By \rlem{fib-id-idm} and \rlem{fib-idm}, it is enough to prove condition~\eqref{it:fib-over} of \rprop{fib-id-comp}.
Let $f : \Hom(A,D)$ be a map and let $q : \Hom(D,B)$ be a fibration such that $q \circ f$ is also a fibration.
We need to prove that $f$ is a fibration.
Let $A \xrightarrow{i} C \overset{p}\twoheadrightarrow D$ be the universal factorization of $f$.
By \rlem{fib-refl}, it is enough to show that $i$ has a retraction over $D$.
By \rlem{fib-refl-lift}, we have a lift in the following square:
\[ \xymatrix{ A \ar@{=}[r] \ar[d]_i                 & A \ar@{->>}[d]^{q \circ f} \\
              C \ar[r]_{q \circ p} \ar@{-->}[ur]^r  & B
            } \]
Both maps $p$ and $f \circ r$ are lifts in the following square:
\[ \xymatrix{ A \ar[r]^f \ar[d]_i   & D \ar@{->>}[d]^q \\
              C \ar[r]_{q \circ p}  & B
            } \]
By \rlem{fib-refl-lift}, we have a homotopy $h$ between $p$ and $f \circ r$ such that $i * h$ is homotopic to the canonical homotopy between $p \circ i$ and $f \circ r \circ i$.
This implies that $r$ is a retraction of $i$ over $D$.
\end{proof}

\subsection{Orthogonal factorization systems}

In this subsection, we defined connected maps and orthogonal factorization systems and prove that they are equivalent to locally reflective classes of fibrations.
Similar equivalence was proved in \cite{modality-hott}, but our proof is more general since it applies to any (indexed) unary theory.
Moreover, the proof in \cite{modality-hott} applies only to stable orthogonal factorization systems and stably locally reflective classes of fibrations.
The reason is that it is done in the internal language of the theory and everything must be stable in this language.
We prove a more general equivalence between (non-stable) orthogonal factorization systems and (non-stable) locally reflective classes of fibrations.

\begin{defn}
Let $\Fib$ be a locally reflective class of fibrations in an indexed unary type theory.
A map $f$ is \emph{connected} if the fibration in the universal factorization of $f$ is an equivalence.
\end{defn}

\begin{lem}[uni-conn]
Let $\Fib$ be a locally reflective class of fibrations closed under compositions.
If $A \xrightarrow{i} B \overset{p}\twoheadrightarrow C$ is the universal factorization of some map, then $i$ is connected.
\end{lem}
\begin{proof}
Let $A \xrightarrow{j} B' \overset{q}\twoheadrightarrow B$ be the universal factorization of $i$.
We need to prove that $q$ is an equivalence.
Since fibrations are closed under compositions, $p \circ q$ is a fibration.
It follows that we have a lift in the following square:
\[ \xymatrix{ A \ar[r]^j \ar[d]_i               & B' \ar@{->>}[d]^{p \circ q} \\
              B \ar@{->>}[r]_p \ar@{-->}[ur]^k  & C
            } \]
Let $h_1 : \Id(j, k \circ i)$ and $h_2 : \Id(p \circ q \circ k, p)$ be the homotopies witnessing the commutativity of triangles in the diagram above.

Let us show that $q \circ k$ is a lift in the following square:
\[ \xymatrix{ A \ar[r]^i \ar[d]_i                   & B \ar@{->>}[d]^p \\
              B \ar@{->>}[r]_p \ar[ur]^{q \circ k}  & C
            } \]
Let $h_0$ be the homotopy between $i$ and $q \circ j$.
The homotopy between $i$ and $q \circ k \circ i$ is defined as $h_0 \ct (h_1 * q)$.
The homotopy between $p \circ q \circ k$ and $p$ is simply $h_2$.
The fact that the combination of these homotopies is homotopic to the trivial homotopy on $p \circ i$ follows from the fact that the combination of $h_1$ and $h_2$ is homotopic to $\sym{h_0} * p$.
Since both $\id_B$ and $q \circ k$ are lifts in the square above, there is a homotopy $h_3$ between them such that $i * h_3$ is homotopic to $h_0 \ct (h_1 * q)$.

Let us show that $k \circ q$ is a lift in the following square:
\[ \xymatrix{ A \ar[r]^j \ar[d]_j                   & B' \ar@{->>}[d]^q \\
              B' \ar@{->>}[r]_q \ar[ur]^{k \circ q} & B
            } \]
The homotopy between $j$ and $k \circ q \circ j$ is defined as $h_1 \ct (h_0 * k)$.
The homotopy between $q \circ k \circ q$ and $q$ is defined as $q * \sym{h_3}$.
The fact that the combination of these homotopies is homotopic to the trivial homotopy on $q \circ j$ follows from the fact that $i * h_3$ is homotopic to $h_0 \ct (h_1 * q)$.
Since both $\id_{B'}$ and $k \circ q$ are lifts in the square above, these maps are homotopic.
It follows that $q$ is an equivalence.
Hence, $i$ is connected.
\end{proof}

\begin{lem}[conn-pullback]
Let $\Fib$ be a stably locally reflective class of fibrations.
Then connected maps are closed under pullbacks.
\end{lem}
\begin{proof}
Let $f : \Hom(A,C)$ be a connected map and let $g : \Hom(D,C)$ be an arbitrary map.
We need to prove that the pullback of $f$ along $g$ is connected.
Let $A \xrightarrow{i} B \overset{p}\twoheadrightarrow C$ be the universal factorization of $f$.
Then we have the following diagram:
\[ \xymatrix{ A' \ar[r] \ar[d] \pb  & A \ar[d]^i \\
              B' \ar[r] \ar[d] \pb  & B \ar@{->>}[d]^p \\
              D  \ar[r]_g           & C
            } \]
Since $\Fib$ is stably locally reflective, $A' \to B' \to D$ is the universal factorization of $A' \to D$.
Since $f$ is connected, $p$ is an equivalence.
Hence, $B' \to D$ is also an equivalence.
Thus, $A' \to D$ is connected.
\end{proof}

\begin{defn}
Let $f : \Hom(A,B)$ and $g : \Hom(C,D)$ be maps in an indexed unary type theory.
We will say that $f$ is \emph{left orthogonal} to $g$ and $g$ is \emph{right orthogonal} to $f$ if the type of lifts in squares of the form depicted below is contractible.
\[ \xymatrix{ A \ar[r] \ar[d]_f         & C \ar[d]^g \\
              B \ar[r] \ar@{-->}[ur]    & D
            } \]
\end{defn}

\begin{lem}[conn-orth]
Let $\Fib$ be a locally reflective class of fibrations closed under pullbacks.
Then connected maps are left orthogonal to fibrations.
\end{lem}
\begin{proof}
Let $i : \Hom(A,B)$ be a connected map and let $p : \Hom(C,D)$ be a fibration.
Consider a commutative square of the following form:
\[ \xymatrix{ A \ar[d]_i \ar[r]^f   & C \ar[d]^p \\
              B \ar[r]_g            & D
            } \]
The type of lifts in this square is
\[ \sum_{r : \Hom(B,C)} \sum_{h_1 : \Id(f, r \circ i)} \sum_{h_2 : \Id(p \circ r, g)} \Id((h_1 * p) \ct (i * h_2), H), \]
where $H$ is the homotopy witnessing the commutativity of the square.
Let us denote this type by $L$.
We need to prove that $L$ is contractible.

Let $A \xrightarrow{i'} B' \overset{q}\twoheadrightarrow B$ be the universal factorization of $i$.
By \rlem{fib-refl-lift}, the type of lifts in the following square is contractible:
\[ \xymatrix{ A \ar[d]_{i'} \ar[rr]^f       &               & C \ar[d]^p \\
              B' \ar[r]_q \ar@{-->}[urr]    & B \ar[r]_g    & D
            } \]
The type of lifts in this square is defined as follows:
\[ \sum_{r' : \Hom(B',C)} \sum_{h_1' : \Id(f, r' \circ i')} \sum_{h_2 : \Id(p \circ r', g \circ q)} \Id((h_1' * p) \ct (i' * h_2'), H \ct (H' * g)), \]
where $H'$ is the homotopy between $i$ and $q \circ i'$.
Let us denote this type by $L'$.
It is enough to prove that $L$ and $L'$ are equivalent.

We have an obvious map $s : L \to L'$ which maps $(r,h_1,h_2,h_3)$ to $(r \circ q, h_1 \ct (H' * r), q * h_2, h_3')$, where $h_3'$ is the following homotopy:
\begin{align*}
((h_1 \ct (H' * r)) * p) \ct (i' * q * h_2) & \sim \\
((h_1 * p) \ct (H' * r * p)) \ct (i' * q * h_2) & \sim \\
(h_1 * p) \ct (H' * h_2) & \sim \\
((h_1 * p) \ct (i * h_2)) \ct (H' * g) & \sim \\
H \ct (H' * g) & ,
\end{align*}
where we use $h_3$ at the last step and other steps are usual interchange laws.

To prove that this map is an equivalence, it is enough to show that it is an equivalence on each component.
Since $i$ is connected, $q$ is an equivalence.
This implies that functions $- \circ q$ and $q * -$ are equivalences and these functions are the first and the third component of $s$, respectively.
The second component of $s$ is $- \ct (H' * r)$, which is also an equivalence.
Finally, the third component of $s$ is an equivalence since it is a function that concatenates its argument with fixed homotopies.
\end{proof}

Let $\mathcal{L}$ and $\mathcal{R}$ be a pair of classes of maps closed under equivalences such that maps in $\mathcal{L}$ are left orthogonal to maps in $\mathcal{R}$.
Then a factorization of a map into a map in $\mathcal{L}$ followed by a map in $\mathcal{R}$ is essentially unique.
The pair $(\mathcal{L},\mathcal{R})$ is called an \emph{orthogonal factorization system} if such a factorization exists for every map.

\begin{lem}[orth-refl]
If $(\mathcal{L},\mathcal{R})$ is an orthogonal factorization system, then $\mathcal{R}$ is a locally reflective class of maps.
Moreover, if $A \xrightarrow{i} B \twoheadrightarrow C$ is the universal factorization of some map, then $i$ belongs to $\mathcal{L}$.
\end{lem}
\begin{proof}
Obviously, any factorization of a map into a map in $\mathcal{L}$ followed by a map in $\mathcal{R}$ is a universal factorization.
The second assertion follows from the facts that the universal factorization is essentially unique and $\mathcal{L}$ is closed under equivalences.
\end{proof}

\begin{prop}
If $(\mathcal{L},\mathcal{R})$ is an orthogonal factorization system, then $\mathcal{L}$ and $\mathcal{R}$ contains all identity morphisms.
\end{prop}
\begin{proof}
By \rlem{orth-refl} and \rlem{fib-idm}, $\mathcal{R}$ contains all identity morphisms.
Since $A \xrightarrow{\id_A} A \xrightarrow{\id_A} A$ is the universal factorization of $\id_A$, \rlem{orth-refl} implies that $\id_A$ belongs to $\mathcal{L}$.
\end{proof}

\begin{prop}[orth-char]
If $(\mathcal{L},\mathcal{R})$ is an orthogonal factorization system, then
$\mathcal{R}$ is precisely the class of maps which are right orthogonal to $\mathcal{L}$ and
$\mathcal{L}$ is precisely the class of maps which are left orthogonal to $\mathcal{R}$.
\end{prop}
\begin{proof}
We prove the first assertion; the other one follows by a dual argument.
Maps in $\mathcal{R}$ are right orthogonal to $\mathcal{L}$ by definition.
Let $f : \Hom(A,C)$ be a map which is right orthogonal to $\mathcal{L}$.
We need to prove that $f$ belongs to $\mathcal{R}$.
By \rlem{orth-refl}, $\mathcal{R}$ is a locally reflective class of maps.
Let $A \xrightarrow{i} B \xrightarrow{p} C$ be the universal factorization of $f$.
By \rlem{orth-refl}, $i$ belongs to $\mathcal{L}$.
Hence, we have a lift in the following square:
\[ \xymatrix{ A \ar@{=}[r] \ar[d]_i     & A \ar[d]^f \\
              B \ar[r]_p \ar@{-->}[ur]  & C
            } \]
By \rlem{fib-refl}, $f$ belongs to $\mathcal{R}$.
\end{proof}

\begin{cor}[orth-unique]
Let $\mathcal{R}$ be a class of maps.
Then a class of maps $\mathcal{L}$ such that $(\mathcal{L},\mathcal{R})$ is an orthogonal factorization system is essentially unique.
That is, if $\mathcal{L}_1$ and $\mathcal{L}_2$ is two such classes, then a map belongs to one of them if and only if it belongs to the other.
Dually, a class of maps $\mathcal{R}$ such that $(\mathcal{L},\mathcal{R})$ is an orthogonal factorization system for a fixed $\mathcal{L}$ is essentially unique.
\end{cor}

\begin{prop}[orth-comp]
If $(\mathcal{L},\mathcal{R})$ is an orthogonal factorization system, then $\mathcal{R}$ and $\mathcal{L}$ are closed under compositions.
\end{prop}
\begin{proof}
We prove this for $\mathcal{R}$; the assertion about $\mathcal{L}$ follows by a dual argument.
Let $f : \Hom(A,B)$ and $g : \Hom(B,C)$ be maps in $\mathcal{R}$.
By \rlem{orth-refl}, there exists a universal factorization $A \xrightarrow{i} D \xrightarrow{p} C$ of $g \circ f$ such that $i \in \mathcal{L}$.
By \rlem{fib-refl}, to prove that $g \circ f$ belongs to $\mathcal{R}$, it is enough to show that $i$ has a retraction over $C$.
Since $i \in \mathcal{L}$, we have two lifts in the following diagram:
\[ \xymatrix{ A \ar@{=}[r] \ar[dd]_i                    & A \ar@{->>}[d]^f \\
                                                        & B \ar@{->>}[d]^g \\
              D \ar[r]_p \ar@{-->}[ur] \ar@{-->}[uur]^r & C
            } \]
Then $r$ is a retraction of $i$ over $C$.
\end{proof}

\begin{prop}[orth-pullback]
If $(\mathcal{L},\mathcal{R})$ is an orthogonal factorization system, then $\mathcal{R}$ is closed under pullbacks and $\mathcal{L}$ is closed under pushouts.
\end{prop}
\begin{proof}
We prove this for $\mathcal{R}$; the assertion about $\mathcal{L}$ follows by a dual argument.
Let $f$ be a pullback of a map $g \in \mathcal{R}$.
Since $g$ is right orthogonal to $\mathcal{L}$, \rlem{pullback-lift} implies that $f$ is also right orthogonal to $\mathcal{L}$.
\rprop{orth-char} implies that $f \in \mathcal{R}$.
\end{proof}

Now, we are ready to prove that orthogonal factorization systems are equivalent to locally reflective classes of maps which are closed under compositions and pullbacks:

\begin{thm}[refl-orth]
Let $\Fib$ be a class of fibrations.
If $\Fib$ is locally reflective and closed under compositions and pullbacks, then $(\mathcal{C},\Fib)$ is an orthogonal factorization system, where $\mathcal{C}$ is the class of connected maps.
The converse is also true in the sense that if $(\mathcal{L},\Fib)$ is an orthogonal factorization system for some class of maps $\mathcal{L}$, then $\Fib$ is locally reflective and closed under compositions and pullbacks.
\end{thm}
\begin{proof}
If $\Fib$ is locally reflective and closed under compositions, then connected maps are left orthogonal to fibrations by \rlem{conn-orth} and the factorization exists by \rlem{uni-conn}.
Thus, $(\mathcal{C},\Fib)$ is an orthogonal factorization system.
Conversely, if we have an orthogonal factorization system $(\mathcal{L},\Fib)$, then $\Fib$ is locally reflective by \rlem{orth-refl} and it is closed under compositions and pullbacks by \rprop{orth-comp} and \rprop{orth-pullback}, respectively.
\end{proof}

\subsection{Locally reflective classes in dependent theories}

We can reformulate the definition of locally reflective class of fibrations in dependent theories with identity types, $\Sigma$-types, and unit types.
Let $\Fib$ be a dependent class of fibrations.
A factorization of a map in a unary theory can be turned into a factorization of a dependent type $\Delta \vdash A \ob$, which consists of a dependent type $\Delta \vdash \| A \| \ob$ and a map $\eta_A : \Hom_\Delta(A, \| A \|)$.
Then $\Fib$ is locally reflective if and only if, for every dependent type $A$, there exists its factorization such that the following function is an equivalence for every fibrant type $\Delta \vdash B \ob$:
\[ \lambda f.\, f \circ \eta_A : \Hom_\Delta(\| A \|, B) \to \Hom_\Delta(A, B). \]
This condition holds if and only if the type of lifts in the following diagram is contractible for every fibrant type $\Delta \vdash B \ob$ and every map $\Hom_\Delta(A,B)$:
\[ \xymatrix{ A \ar[r] \ar[d]_{\eta_A} & B \\
              \| A \| \ar@{-->}[ur]
            } \]

The constructions $\| - \|$ and $\eta$ are unstable, that is they are not stable under substitutions.
An \emph{(unstable) higher modality} consists of a class of fibrations $\Fib$ closed under identity types together with the following unstable rules:
\begin{center}
\AxiomC{$\Gamma \mid \Delta \vdash A \ob$}
\UnaryInfC{$\Gamma \mid \Delta \vdash \| A \| \fib$}
\DisplayProof
\qquad
\AxiomC{$\Gamma \mid \Delta \vdash A \ob$}
\UnaryInfC{$\Gamma \mid \Delta, x : A \vdash \eta_A(x) : \|A\|$}
\DisplayProof
\end{center}
\medskip

\begin{center}
\AxiomC{$\Gamma \mid \Delta, z : \| A \| \vdash D \fib$}
\AxiomC{$\Gamma \mid \Delta, x : A \vdash d : D[\eta_A(x)/z]$}
\BinaryInfC{$\Gamma \mid \Delta, z : \| A \| \vdash \| A \|\text{-}\fs{elim}(z.D, x.d) : D$}
\DisplayProof
\end{center}
\medskip

\begin{center}
\AxiomC{$\Gamma \mid \Delta, z : \| A \| \vdash D \fib$}
\AxiomC{$\Gamma \mid \Delta, x : A \vdash d : D[\eta_A(x)/z]$}
\BinaryInfC{$\Gamma \mid \Delta, x : A \vdash \| A \|\text{-}\fs{elim_h}(z.D, x.d) : \Id(\| A \|\text{-}\fs{elim}(z.D, x.d)[\eta_A(x)/z], d)$}
\DisplayProof
\end{center}

The following proposition shows that unstable higher modalities are the same as locally reflective classes of fibrations:

\begin{prop}[fib-dep-unst]
Let $\Fib$ be a class of fibrations closed under substitutions and $\Sigma$-types and let $\| - \|$, $\eta$ be a pair of unstable constructions as described above.
Then this pair extends to a higher modality if and only if it makes $\Fib$ into a locally reflective class.
\end{prop}
\begin{proof}
First, suppose that $\Fib$ is locally reflective class.
By \rprop{fib-refl-id}, $\Fib$ is closed under identity types.
Let $j : \Hom(\Sigma(\Delta, x : A), \Sigma(\Delta, z : \| A \|))$ be the obvious map defined in terms of $\eta$.
Since $\Fib$ is closed under $\Sigma$-types, \rprop{fib-sigma} implies that it is also closed under compositions.
Then \rlem{uni-conn} implies that $j$ is connected.
A term $\Delta, x : A \vdash d : D[\eta_A(x)/z]$ determines the top map in the following diagram:
\[ \xymatrix{ \Sigma(\Delta, x : A) \ar[r] \ar[d]_j                  & \Sigma(\Delta, z : \| A \|, D) \ar@{->>}[d] \\
              \Sigma(\Delta, z : \| A \|) \ar@{-->}[ur] \ar@{=}[r]  & \Sigma(\Delta, z : \| A \|)
            } \]
The right map is the obvious projection, which is a fibration since $D$ is fibrant over $\Delta, z : \| A \|$.
Since $j$ is connected, we have a lift in this square.
This lift determines a term $\Delta, z : \| A \| \vdash \| A \|\text{-}\fs{elim}(z.D, x.d) : D$.
The commutativity of the upper triangle implies the existence of a term $\Delta, x : A \vdash \| A \|\text{-}\fs{elim_h}(z.D, x.d) : \Id(\| A \|\text{-}\fs{elim}(z.D, x.d)[\eta_A(x)/z],d)$.

Now, suppose that $\Fib$ extends to a higher modality.
Let $B$ be a fibrant type over $\Delta$.
We need to prove that the type of lifts in the following diagram is contractible:
\[ \xymatrix{ A \ar[r]^b \ar[d]_{\eta_A} & B \\
              \| A \| \ar@{-->}[ur]
            } \]
We can define a lift in this diagram as $\lambda z.\,\| A \|\text{-}\fs{elim}(z.B, x.\,b\,x)$.
The commutativity of the triangle is witnessed by the term $\lambda x.\,\| A \|\text{-}\fs{elim_h}(z.B, x.\,b\,x)$.

Let $\Delta, z : \| A \| \vdash f_i : B$, $\Delta, x : A \vdash h_i : \Id(f_i[|x|/z],\,b\,x)$ be lifts in this diagram for $i \in \{1,2\}$.
We need to construct a homotopy between them.
Since $\Fib$ is closed under identity types, the type $\Id(f_1,f_2)$ is fibrant over $\Delta, z : \| A \|$.
Thus, we have a homotopy between $f_1$ and $f_2$:
\[ \Delta, z : \| A \| \vdash \| A \|\text{-}\fs{elim}(z.\,\Id(f_1,f_2), x.\,h_1 \ct \sym{h_2}) : \Id(f_1, f_2). \]
Let us denote this homotopy by $H$.
We need to prove that $\sym{H[|x|/z]} \ct h_1$ is homotopic to $h_2$.
This homotopy can be constructed from the following one:
\[ \Delta, x : A \vdash \| A \|\text{-}\fs{elim_h}(z.\,\Id(f_1,f_2), x.\,h_1 \ct \sym{h_2}) : \Id(H[|x|/z], h_1 \ct \sym{h_2}). \]
\end{proof}

\begin{cor}
A class of fibrations closed under substitutions and $\Sigma$-types is locally reflective if and only if it extends to a higher modality.
\end{cor}

If we assume the stability condition for $\eta$, then we can replace it with the following stable rule:
\begin{center}
\AxiomC{$\Gamma \mid \Delta \vdash a : A$}
\UnaryInfC{$\Gamma \mid \Delta \vdash | a | : \| A \|$}
\DisplayProof
\end{center}
\medskip
Indeed, $\eta_A(x)$ can be defined as $| x |$.
Conversely, $| a |$ can be defined in terms of $\eta_A$ as $\eta_A(x)[a/x]$.
Since $|-|$ is stable, these constructions are mutually inverse.

We will say that a higher modality is \emph{stable} if $\| - \|$ and $| - |$ are stable.
The underlying class of fibrations of a stable higher modality is stably locally reflective.
The converse does not hold since the stability under substitutions is a stricter condition, then the stability under pullbacks.

If a higher modality is stable, then we can define the local version of its eliminator:
\begin{center}
\AxiomC{$\Gamma \mid \Delta, z : \| A \|, E \vdash D \fib$}
\AxiomC{$\Gamma \mid \Delta, x : A, E[|x|/z] \vdash d : D[|x|/z]$}
\BinaryInfC{$\Gamma \mid \Delta, z : \| A \|, E \vdash \| A \|\text{-}\fs{elim}(z E .D, x E.d) : D$}
\DisplayProof
\end{center}
\medskip

\begin{center}
\AxiomC{$\Gamma \mid \Delta, z : \| A \|, E \vdash D \fib$}
\AxiomC{$\Gamma \mid \Delta, x : A, E[|x|/z] \vdash d : D[| x | / z]$}
\BinaryInfC{$\Gamma \mid \Delta, x : A, E[|x|/z] \vdash \| A \|\text{-}\fs{elim_h}(z E.D, x E.d) : \Id(d'[|x|/z], d)$}
\DisplayProof
\end{center}
where $d' = \| A \|\text{-}\fs{elim}(z E.D, x E.d)$.

\begin{prop}[fib-dep-st]
Let $\Fib$ be a class of fibrations closed under $\Sigma$-types and let $\| - \|$, $| - |$ be a pair of stable constructions as described above.
Then the following conditions are equivalent:
\begin{enumerate}
\item \label{it:fib-non-dep} The pair $\| - \|$, $| - |$ makes $\Fib$ into a (stably) locally reflective class.
\item \label{it:fib-dep} The pair $\| - \|$, $| - |$ makes $\Fib$ into a (stable) higher modality.
\item \label{it:fib-dep-local} The class $\Fib$ is closed under identity types and the unstable local eliminator is definable.
\end{enumerate}
\end{prop}
\begin{proof}
The equivalence between \eqref{it:fib-non-dep} and \eqref{it:fib-dep} holds by \rprop{fib-dep-unst}.
Obviously, \eqref{it:fib-dep-local} implies \eqref{it:fib-dep} since the (global) eliminator is a special case of the local one.
Let us prove that \eqref{it:fib-non-dep} implies the existence of the unstable local eliminator.
Consider the following pullback square:
\[ \xymatrix{ \Sigma(\Delta, x : A, E[|x|/z]) \ar[r] \ar[d]_i \pb   & \Sigma(\Delta, x : A) \ar[d]^j \\
              \Sigma(\Delta, z : \| A \|, E) \ar[r]                 & \Sigma(\Delta, z : \| A \|),
            } \]
where the bottom and top maps are the obvious projection and the left and right maps are defined in terms of $| - |$.
Since $\Fib$ is closed under $\Sigma$-types, \rprop{fib-sigma} implies that it is also closed under compositions.
Then \rlem{uni-conn} implies that $j$ is connected.
Since $\Fib$ is stably locally reflective, \rlem{conn-pullback} implies that $i$ is also connected.

Any term $\Delta, x : A, E[|x|/z] \vdash d : D[|x|/z]$ determines the top map in the following diagram:
\[ \xymatrix{ \Sigma(\Delta, x : A, E[|x|/z]) \ar[r] \ar[d]_i           & \Sigma(\Delta, z : \| A \|, E, D) \ar@{->>}[d] \\
              \Sigma(\Delta, z : \| A \|, E) \ar@{-->}[ur] \ar@{=}[r]   & \Sigma(\Delta, z : \| A \|, E)
            } \]
The right map is the obvious projection, which is a fibration since $D$ is fibrant over $\Delta, z : \| A \|, E$.
Since $i$ is connected, we have a lift in this square.
This lift determines a term $\Delta, z : \| A \|, E \vdash \| A \|\text{-}\fs{elim}(z E. D, x E. d) : D$.
The commutativity of the upper triangle implies the existence of a term $\Delta, x : A, E[|x|/z] \vdash \| A \|\text{-}\fs{elim_h}(z E. D, x E .d) : \Id(\| A \|\text{-}\fs{elim}(z E. D, x E. d)[|x|/z],d)$.
\end{proof}

\begin{example}
Since the class of $n$-truncated maps is closed under identity types, it is locally reflective if and only if we have the weak version of the $n$-truncation operation.
\end{example}

\bibliographystyle{amsplain}
\bibliography{ref}

\end{document}